\documentclass[11pt]{article}

\usepackage[margin=2.6cm]{geometry}
\usepackage{amsmath,amssymb,amsthm,mathtools}
\usepackage[round,authoryear]{natbib}
\usepackage[colorlinks,citecolor=blue,linkcolor=blue,urlcolor=blue]{hyperref}
\usepackage{graphicx,flafter,placeins,booktabs,array,enumitem,microtype}
\usepackage[utf8]{inputenc}
\usepackage[T1]{fontenc}
\usepackage[osf,sc]{mathpazo}

\usepackage{xcolor}
\definecolor{RevisionGreen}{RGB}{0,105,45}
\definecolor{RevisionRed}{RGB}{170,30,30}

\newtheorem{theorem}{Theorem}[section]
\newtheorem{proposition}[theorem]{Proposition}
\newtheorem{lemma}[theorem]{Lemma}

\theoremstyle{definition}
\newtheorem{definition}[theorem]{Definition}
\newtheorem{remark}[theorem]{Remark}

\newcommand{\E}{\mathbb E}
\newcommand{\R}{\mathbb R}
\newcommand{\one}{\mathbf 1}
\newcommand{\Sdk}{\mathcal S_{d,k}}

\newcommand{\norm}[1]{\left\lVert #1\right\rVert}
\newcommand{\inner}[2]{\left\langle #1,#2\right\rangle}
\newcommand{\Var}{\operatorname{Var}}

\renewcommand{\P}{\mathbb P}

\title{Nonparametric Regression on Fixed-Cardinality Subsets:\\
Minimax Risk and Random-Design Effects}
\author{%
  G\'erard Biau\\[4pt]
  \small Sorbonne Universit\'e, Institut universitaire de France\\[2pt]
  \small \texttt{gerard.biau@sorbonne-universite.fr}
}
\date{}

\begin{document}
\maketitle

\begin{abstract}
\noindent
We study regression from noisy evaluations of uniformly sampled subsets containing exactly \(k\) items from a ground set of size \(d\), a setting motivated by combination screening. The fixed-cardinality constraint couples membership coordinates and prevents a direct use of product-domain notions of interaction and smoothness. Using the harmonic decomposition of the Johnson graph, we define an intrinsic Johnson--Sobolev scale and determine the exact dimensions of low-order interaction spaces. We establish finite-sample risk bounds for harmonic projection and a nonasymptotic minimax characterization when the Sobolev energy budget is bounded relative to the noise variance proxy. We also study sparsity when responses depend on few items, quantifying the reduced interaction dimension and the effects of unknown support and rare configurations. When the empirical Gram matrix is well conditioned, stable least squares removes the design fluctuation of the component already in the fitted space. A rank-deficiency analysis identifies the components left unresolved by random sampling. Combined with a centered completion estimator, it yields minimax bounds matching up to constants whenever \(\min(k,d-k)\) and the smoothness order are fixed, without restricting the signal-to-noise ratio. Numerical experiments distinguish observation noise, design fluctuation, and incomplete coverage; NCI-ALMANAC drug-pair scores provide a real-data illustration.

\smallskip
\noindent\textbf{Keywords.}
Combinatorial regression; fixed-cardinality subsets; active-item sparsity; harmonic
analysis; Johnson graph; minimax estimation; random design; Sobolev
ellipsoid.
\end{abstract}

\section{Introduction}\label{sec:introduction}

Suppose that a pharmacologist wants to estimate the expected efficacy of every three-compound combination drawn from a library of twenty candidates. The unknown response function maps each combination to its expected efficacy. Since there are \(\binom{20}{3}=1{,}140\) possible combinations, an experiment typically provides noisy evaluations for only a fraction of them. The statistical goal is to learn the response function over the entire collection of combinations, including those that were not observed. This is a regression problem in which each input is a three-element
subset of the fixed compound library rather than a vector of
unconstrained covariates.  The NCI-ALMANAC study provides a concrete example with drug pairs evaluated on the NCI-60 panel of cancer cell lines \citep{HolbeckEtAl2017}. After the preprocessing described in Section~\ref{supp:almanac} of the Supplementary Material, the CellMiner release \citep{CellMinerAlmanac} yields scores for 5,232 pairs among 104 drugs.

Bundle preference studies and configuration testing can pose the same
regression problem when inputs have a prescribed size and interactions
among selected items matter. Motivated by these examples, we consider regression on subsets of
exactly \(k\) items drawn from a fixed ground set of \(d\) items.
In our model, the observed subsets are sampled uniformly, providing
a symmetric benchmark. Nonuniform or adaptive sampling and prediction
across cardinalities are left for separate work.

Let \(d\geqslant2\) and \(1\leqslant k\leqslant d-1\), and define
\[
  \mathcal S_{d,k}=\{S\subseteq[d]:|S|=k\},
  \qquad
  [d]=\{1,\ldots,d\},
  \qquad
  N=|\mathcal S_{d,k}|=\binom dk.
\]
We observe independent pairs \((S_j,Y_j)\) satisfying
\[
  S_j\sim\nu_{d,k},
  \qquad
  Y_j=f^\star(S_j)+\varepsilon_j,
  \qquad j=1,\ldots,n,
\]
where \(\nu_{d,k}\) is the uniform distribution on
\(\mathcal S_{d,k}\). The noise variables are i.i.d., centered, independent of the design, and sub-Gaussian with
variance proxy \(\sigma^2>0\).  We evaluate an estimator \(\widehat f\) using the global squared
\(L^2(\nu_{d,k})\)-loss
\[
  \|\widehat f-f^\star\|_2^2
  =
  \frac1N\sum_{S\in\mathcal S_{d,k}}
  \{\widehat f(S)-f^\star(S)\}^2,
\]
and define its risk as the expectation of this loss with respect to
both the random design and the noise. Because sampling is performed with replacement from \(\mathcal S_{d,k}\), many subsets may remain unobserved, even when \(n\) is comparable to \(N\). Thus, without a structural assumption linking the response values
across subsets, the function cannot be recovered on unobserved
inputs.

\paragraph{Johnson geometry and interaction structure.}
To make such a structural assumption useful, we first need to specify
when two subsets should be regarded as close.  Because every \(S\in\mathcal S_{d,k}\) contains exactly \(k\) items,
adding or removing one item takes us outside \(\mathcal S_{d,k}\).
Thus, the smallest cardinality-preserving change exchanges one
selected item for one unselected item. The Johnson graph formalizes this operation. Its vertices are the \(k\)-subsets of \([d]\), and two vertices are adjacent when they differ by one such exchange \citep{BrouwerCohenNeumaier1989}. The graph distance is
\[
\operatorname{dist}_J(S,S')=k-|S\cap S'|,
\qquad S,S'\in\mathcal S_{d,k},
\]
the number of exchanges needed to transform \(S\) into \(S'\). In the present paper, variation under these exchanges provides the
basis for our smoothness scale. The eigenspaces of the associated random walk decompose
functions on \(\mathcal S_{d,k}\) into Johnson harmonic levels of
increasing interaction order \citep{Delsarte1973,BannaiIto1984,Filmus2016}. Level zero contains constants; level one, centered additive effects; level two, when present, pairwise effects beyond constants and additive terms; and higher levels, genuinely higher-order interactions.

As developed formally in Section~\ref{sec:harmonic-analysis}, the
harmonic levels are indexed by \(r=0,\ldots,m\), where
\(m=\min(k,d-k)\). This upper limit reflects the symmetry between
selected and excluded items. The normalized Laplacian of the
one-exchange walk acts on level \(r\) by multiplication by
\[
  \gamma_{r,d}=\frac{r(d-r+1)}d.
\]
The constant level has eigenvalue zero, while the positive eigenvalues
increase with \(r\). We therefore separate the intercept and analyze
functions centered under the uniform design,
\(\E_{S\sim\nu_{d,k}}[f(S)]=0\); the intercept can be estimated
separately at the parametric rate. The weights
\(\gamma_{r,d}^{\,s}\) then define the Johnson--Sobolev class of
smoothness order \(s>0\) and budget \(B>0\) introduced in
Section~\ref{sec:statistical-theory}.

For estimation, we retain the nonconstant levels up to a cutoff
\(D\in\{0,\ldots,m\}\). We show that their combined dimension is
\[
  p_D=\binom dD-1,
\]
with \(p_0=0\). An intercept plus membership indicators and their
products spans the same functions, but its coefficients are nonunique
on \(\mathcal S_{d,k}\). A central role of the Johnson harmonic decomposition is precisely to identify
interaction order and, through the associated eigenvalues, smoothness
under exchanges independently of this nonunique coefficient representation.

\paragraph{Minimax risk and random-design uncertainty.}
Truncating the harmonic decomposition at level \(D\) creates a
familiar statistical tradeoff.  Increasing \(D\) reduces the
approximation error by retaining more interactions, but it also
increases the dimension \(p_D\) of the fitted model.  For
\(0\leqslant D<m\), the squared error contributed by the discarded
levels is bounded by
\[
  b_D(s,B)
  =
  \frac{B}{\gamma_{D+1,d}^{\,s}},
\]
because \(D+1\) is the first omitted level.  At the full cutoff
\(D=m\), no level is omitted, and we set \(b_m(s,B)=0\).  Estimating the \(p_D\) retained coefficients from \(n\) observations
has the usual noise-dependent cost \(\sigma^2p_D/n\), where
\(\sigma^2\) is the sub-Gaussian variance proxy.  Our first main result
shows that this approximation--estimation tradeoff determines the
minimax risk over the Johnson--Sobolev class when the ratio \(B/\sigma^2\) is uniformly bounded. 
More precisely, for every
fixed \(C_{\mathrm{snr}}>0\), uniformly over all admissible
\(d,k,n,s,B,\sigma^2\) satisfying
\begin{equation}
\label{eq:BRC}
  \frac{B}{\sigma^2}\leqslant C_{\mathrm{snr}},
\end{equation}
the minimax risk is comparable, up to constants depending only on
\(C_{\mathrm{snr}}\), to
\[
  \min_{0\leqslant D\leqslant m}
  \left\{
    \frac{\sigma^2p_D}{n}
    +
    b_D(s,B)
  \right\}.
\]
Thus, the optimal cutoff balances the noise incurred by estimating
the retained interaction levels against the approximation error
created by omitting the remaining ones. A max--min formulation, equivalent up to a factor of two and used to establish the minimax lower bound, is given in Section~\ref{sec:statistical-theory}.

Empirical projection averages responses over sampled subsets, so random sampling contributes a term of order \(Bp_D/n\) to its risk bound even without observation noise. The bounded-ratio condition \eqref{eq:BRC} allows us to absorb this term into \(\sigma^2p_D/n\). To examine what happens without that condition, we next consider least squares on \(W_D\). When the empirical Gram matrix is well conditioned, the component of the response already in \(W_D\) cancels from the estimation error. The constant leverage \(p_D\) gives this conditioning with high probability when \(n\) is of order \(p_D\log(2p_D)\). This gain from least squares, however, depends on the sampled design. At cutoffs for which its Gram matrix loses rank, some directions in \(W_D\) remain unidentified even without observation noise. For \(1\leqslant D\leqslant m\), we quantify this loss by the normalized expected rank deficiency
\[
\rho_{D,n}
=\frac{1}{p_D}\,
\E\!\left[p_D-\operatorname{rank}(\mathbf X_D)\right],
\]
where \(\mathbf X_D\) is the sampled evaluation matrix on \(W_D\). At the full interaction level, this deficiency is determined by the number of subsets never observed. This leads us to a centered completion estimator. Combining its risk bound with that of harmonic projection yields minimax bounds for arbitrary \(B/\sigma^2\), matching up to the factor \(\gamma_{m,d}^{\,s}\leqslant m^s\), which is uniformly bounded when \(m=\min(k,d-k)\) and \(s\) are fixed. 

\paragraph{Contributions and scope.}

Our results use Johnson harmonic analysis to link local-exchange smoothness, interaction dimensions, and risk under noisy and incomplete sampling. We also study sparsity in the form of juntas, functions depending on
only a few items, and quantify their reduced interaction dimension
and the effects of unknown support and rare active patterns. As a constructive complement,
Section~\ref{ssec:canonical-frame} of the Supplementary Material derives a unit-norm tight frame at each harmonic level. This representation is not needed
for the statistical guarantees, which depend only on the harmonic
subspaces.

To the best of our knowledge, this is the first nonasymptotic minimax
analysis of regression on uniformly sampled fixed-cardinality subsets
under Johnson--Sobolev smoothness, with risk bounds that explicitly
account for incomplete coverage.

Section~\ref{sec:related-work} reviews related work. Sections~\ref{sec:harmonic-analysis}--\ref{sec:statistical-theory} develop the Johnson geometry, smoothness class, and minimax bounds. Section~\ref{sec:active-items} treats active-item sparsity, and Section~\ref{sec:random-design} studies design stability and incomplete coverage. Section~\ref{sec:discussion} concludes. The Supplementary Material contains numerical studies, the NCI-ALMANAC illustration, the canonical frame, and proofs.

\section{Related work}\label{sec:related-work}

\paragraph{Learning functions on subsets.}
The closest learning-theoretic literature studies functions defined
on the full power set \(2^{[d]}\), whose standard harmonic
representation is the Fourier--Walsh expansion
\citep{ODonnell2014}.  Under sparsity assumptions on this
representation, \citet{StobbeKrause2012} establish recovery guarantees from randomly sampled function values.  Other
work analyzes the learnability of valuation classes from values or
pairwise comparisons, or imposes submodularity through explicit or
neural parameterizations
\citep{BalcanConstantinIwataWang2012,BalcanVitercikWhite2016,
DolhanskyBilmes2016,DeChakrabarti2022}.  These approaches differ from
the present setting in both assumptions and objective.  We observe only one fixed-cardinality layer and study global \(L^2\)-risk over a
smoothness class. Our dense minimax benchmark imposes neither sparsity
nor a shape constraint such as monotonicity or submodularity. The active-item extension in Section~\ref{sec:active-items} assumes that the response depends on only a small set of items. Sparsity therefore concerns which items matter, rather than the number of nonzero coefficients in a chosen expansion of the response. This differs from recovering a sparse coefficient expansion or
approximating function values up to a multiplicative factor on the full
power set.

Permutation-invariant neural architectures provide general
representations for functions of sets
\citep{ZaheerEtAl2017,WagstaffEtAl2019}.  They can be applied here
after encoding the identities of the items, but their principal
results concern representation and approximation rather than
smoothness-based statistical guarantees under noisy random sampling.
Likewise, Gaussian-process models for choice functions learn the
possibly set-valued choice made from a presented collection 
\citep{BenavoliAzzimontiPiga2023}; this differs from estimating a
scalar response attached to every subset.

\paragraph{Regression and signal recovery on graphs.}
Laplacian regularity is widely used to estimate functions observed on
graph vertices.  \citet{KirichenkoVanZanten2017} develop Bayesian
Laplacian regularization on growing graphs, and
\citet{KirichenkoVanZanten2018} establish corresponding minimax lower
bounds.  For neighborhood graphs approximating an underlying
Euclidean domain, Laplacian-based estimators can recover continuum
Sobolev rates
\citep{GreenBalakrishnanTibshirani2021,
ShiBalasubramanianPolonik2024}.  The Johnson graph plays a different role here: its vertex set is exactly the finite covariate domain, and both its spectral weights and their multiplicities vary explicitly with \((d,k)\).

Graph signal processing also considers the recovery of bandlimited
signals from randomly sampled vertices.
\citet{PuyTremblayGribonvalVandergheynst2018} derive stable-sampling
conditions governed by the dimension of the bandlimited space and an
appropriate graph-coherence parameter; under sampling adapted to the local graph coherence, the resulting
scale is of order \(p\log p\) for a
\(p\)-dimensional space.  Our target functions need not be
bandlimited.  Johnson--Sobolev regularity instead controls the \(L^2\)-mass
outside each low-order approximation space, and our risk analysis
accounts jointly for approximation error, observation noise, and
directions left unresolved by the random design.

\paragraph{Random least squares.}
The stability of least-squares approximation from random function
evaluations is governed by the supremum of the diagonal of the
projection kernel, often called the Christoffel or leverage function
\citep{CohenDavenportLeviatan2013}; see also the correction
\citep{CohenDavenportLeviatan2019}.  Matrix Chernoff inequalities
provide the corresponding concentration tool \citep{Tropp2012}.  We
specialize this general mechanism to the Johnson interaction spaces.
Permutation invariance of these spaces, together with transitivity of
the action on the Johnson domain, forces the projection-kernel
diagonal to be constant and equal to the dimension, so uniform sampling is already
leverage-optimal.  We then incorporate this stability property into a
smoothness-based risk analysis and separate observation noise from
variation caused by the random design.

\paragraph{Noisy combinatorial optimization.}
Noisy evaluations of subsets also arise in subset selection,
Bayesian optimization, and combinatorial bandits
\citep{QianEtAl2017,OhTomczakGavvesWelling2019,LiangWanDong2024,TajdiniJainJamieson2024}.  These methods use
sequentially chosen evaluations to locate or approximate an optimal
subset.  Their performance is consequently measured through
optimization accuracy, regret, or query complexity.  Our design is
instead i.i.d.\ and nonadaptive, and the objective is to reconstruct
the response function over the entire fixed-cardinality domain.
Optimization guarantees therefore do not directly control the global
prediction risk considered here.

\paragraph{Harmonic analysis of subset data.}
The Johnson association scheme, its eigenspaces, and its zonal
functions are classical
\citep{Delsarte1973,BannaiIto1984}.  Statistical spectral analysis
based on representations of the symmetric group was developed by
\citet{Diaconis1988,Diaconis1989}, including the analysis of data
indexed by combinatorial objects, while
\citet{DiaconisRockmore1993} study the computation of the associated
isotypic projections.  More recently, \citet{Filmus2016} constructed
an explicit orthogonal basis for functions on a fixed-cardinality
slice, and \citet{IglesiasNatale2022} developed a fast full-table
Fourier transform on the Johnson graph.

These results provide the harmonic foundation used throughout the
paper.  Section~\ref{ssec:canonical-frame} of the Supplementary Material gives an explicit unit-norm
normalization of the canonical zonal frame and the resulting
permutation-equivariant representation of the harmonic projections.

\section{Model and Johnson harmonic geometry}
\label{sec:harmonic-analysis}

\subsection{Statistical experiment and risk}

Throughout the paper, \(d\geqslant2\),
\(1\leqslant k\leqslant d-1\), and \(n\geqslant 1\).  Set
\[
  [d]=\{1,\ldots,d\},
  \qquad
  \Sdk=\{S\subseteq[d]:|S|=k\},
  \qquad
  N=|\Sdk|=\binom dk.
\]
The excluded cases \(k=0\) and \(k=d\) are degenerate: the domain
then contains a single subset, and no cardinality-preserving exchange
is possible.  We equip \(\Sdk\) with the uniform probability measure
\(\nu_{d,k}\).  For real-valued functions \(f,g\) on \(\Sdk\), write
\[
  \inner{f}{g}
  =
  \E_{S\sim\nu_{d,k}}[f(S)g(S)],
  \qquad
  \norm{f}_2^2=\inner{f}{f}.
\]
Thus,
\[
  \norm{f-g}_2^2
  =
  \frac1N
  \sum_{S\in\Sdk}\{f(S)-g(S)\}^2.
\]

The design variables \(S_1,\ldots,S_n\) are i.i.d.\ with law
\(\nu_{d,k}\), and the responses satisfy
\[
  Y_j=f^\star(S_j)+\varepsilon_j,
  \qquad j=1,\ldots,n.
\]
The noise variables are i.i.d.\ with a common law \(P\) and are
independent of the design.  For \(\sigma^2>0\), let
\(\mathcal P_{\sigma^2}\) denote the class of probability laws \(P\)
on \(\R\) such that, for \(\varepsilon\sim P\),
\(
  \E_P[\varepsilon]=0
\)
and
\[
  \E_P[\exp(t\varepsilon)]
  \leqslant
  \exp\left(\frac{\sigma^2t^2}{2}\right),
  \qquad t\in\R.
\]
In particular, every \(P\in\mathcal P_{\sigma^2}\) satisfies
\(\E_P[\varepsilon^2]\leqslant\sigma^2\).

For a function \(f:\Sdk\to\R\), regarded as a possible value of the
unknown response \(f^\star\), and a noise law
\(P\in\mathcal P_{\sigma^2}\), denote by \(\P_{f,P}\) the joint law
of \((S_j,Y_j)_{j=1}^n\) and by \(\E_{f,P}\) the corresponding
expectation.  Expectations involving only a uniform subset of
\(\Sdk\) are written explicitly as
\(\E_{S\sim\nu_{d,k}}\).  For a parameter class \(\mathcal A\), define
the minimax risk
\[
  R_n^\star(\mathcal A)
  =
  \inf_{\widehat f}
  \sup_{f\in\mathcal A}
  \sup_{P\in\mathcal P_{\sigma^2}}
  \E_{f,P}\!\left[
    \norm{\widehat f-f}_2^2
  \right],
\]
where the infimum ranges over all measurable estimators, based on
\((S_j,Y_j)_{j=1}^n\), taking values in the space of real-valued
functions on \(\Sdk\).  Upper bounds will hold uniformly over
\(\mathcal P_{\sigma^2}\), whereas lower bounds may be established by restricting to the Gaussian submodel \(N(0,\sigma^2)\) or the noiseless law \(\delta_0\). All results below are nonasymptotic and uniform in \(d\), \(k\), and
\(n\).  They therefore also apply to sequences of experiments in which
\(d=d_n\) and \(k=k_n\) vary with the sample size.

\subsection{Inclusion filtration and harmonic levels}

For \(T\subseteq[d]\), define the inclusion indicator
\[
  h_T(S)=\one_{\{T\subseteq S\}},
  \qquad S\in\Sdk.
\]
Thus, \(h_T(S)\) records whether all items in \(T\) are present in
the input subset \(S\).  Set \(m=\min(k,d-k)\), and, for
\(0\leqslant r\leqslant m\), define
\[
  U_r
  =
  \operatorname{span}\{h_T:T\subseteq[d],\ |T|=r\}.
\]
The space \(U_r\) therefore contains the functions representable by
inclusion effects of order \(r\), together with all lower-order
effects.  To see the latter point, let \(1\leqslant r\leqslant m\)
and \(|T|=r-1\).  For every \(S\in\Sdk\),
\[
  h_T(S)
  =
  \frac1{k-r+1}
  \sum_{i\in[d]\setminus T}h_{T\cup\{i\}}(S).
\]
Indeed, if \(T\not\subseteq S\), both sides vanish.  If
\(T\subseteq S\), exactly \(k-r+1\) elements
\(i\in S\setminus T\) make \(T\cup\{i\}\) a subset of \(S\).
Consequently,
\[
  U_0\subseteq U_1\subseteq\cdots\subseteq U_m.
\]

The dimensions of these spaces follow from the rank of the inclusion
matrix.  Let \(M_{r,k}\) be the matrix whose rows are indexed by the
\(r\)-subsets \(T\), whose columns are indexed by the \(k\)-subsets
\(S\), and whose entries are
\[
  (M_{r,k})_{T,S}=h_T(S)=\one_{\{T\subseteq S\}}.
\]
Each row is the vector of values of \(h_T\) over the whole domain
\(\Sdk\).  Hence, the row space of \(M_{r,k}\) is naturally
identified with \(U_r\), and
\(
  \dim(U_r)=\operatorname{rank}(M_{r,k}).
\)
Gottlieb's rank theorem states that \(M_{r,k}\) has full row rank
whenever \(r\leqslant k\) and \(r+k\leqslant d\)
\citep{Gottlieb1966}.  Both conditions hold for
\(0\leqslant r\leqslant m\), and therefore
\[
  \dim(U_r)=\binom dr,
  \qquad 0\leqslant r\leqslant m.
\]
At the final level,
\[
  \dim(U_m)=\binom dm=\binom dk=N,
\]
where the second equality follows from
\(m\in\{k,d-k\}\).  Since a function on \(\Sdk\) is determined by
its \(N\) values, the full space \(L^2(\Sdk,\nu_{d,k})\) also has
dimension \(N\).  Hence,
\(
  U_m=L^2(\Sdk,\nu_{d,k}).
\)

To isolate the new component appearing at each order, set
\(U_{-1}=\{0\}\) and define
\[
  V_r=U_r\cap U_{r-1}^{\perp},
  \qquad 0\leqslant r\leqslant m.
\]
Because the inclusion spaces are nested,
\[
  U_r=U_{r-1}\mathbin{\oplus^\perp}V_r.
\]
Iterating this identity and using \(U_m=L^2(\Sdk,\nu_{d,k})\) gives
\begin{equation}\label{eq:harmonic-decomposition}
  L^2(\Sdk,\nu_{d,k})
  =
  V_0\mathbin{\oplus^\perp}\cdots
  \mathbin{\oplus^\perp}V_m,
  \qquad
  \dim(V_r)
  =
  \binom dr-\binom d{r-1},
\end{equation}
with the convention \(\binom d{-1}=0\).  Let \(P_r\) denote the orthogonal projector onto \(V_r\).  Since
\(V_r=U_r\cap U_{r-1}^{\perp}\), the function \(P_rf\) is the part of
\(f\) that lies in \(U_r\) and is orthogonal to all lower-order
inclusion effects.  We call \(P_rf\) the order-\(r\) harmonic
component of \(f\).  As an orthogonal projection, it is uniquely
determined by \(f\) and does not depend on any choice of coefficients
in a representation using the indicators \(h_T\).

To see what the harmonic decomposition contributes, fix a maximal
interaction order \(D\in\{1,\ldots,m\}\) and consider an ordinary
indicator regression of the form
\[
  g(S)=\beta_0+
  \sum_{\substack{T\subseteq[d]\\1\leqslant |T|\leqslant D}}
  \beta_T h_T(S).
\]
These functions form \(U_D\). Their coefficients, however, are not
uniquely determined on \(\Sdk\), since
\(\sum_{i=1}^d h_{\{i\}}=k\one\), and a penalty
on them depends on the chosen parameterization. Since both representations span \(U_D\), unpenalized least squares
gives the same fitted values on the sampled subsets. If the sampled
evaluation map is injective on \(U_D\), these fitted values determine
a unique predictor on \(\Sdk\), although its indicator coefficients
may remain nonunique. A key advantage of the harmonic decomposition is that it separates
\(U_D\) into uniquely defined orthogonal interaction levels. We now relate these levels to the one-exchange random walk on the
Johnson graph, whose normalized Laplacian provides the eigenvalues
used to define smoothness under exchanges. These structures
underpin our approximation and random-design risk bounds.

For \(S\in\Sdk\), \(a\in S\), and
\(b\notin S\), write
\[
  S^{a\to b}
  =
  (S\setminus\{a\})\cup\{b\}.
\]
Let \(S_0\) denote the current subset.  Conditional on \(S_0=S\), choose \(A\) uniformly from the \(k\)
elements of \(S\) and, independently, choose \(B\) uniformly from the
\(d-k\) elements of \([d]\setminus S\).  Set \(S_1=S^{A\to B}\). The Markov operator of this one-exchange walk is
therefore
\[
  Kf(S)
  :=
  \E\!\left[f(S_1)\mid S_0=S\right]
  =
  \frac1{k(d-k)}
  \sum_{a\in S}\sum_{b\notin S}f(S^{a\to b}).
\]
Equivalently, \(K\) is the normalized adjacency operator of the
Johnson graph \citep{BrouwerCohenNeumaier1989}. Indeed, for \(S,S'\in\Sdk\), the transition probabilities are
\[
    q(S,S')
  =
  \mathbb P(S_1=S'\mid S_0=S)
  =
  \frac{1}{k(d-k)}
  \one_{\{|S\cap S'|=k-1\}}.
\]
Thus \(q\) is the adjacency matrix divided by the degree
\(k(d-k)\), and the Markov operator is
\[
  Kf(S)=\sum_{S'\in\Sdk}q(S,S')f(S').
\]
The uniform measure \(\nu_{d,k}\) is reversible for this walk.  Indeed,
every exchange \(S\to S'\) has the equally likely reverse exchange
\(S'\to S\), so \(q(S,S')=q(S',S)\).  Since \(\nu_{d,k}\) is
uniform, it follows that
\[
  \nu_{d,k}(S)q(S,S')
  =
  \nu_{d,k}(S')q(S',S),
  \qquad S,S'\in\Sdk.
\]
This detailed-balance identity is precisely the reversibility
condition.  It implies that \(K\) is self-adjoint with respect to
\(\inner{\cdot}{\cdot}\):
\[
  \inner{f}{Kg}
  =
  \inner{Kf}{g},
  \qquad
  f,g\in L^2(\Sdk,\nu_{d,k}).
\]

The classical spectrum of the Johnson graph shows that the harmonic
spaces in \eqref{eq:harmonic-decomposition} diagonalize this operator
\citep{Delsarte1973,BannaiIto1984,BrouwerCohenNeumaier1989}.  The
precise correspondence, together with the associated Dirichlet form identity, is recorded in the following proposition.

\begin{proposition}[Johnson spectrum]\label{prop:johnson-spectrum}
For \(0\leqslant r\leqslant m\) and \(f\in V_r\),
\begin{equation}\label{eq:johnson-eigenvalue}
  Kf=\theta_{r,d,k}f,
  \qquad
  \theta_{r,d,k}
  =
  1-\frac{r(d-r+1)}{k(d-k)}.
\end{equation}
Moreover, every \(f\in L^2(\Sdk,\nu_{d,k})\) satisfies
\begin{equation}\label{eq:johnson-dirichlet-form}
  \inner{f}{(I-K)f}
  =
  \frac1{2k(d-k)}
  \E_{S\sim\nu_{d,k}}
  \left[
    \sum_{a\in S}\sum_{b\notin S}
    \{f(S)-f(S^{a\to b})\}^2
  \right].
\end{equation}
\end{proposition}

Since \(r\mapsto r(d-r+1)\) is strictly increasing on
\(\{0,\ldots,m\}\), the eigenvalues
\(\theta_{r,d,k}\) of \(K\) are strictly decreasing with \(r\),
whereas the eigenvalues
\[
  1-\theta_{r,d,k}
  =
  \frac{r(d-r+1)}{k(d-k)}
\]
of \(I-K\) are strictly increasing. Together with the complete decomposition \eqref{eq:harmonic-decomposition} and the fact that the eigenvalues \(\theta_{r,d,k}\) are pairwise distinct, the proposition shows that \(V_r\) is precisely the eigenspace of \(K\) associated with \(\theta_{r,d,k}\). Equation~\eqref{eq:johnson-dirichlet-form} gives this spectral ordering
a local interpretation: the quadratic form
\(\inner{f}{(I-K)f}\) is one half of the mean squared change in \(f\)
under a uniformly chosen exchange.  Among functions of unit \(L^2\)-norm, higher harmonic levels
correspond to larger eigenvalues of \(I-K\) and hence to greater mean
squared variation across neighboring subsets.

The harmonic projections also admit a canonical subset-indexed tight
frame that respects permutations of the ground set.  This basis-free,
permutation-equivariant representation may be of independent
interest.  Since it is not needed for the statistical risk analysis,
its explicit construction and proof are deferred to
Section~\ref{ssec:canonical-frame} of the Supplementary Material.

\section{Smoothness, projection, and minimax risk}
\label{sec:statistical-theory}

\subsection{The intrinsic Johnson--Sobolev scale}

Proposition~\ref{prop:johnson-spectrum} shows that the quadratic form
associated with \(I-K\) measures the variation of a function under a
single admissible exchange.  Hence \(I-K\) is the normalized random-walk Laplacian of the Johnson
graph.  We rescale it
by defining
\begin{equation}\label{eq:johnson-laplacian}
  L_J
  =
  \frac{k(d-k)}d(I-K).
\end{equation}
The normalization is chosen so that the smallest nonzero eigenvalue
of \(L_J\) is equal to one.  The scaling factor \(k(d-k)/d\) is unchanged when \(k\) is replaced
by \(d-k\), consistently with the symmetry between a subset and its
complement.

By Proposition~\ref{prop:johnson-spectrum}, \(L_J\) acts on \(V_r\)
as multiplication by
\[
  \gamma_{r,d}
  =
  \frac{r(d-r+1)}d,
  \qquad 0\leqslant r\leqslant m.
\]
In particular,
\(
  \gamma_{0,d}=0
\) and
\(
  \gamma_{1,d}=1.
\)
Moreover,
\[
  \gamma_{r+1,d}-\gamma_{r,d}
  =
  \frac{d-2r}{d}>0,
  \qquad 0\leqslant r<m,
\]
because \(m\leqslant d/2\).  Hence, the eigenvalues increase strictly
with the harmonic order.

Since \(L_J\) acts on \(V_r\) by multiplication by
\(\gamma_{r,d}\), the harmonic decomposition of \(f\) gives
\[
  L_Jf
  =
  \sum_{r=0}^m\gamma_{r,d}P_rf
  =
  \sum_{r=1}^m\gamma_{r,d}P_rf,
\]
where the second equality uses \(\gamma_{0,d}=0\).
More generally, the spectral functional calculus defines, for
\(s>0\),
\[
  L_J^sf
  =
  \sum_{r=1}^m
  \gamma_{r,d}^{\,s}P_rf.
\]
Equivalently, \(L_J^s\) vanishes on \(V_0\) and acts by multiplication
by \(\gamma_{r,d}^{\,s}\) on \(V_r\), for
\(1\leqslant r\leqslant m\).

\begin{definition}[Johnson--Sobolev energy]
\label{def:johnson-sobolev}
For \(s>0\) and
\(f\in L^2(\Sdk,\nu_{d,k})\), define
\begin{equation}\label{eq:johnson-sobolev}
  I_{J,d,k}^{(2,s)}(f)
  :=
  \inner{f}{L_J^sf}.
\end{equation}
Equivalently,
\[
  I_{J,d,k}^{(2,s)}(f)
  =
  \sum_{r=1}^m
  \gamma_{r,d}^{\,s}\norm{P_rf}_2^2.
\]
\end{definition}

The superscript \(2\) indicates that the energy is quadratic in the
\(L^2(\nu_{d,k})\) norm, while \(s\) is the spectral smoothness order.
Since \(\gamma_{1,d}=1\), the first nonconstant harmonic level has
unit weight for every \(s>0\).  At \(s=1\), the energy has the local
exchange representation
\[
  I_{J,d,k}^{(2,1)}(f)
  =
  \frac1{2d}
  \E_{S\sim\nu_{d,k}}
  \left[
    \sum_{a\in S}\sum_{b\notin S}
    \{f(S)-f(S^{a\to b})\}^2
  \right].
\]
This identity follows by combining \eqref{eq:johnson-laplacian} with the Dirichlet representation in
Proposition~\ref{prop:johnson-spectrum}.  Thus, at smoothness order \(s=1\),
the Johnson--Sobolev energy measures the aggregate squared variation of
\(f\) across admissible exchanges, averaged over the initial subset
and normalized by \(1/(2d)\).

The entire scale is invariant under complementation.  More precisely,
if
\[
  f^{\mathrm c}(A)
  =
  f([d]\setminus A),
  \qquad
  A\in\mathcal S_{d,d-k},
\]
then complementation is a measure-preserving bijection between
\(\mathcal S_{d,k}\) and \(\mathcal S_{d,d-k}\) that maps admissible
exchanges bijectively and hence intertwines the two Johnson walks. Since the scaling factor \(k(d-k)/d\) is
unchanged, the same relation holds for the corresponding operators
\(L_J\) and their spectral powers.  Thus,
\[
  I_{J,d,d-k}^{(2,s)}(f^{\mathrm c})
  =
  I_{J,d,k}^{(2,s)}(f).
\]

Because the constant component \(P_0f\) does not appear in
\eqref{eq:johnson-sobolev}, the energy is a squared seminorm rather
than a squared norm.  Its square root is unchanged when a constant is
added to \(f\), and it vanishes exactly when \(f\) is constant.

Finally, for \(0\leqslant D<m\), the monotonicity of the spectral
weights gives
\[
  I_{J,d,k}^{(2,s)}(f)
  \geqslant
  \sum_{r>D}
  \gamma_{r,d}^{\,s}\norm{P_rf}_2^2
  \geqslant
  \gamma_{D+1,d}^{\,s}
  \sum_{r>D}\norm{P_rf}_2^2.
\]
Therefore,
\begin{equation}\label{eq:harmonic-tail}
  \sum_{r>D}\norm{P_rf}_2^2
  \leqslant
  \frac{I_{J,d,k}^{(2,s)}(f)}
       {\gamma_{D+1,d}^{\,s}}.
\end{equation}
This is the approximation inequality used in the statistical
analysis below.

\subsection{Projection estimation}

For \(s,B>0\), consider the centered ellipsoid
\begin{equation}\label{eq:johnson-ellipsoid}
  \mathcal E^s_{d,k}(B)
  =
  \left\{
    f\in L^2(\Sdk,\nu_{d,k}):
    P_0f=0,\ 
    I_{J,d,k}^{(2,s)}(f)\leqslant B
  \right\}.
\end{equation}
Because \(\gamma_{r,d}\geqslant1\) for \(r\geqslant1\),
\begin{equation}\label{eq:ellipsoid-l2-bound}
  \norm{f}_2^2\leqslant B,
  \qquad f\in\mathcal E^s_{d,k}(B).
\end{equation}
Since \(V_0=\operatorname{span}\{\one\}\), the constant component of
any function \(f\) is
\[
  P_0f
  =
  \E_{S\sim\nu_{d,k}}[f(S)]\,\one.
\] 
Centering therefore isolates the nonconstant component controlled by
the energy.  More precisely, suppose that the response can be written as
\[
  f=c+g,
  \qquad
  c=\E_{S\sim\nu_{d,k}}[f(S)],
  \qquad
  g\in\mathcal E^s_{d,k}(B).
\]
Since \(g\) and the noise are centered,
\(\overline Y=n^{-1}\sum_{j=1}^nY_j\) is unbiased for \(c\), and
\[
  \E_{f,P}[(\overline Y-c)^2]
  =\frac{\norm{g}_2^2+\Var_P(\varepsilon)}n
  \leqslant\frac{B+\sigma^2}{n}.
\]
Thus an unknown intercept can be handled at the parametric rate, for
example on an independent sample split. 
No intercept is included in
the minimax class \eqref{eq:johnson-ellipsoid}.

For \(0\leqslant D\leqslant m\), let
\[
  W_D=V_1\mathbin{\oplus^\perp}\cdots
      \mathbin{\oplus^\perp}V_D,
  \qquad
  P_{\leqslant D}^{\circ}=P_1+\cdots+P_D,
\]
with \(W_0=\{0\}\) and \(P_{\leqslant0}^{\circ}=0\).  The constant level \(V_0\) is omitted because the response class
\(\mathcal E^s_{d,k}(B)\) is centered. The dimension of \(W_D\) is
\[
  p_D:=\dim(W_D)=\sum_{r=1}^D\dim(V_r)=\binom dD-1,
\]
including \(p_0=0\).

Choose any orthonormal basis
\(e_1,\ldots,e_{p_D}\) of \(W_D\).  The coefficient of the
population projection \(P_{\leqslant D}^{\circ}f\) in direction
\(e_\ell\) is
\[
  \theta_\ell
  =
  \inner{f}{e_\ell}
  =
  \E_{S\sim\nu_{d,k}}[f(S)e_\ell(S)].
\]
Since the noise is centered and independent of the design,
\(
  \E_{f,P}[Y_je_\ell(S_j)]
  =
  \theta_\ell.
\)
The sample average
\[
  \widehat\theta_\ell
  =
  \frac1n\sum_{j=1}^nY_je_\ell(S_j)
\]
is therefore an unbiased estimator of the population projection
coefficient \(\theta_\ell\).  This leads to the empirical projection
estimator
\begin{equation}\label{eq:projection-estimator}
  \widehat f_D^{\,\mathrm{proj}}
  =
  \sum_{\ell=1}^{p_D}
  \widehat\theta_\ell e_\ell.
\end{equation}

When \(D=0\), the sum in
\eqref{eq:projection-estimator} is empty, and hence
\(\widehat f_0^{\,\mathrm{proj}}=0\). Although this definition uses an orthonormal basis, the resulting
estimator depends only on \(W_D\).  Indeed, substituting the
definition of \(\widehat\theta_\ell\) gives
\[
  \widehat f_D^{\,\mathrm{proj}}(S)
  =
  \sum_{\ell=1}^{p_D}
  \left\{
    \frac1n\sum_{j=1}^n
    Y_je_\ell(S_j)
  \right\}e_\ell(S)
  =
  \frac1n\sum_{j=1}^n
  Y_j\mathcal K_D(S_j,S),
\]
where
\[
  \mathcal K_D(S,S')
  =
  \sum_{\ell=1}^{p_D}
  e_\ell(S)e_\ell(S')
\]
is the orthogonal projection kernel of \(W_D\).  Both the kernel and the estimator are basis-independent, since the
kernel represents the orthogonal projector onto \(W_D\) and the
estimator is expressed in terms of this kernel.

\begin{remark}
Section~\ref{ssec:canonical-frame} of the Supplementary Material gives an alternative
permutation-equivariant representation of the same estimator in terms
of a canonical tight frame indexed by subsets.
\end{remark}

Define the worst-case approximation bound
\[
  b_D(s,B)
  =
  \begin{cases}
    \displaystyle
    \frac{B}{\gamma_{D+1,d}^{\,s}},
      &0\leqslant D<m,\\[1.2ex]
    0,&D=m.
  \end{cases}
\]

\begin{proposition}[Projection risk]\label{prop:projection-risk}
For every \(s,B>0\) and \(0\leqslant D\leqslant m\),
\begin{equation}\label{eq:projection-risk}
  \sup_{f\in\mathcal E^s_{d,k}(B)}
  \sup_{P\in\mathcal P_{\sigma^2}}
  \E_{f,P}\!\left[
    \norm{\widehat f_D^{\,\mathrm{proj}}-f}_2^2
  \right]
  \leqslant
  \left(\sigma^2+B\right)\frac{p_D}{n}
  +
  b_D(s,B).
\end{equation}
\end{proposition}

The two terms on the right-hand side of
\eqref{eq:projection-risk} bound distinct contributions:
\((\sigma^2+B)p_D/n\) bounds the coefficient-estimation term, whereas
\(b_D(s,B)\) bounds the approximation term contributed by the omitted
harmonic levels. The key fact in the variance calculation is that
the projection kernel has constant diagonal:
\begin{equation}\label{eq:constant-kernel-diagonal}
  \mathcal K_D(S,S)
  =
  p_D,
  \qquad S\in\Sdk.
\end{equation}
This identity follows from the permutation invariance of \(W_D\);
its proof is given together with that of
Proposition~\ref{prop:projection-risk}.  It yields
\[
  \E_{f,P}\!\left[
    \norm{
      \widehat f_D^{\,\mathrm{proj}}
      -P_{\leqslant D}^{\circ}f
    }_2^2
  \right]
  \leqslant
  (\sigma^2+B)\frac{p_D}{n}.
\]
The estimation component
\(
  \widehat f_D^{\,\mathrm{proj}}
  -
  P_{\leqslant D}^{\circ}f
\)
belongs to \(W_D\), whereas the approximation component
\(
  f-P_{\leqslant D}^{\circ}f
\)
belongs to \(W_D^\perp\). Their
squared norms therefore add exactly, and
\eqref{eq:harmonic-tail} bounds the latter by \(b_D(s,B)\).

The contribution \(Bp_D/n\) is not observation noise.  It arises
because the population projection coefficients are estimated by
averages over randomly sampled subsets, which fluctuate with the
design even when the responses are noiseless.

\subsection{Nonasymptotic minimax risk}

For \(1\leqslant D\leqslant m\), define
\[
  a_D(s,B)
  =
  \frac{B}{\gamma_{D,d}^{\,s}},
  \qquad
  v_D(n,\sigma^2)
  =
  \frac{\sigma^2p_D}{n}.
\]
The quantity \(a_D(s,B)\) is the squared radius of the largest \(L^2\)-ball centered at zero in \(W_D\) and contained in \(\mathcal E^s_{d,k}(B)\). Indeed, for every \(g\in W_D\),
\[
  I_{J,d,k}^{(2,s)}(g)
  =
  \sum_{r=1}^D
  \gamma_{r,d}^{\,s}\norm{P_rg}_2^2
  \leqslant
  \gamma_{D,d}^{\,s}
  \sum_{r=1}^D\norm{P_rg}_2^2
  =
  \gamma_{D,d}^{\,s}\norm{g}_2^2.
\]
Consequently,
\[
  \{g\in W_D:\norm{g}_2^2\leqslant a_D(s,B)\}
  \subseteq
  \mathcal E^s_{d,k}(B).
\]
The second quantity, \(v_D(n,\sigma^2)\), is the usual parametric risk
scale for estimating \(p_D\) coefficients from \(n\) observations
under Gaussian noise with variance \(\sigma^2\).  For each \(D\), a Gaussian testing argument on this ball gives a risk
lower bound proportional to the smaller of \(a_D(s,B)\) and
\(v_D(n,\sigma^2)\). Taking the largest of these bounds over \(D\),
we define the
spectral scale
\begin{equation}\label{eq:spectral-minimax-scale}
 \mathfrak  R_{n,d,k}(s,B,\sigma^2)
  =
  \max_{1\leqslant D\leqslant m}
  \min\left\{
    a_D(s,B),
    v_D(n,\sigma^2)
  \right\}.
\end{equation}
For the upper bound, we optimize the empirical-projection risk bound
over the cutoff \(D\). Increasing \(D\) decreases the approximation
error \(b_D(s,B)\) but increases the estimation cost. The proof shows that the optimized projection bound is at most \(2(1+B/\sigma^2)\) times \(\mathfrak R_{n,d,k}(s,B,\sigma^2)\).

\begin{theorem}[Uniform nonasymptotic minimax bounds]
\label{thm:johnson-minimax}
There exists a universal constant \(c>0\) such that, for every
\(d\geqslant2\), \(1\leqslant k\leqslant d-1\), \(n\geqslant 1\),
and \(s,B,\sigma^2>0\),
\begin{equation}\label{eq:johnson-minimax}
  c\,\mathfrak R_{n,d,k}(s,B,\sigma^2)
  \leqslant
  R_n^\star(\mathcal E^s_{d,k}(B))
  \leqslant
  2\left(1+\frac{B}{\sigma^2}\right)
  \mathfrak R_{n,d,k}(s,B,\sigma^2).
\end{equation}
Consequently, for every fixed \(C_{\mathrm{snr}}>0\), the condition
\(
  {B}/{\sigma^2}\leqslant C_{\mathrm{snr}}
\)
implies
\[
  c\,\mathfrak R_{n,d,k}(s,B,\sigma^2)
  \leqslant
  R_n^\star(\mathcal E^s_{d,k}(B))
  \leqslant
  2(1+C_{\mathrm{snr}})
  \mathfrak R_{n,d,k}(s,B,\sigma^2).
\]
This comparison is nonasymptotic and holds uniformly over all
\(d,k,n,s,B,\sigma^2\) satisfying the displayed condition.
\end{theorem}

For the upper bound, choose the cutoff that minimizes the projection
risk bound in Proposition~\ref{prop:projection-risk}:
\[
  D_{\mathrm{or}}
  \in
  \arg\min_{0\leqslant D\leqslant m}
  \left\{
    (\sigma^2+B)\frac{p_D}{n}+b_D(s,B)
  \right\}.
\]
By Proposition~\ref{prop:projection-risk}, the worst-case risk at
this cutoff is bounded by 
\begin{align*}
  (\sigma^2+B)\frac{p_{D_{\mathrm{or}}}}{n}
    +b_{D_{\mathrm{or}}}(s,B)
  &=
  \min_{0\leqslant D\leqslant m}
  \left\{(\sigma^2+B)\frac{p_D}{n}+b_D(s,B)\right\}\\
  &\leqslant
  \left(1+\frac{B}{\sigma^2}\right)
  \min_{0\leqslant D\leqslant m}
  \left\{\frac{\sigma^2p_D}{n}+b_D(s,B)\right\}\\
  &\leqslant
  2\left(1+\frac{B}{\sigma^2}\right)
  \mathfrak R_{n,d,k}(s,B,\sigma^2).
\end{align*}
The first inequality uses \(b_D(s,B)\geqslant0\), and the final
inequality follows from the crossing comparison in the proof of
Theorem~\ref{thm:johnson-minimax}. Thus, when \(B/\sigma^2\) is bounded, empirical projection at
\(D_{\mathrm{or}}\) matches the minimax lower bound up to constants.
This cutoff is an oracle choice because it depends on \(s\), \(B\),
and \(\sigma^2\), although the estimator at any fixed \(D\) does not.
We do not claim data-driven adaptation to these parameters. Both inequalities in \eqref{eq:johnson-minimax} hold without a
restriction on \(B/\sigma^2\), but their comparison deteriorates as
this ratio grows. Section~\ref{sec:random-design} examines the effects
of random sampling and develops risk bounds for arbitrary
\(B/\sigma^2\).

The nonasymptotic formulation permits \(d\) and \(k\) themselves to vary with \(n\).  There is therefore no
single rate as a function of \(n\) alone without specifying their joint
growth.  Along any such sequence, Proposition~\ref{prop:projection-risk}
shows that, for fixed \(s\), \(B\), and \(\sigma^2\), the worst-case risk of
\(\widehat f_{D_n}^{\,\mathrm{proj}}\) over
\(\mathcal E^s_{d,k}(B)\) and \(\mathcal P_{\sigma^2}\) tends to zero if one can choose cutoffs
\(0\leqslant D_n\leqslant m\) such that
\[
  \frac{p_{D_n}}{n}\longrightarrow0
  \qquad\text{and}\qquad
  b_{D_n}(s,B)\longrightarrow0.
\]
The first condition requires the sample size to dominate the dimension
of the retained interaction space, while the second requires its
omitted harmonic tail to vanish.

\section{Active-item structure and its sampling cost}
\label{sec:active-items}

Suppose that only a few items can affect the response. The indicator
expansion discussed in Section~\ref{sec:harmonic-analysis} does not
express this property intrinsically, since its coefficients are
nonunique on \(\Sdk\). The harmonic spaces \(V_r\) are intrinsic,
but sparsity of coefficients in a chosen basis of these spaces depends
on that basis. Requiring only a few nonzero harmonic levels concerns
interaction orders rather than the number of relevant items. We
therefore define item sparsity through the function's dependence on
a subset of \([d]\).

For \(J\subseteq[d]\), let
\[
  \mathcal Q_J
  =\{f:\Sdk\to\R:
      f(S)=g(S\cap J)\text{ for some }g:2^J\to\R\}.
\]
Here \(2^J\) is the set of all subsets of \(J\), including
\(\varnothing\). Functions in \(\mathcal Q_J\) are \(J\)-juntas, meaning that only
items in \(J\) can affect the response. In particular, if \(S\cap J=\varnothing\), then
\(f(S)=g(\varnothing)\). For \(0\leqslant D\leqslant m\), define
\(W_{J,D}:=W_D\cap\mathcal Q_J\).

\begin{proposition}[Active-item harmonic dimensions]
\label{prop:active-dimension}
Suppose \(1\leqslant q=|J|\leqslant d-k\), and set
\(r_J=\min(q,m)\). For \(0\leqslant r\leqslant m\),
\(P_r\mathcal Q_J\subseteq\mathcal Q_J\), and for
\(0\leqslant D\leqslant r_J\),
\[
p_{q,D}:=
  \dim(W_{J,D})
  =\sum_{r=1}^D\binom qr.
\]
Every \(f\in\mathcal Q_J\) has \(P_rf=0\) for \(r>r_J\).
Consequently, \(W_{J,D}=W_{J,r_J}\) for \(r_J<D\leqslant m\). Moreover, \(\sum_{r=0}^{r_J}P_rf=f\) for every
\(f\in\mathcal Q_J\).
\end{proposition}

The condition \(q\leqslant d-k\) ensures that a \(k\)-set
disjoint from \(J\) exists. More generally, every intersection
pattern \(S\cap J\) of size at most \(\min(q,k)\) can then occur.
This condition allows \(q\geqslant k\), which is useful for pairwise screens:
when \(k=2\) and \(2\leqslant q\leqslant d-2\),
\(p_{q,2}=q+\binom q2\). For example, in the NCI-ALMANAC screen with \(d=104\) drugs, a
support of size \(20\) has \(210\) nonconstant degrees of freedom,
compared with \(\binom{104}{2}-1=5{,}355\) in the dense order-two
space.

We first give an upper bound for a fixed, known \(J\). Write
\(\mathcal K_{J,D}\) for the orthogonal projection kernel of
\(W_{J,D}\), and let
\[
  \Lambda_{J,D}
  =\max_{S\in\Sdk}\mathcal K_{J,D}(S,S).
\]
Permutations preserving \(J\) are not transitive on \(\Sdk\), so
symmetry no longer forces \(\mathcal K_{J,D}(S,S)\) to be constant,
unlike the full-space diagonal \(\mathcal K_D(S,S)=p_D\). Define \(\widehat f_{J,D}^{\,\mathrm{proj}}\) by the
same empirical-coefficient construction as
\eqref{eq:projection-estimator}, using an orthonormal basis of
\(W_{J,D}\). For \(0\leqslant D\leqslant r_J\), \(f\in\mathcal E^s_{d,k}(B)\cap\mathcal Q_J\), and \(P\in\mathcal P_{\sigma^2}\), the proof of Proposition~\ref{prop:projection-risk} yields
\begin{equation}\label{eq:active-projection-risk}
  \E_{f,P}\!\left[
    \|\widehat f_{J,D}^{\,\mathrm{proj}}-f\|_2^2
  \right]
  \leqslant
  \frac{\sigma^2p_{q,D}+B\Lambda_{J,D}}n
  +b_{q,D}(s,B),
\end{equation}
where \(b_{q,D}(s,B)=B/\gamma_{D+1,d}^{\,s}\) for
\(D<r_J\), and \(b_{q,r_J}(s,B)=0\). The noise contribution depends on the reduced dimension, whereas
the design contribution involves \(\Lambda_{J,D}\), the maximum
diagonal value of the projection kernel.

For the same fixed, known \(J\), the lower-bound argument of
Theorem~\ref{thm:johnson-minimax} applies inside \(W_{J,D}\).
For \(1\leqslant D\leqslant r_J\), the minimax risk over
\(\mathcal E^s_{d,k}(B)\cap\mathcal Q_J\) is at least a universal
constant times
\[
  \min\left\{
    \frac{B}{\gamma_{D,d}^{\,s}},
    \frac{\sigma^2p_{q,D}}n
  \right\}.
\]
The lower bound involves \(p_{q,D}\), whereas the projection upper
bound \eqref{eq:active-projection-risk} also involves \(\Lambda_{J,D}\). To see why this matters, let
\(k=D=2\), \(2\leqslant q\leqslant d-2\), and \(N=\binom d2\).
A specified pair \(S_0\subseteq J\) is sampled with probability
\(1/N\), and its centered indicator belongs to \(W_{J,2}\).
As shown at the end of the proof of
Proposition~\ref{prop:active-dimension} in the Supplementary Material, \(\Lambda_{J,2}=N-1\), the full-space diagonal value.
Yet
\(\E_{S\sim\nu_{d,2}}\mathcal K_{J,2}(S,S)
=p_{q,2}=\binom{q+1}{2}\leqslant\binom{d-1}{2}<N-1\).
Thus, despite the smaller dimension \(p_{q,2}\), the maximal
diagonal \(\Lambda_{J,2}\) remains as large as in the full space.

When \(J\) is unknown, the estimator must also determine which
items influence the response. To quantify this difficulty, define 
\[
  \mathcal E^s_{d,k}(B;q)
  =\bigcup_{\substack{J\subseteq[d]\\ |J|\leqslant q}}
  \bigl(\mathcal E^s_{d,k}(B)\cap\mathcal Q_J\bigr).
\]
\begin{proposition}[Cost of an unknown active set]
\label{prop:unknown-active-lower}
There is a universal constant \(c>0\) such that, for
\(1\leqslant q\leqslant d/4\),
\[
  R_n^\star\bigl(\mathcal E^s_{d,k}(B;q)\bigr)
  \geqslant
  c\min\left\{B,\frac{\sigma^2q\log(ed/q)}n\right\}.
\]
\end{proposition}
The proof uses a packing of active sets and is given in the
Supplementary Material.

A practical support can be built by forward selection. Fix a cap \(1\leqslant q_{\max}\leqslant d-k\), and split the training data into a fitting
sample and two validation samples. For each
\(1\leqslant D\leqslant\min(q_{\max},m)\), start from
\(J_0=\varnothing\). At step \(t\), fit the models on
\(W_{J_t\cup\{i\},D}\) for \(i\notin J_t\), and use the first
validation sample to choose the next item. Use the second to select
\((t,D)\), then refit on the full training sample. The largest fitted
space for degree \(D\) has dimension \(p_{q_{\max},D}\); across all
degrees, the procedure examines
\(O(dq_{\max}\min(q_{\max},m))\) candidate extensions. We make no
claim of support recovery or adaptive risk control for this rule.

\section{Random-design uncertainty and stable least squares}
\label{sec:random-design}

As we have seen in the preceding sections, empirical projection can
fluctuate with the sampled subsets even without observation noise. We first ask when least squares avoids
this fluctuation for the component already in \(W_D\), and use
concentration of the empirical Gram matrix to establish a stable
fit. We then measure the information lost when the sample leaves
directions in \(W_D\) unidentified. A centered completion estimator
allows us to combine these findings into minimax bounds for arbitrary
signal-to-noise ratios.

\subsection{Constant leverage and Gram concentration}

Empirical projection estimates coefficients by averaging
\(Y_j e_\ell(S_j)\). These averages fluctuate with the sampled
subsets even when the observations are noiseless and \(f\in W_D\).
Least squares instead fits the responses jointly in \(W_D\) and
recovers such an \(f\) exactly if the sampled evaluations identify
a unique function in \(W_D\). To make this recovery stable in the
presence of noise, we ask whether
\(n^{-1}\sum_{j=1}^n g(S_j)^2\) is comparable to \(\|g\|_2^2\)
simultaneously for all \(g\in W_D\).

Fix \(1\leqslant D\leqslant m\), and let
\(e_1,\ldots,e_{p_D}\) be an orthonormal basis of \(W_D\).  For
\(S\in\Sdk\), define
\(
  x_D(S)
  =
  (e_1(S),\ldots,e_{p_D}(S))^\top,
\)
and define the empirical Gram matrix by
\[
  \widehat G_D
  =
  \frac1n\sum_{j=1}^n
  x_D(S_j)x_D(S_j)^\top.
\]
Orthonormality gives
\(
  \E_{S\sim\nu_{d,k}}
  [x_D(S)x_D(S)^\top]
  =
  I_{p_D},
\)
so the population Gram matrix is the identity. The leverage at \(S\) is
\(\|x_D(S)\|_2^2=\mathcal K_D(S,S)\). By
\eqref{eq:constant-kernel-diagonal}, it equals \(p_D\) at every
subset. Its maximum is therefore as small as possible for a
\(p_D\)-dimensional space, since the average leverage is \(p_D\).
This uniform bound on the evaluation vectors drives Gram
concentration.

Define
\[
  \mathcal A_D
  =
  \left\{
    \|\widehat G_D-I_{p_D}\|_{\mathrm{op}}
    \leqslant\frac12
  \right\}.
\]

\begin{proposition}[Johnson Gram concentration]
\label{prop:johnson-gram}
With
\[
  c_0
  =
  \frac32\log\left(\frac32\right)-\frac12>0,
\]
one has
\begin{equation}\label{eq:gram-concentration}
  \P_{\nu_{d,k}^{\otimes n}}(\mathcal A_D^c)
  \leqslant
  2p_D\exp\left(-c_0\frac{n}{p_D}\right).
\end{equation}
\end{proposition}

For \(g\in W_D\), write \(g(S)=u^\top x_D(S)\) with
\(u\in\mathbb R^{p_D}\). Orthonormality gives
\(\|g\|_2^2=\|u\|_2^2\), while the definition of
\(\widehat G_D\) gives
\[
  \frac1n\sum_{j=1}^n g(S_j)^2
  =u^\top\widehat G_Du.
\]
Since \(\widehat G_D\) is symmetric, \(\mathcal A_D\) is exactly
the event on which all its eigenvalues lie in \([1/2,3/2]\). The Rayleigh quotient therefore yields
\[
  \frac12\|g\|_2^2
  \leqslant \frac1n\sum_{j=1}^n g(S_j)^2
  \leqslant \frac32\|g\|_2^2
  \qquad(g\in W_D).
\]
By \eqref{eq:gram-concentration}, a sample of order
\(p_D\log(2p_D/\delta)\) ensures
\(\mathbb P(\mathcal A_D)\geqslant 1-\delta\). 
We now use this high-probability norm comparison to define a stable
least-squares estimator and bound its risk.

\subsection{Stable least squares}

Throughout this subsection, fix a deterministic cutoff
\(1\leqslant D\leqslant m\).  This cutoff determines both the fitted
space \(W_D\) and the stability event \(\mathcal A_D\); no
data-driven selection of \(D\) is considered here.  As established in the preceding subsection, \(\widehat G_D\) is
well conditioned on \(\mathcal A_D\).  Least squares can therefore be applied without the
error amplification caused by a nearly singular Gram matrix. Accordingly, define the stable least-squares estimator by
\[
  \widehat f_D^{\,\mathrm{LS}}
  =
  \begin{cases}
    \displaystyle
    \operatorname*{arg\,min}_{g\in W_D}
    \frac1n\sum_{j=1}^n\{Y_j-g(S_j)\}^2,
      & \text{on }\mathcal A_D,\\[3ex]
    0,
      & \text{on }\mathcal A_D^c.
  \end{cases}
\]
On \(\mathcal A_D\), \(\lambda_{\min}(\widehat G_D)\geqslant1/2\),
so \(\widehat G_D\) is invertible and the least-squares minimizer is unique. Since
\(\mathcal A_D\) depends only on the observed design points, this
fallback rule is data-driven and prevents an ill-conditioned Gram
matrix from producing an arbitrarily large estimate. On this event, the coefficient vector in the chosen orthonormal basis is
\begin{equation}\label{eq:ls-coefficients}
  \widehat\theta_D^{\,\mathrm{LS}}
  =\widehat G_D^{-1}\frac1n\sum_{j=1}^nY_jx_D(S_j).
\end{equation}
Although this formula uses a basis, the fitted function is
basis-independent, since it is the unique least-squares minimizer
over \(W_D\).

Write \(g_D=P_{\leqslant D}^{\circ}f\), \(h_D=f-g_D\), and
\(g_D(S)=x_D(S)^\top\theta_D\). Then
\(Y_j=g_D(S_j)+h_D(S_j)+\varepsilon_j\), and substitution into
\eqref{eq:ls-coefficients} gives, on \(\mathcal A_D\),
\begin{equation}\label{eq:ls-error-decomposition}
  \widehat\theta_D^{\,\mathrm{LS}}-\theta_D
  =\widehat G_D^{-1}
  \left\{\frac1n\sum_{j=1}^nx_D(S_j)h_D(S_j)
  +\frac1n\sum_{j=1}^nx_D(S_j)\varepsilon_j\right\}.
\end{equation}
The fitted component cancels exactly. Moreover,
\(h_D\perp W_D\) makes the summands involving \(h_D\) centered;
only this omitted component and observation noise contribute to
the coefficient error.

\begin{theorem}[Stable least-squares risk]
\label{thm:stable-ls}
For every \(s,B>0\) and \(1\leqslant D\leqslant m\),
\begin{align}
  &\sup_{f\in\mathcal E^s_{d,k}(B)}
   \sup_{P\in\mathcal P_{\sigma^2}}
   \E_{f,P}\!\left[
     \|\widehat f_D^{\,\mathrm{LS}}-f\|_2^2
   \right]
  \nonumber\\
  &\quad\leqslant
  \left(1+4\frac{p_D}{n}\right)b_D(s,B)
  +4\sigma^2\frac{p_D}{n}
  +B
   \min\left\{
     1,\,
     2p_D\exp\left(-c_0\frac{n}{p_D}\right)
   \right\}.
  \label{eq:stable-ls-risk}
\end{align}
\end{theorem}

The cutoff in Theorem~\ref{thm:stable-ls} is deterministic.
The first term in \eqref{eq:stable-ls-risk} combines approximation
error with the empirical correlation of the omitted component
\(h_D\) with \(W_D\). The second term is observation noise, while
the last accounts for the fallback when the Gram matrix is unstable.
By \eqref{eq:ls-error-decomposition}, the fitted component \(g_D\)
contributes no design variance. In particular, if \(h_D=0\) and the
observations are noiseless, then
\(\widehat f_D^{\,\mathrm{LS}}=f\) on \(\mathcal A_D\).

For a concrete bound, suppose that, for some \(\kappa>0\),
\[
  n\geqslant \frac{\kappa+1}{c_0}\,p_D\log(2p_D).
\]
Then \eqref{eq:stable-ls-risk} gives
\[
  \sup_{f\in\mathcal E^s_{d,k}(B)}
  \sup_{P\in\mathcal P_{\sigma^2}}
  \E_{f,P}\!\left[
    \|\widehat f_D^{\,\mathrm{LS}}-f\|_2^2
  \right]
  \leqslant
  5b_D(s,B)+4\sigma^2\frac{p_D}{n}+Bp_D^{-\kappa}.
\]
Thus the design-variance term \(Bp_D/n\) associated with empirical
projection of the fitted component is absent.

This guarantee relies on Gram stability. If the sampled evaluation
map loses rank, distinct functions in \(W_D\) can agree at every sampled subset. 
The next subsection quantifies these
unidentified directions and relates them, at the full level, to
unobserved subsets.

\subsection{Unidentified directions and incomplete coverage}

Proposition~\ref{prop:johnson-gram} ensures that, on
\(\mathcal A_D\), the sampled evaluation map is injective and well
conditioned on \(W_D\).  We now study the complementary situation in which the sampled
evaluations do not uniquely identify functions in \(W_D\), and
quantify the part of this space that remains invisible to the design. This loss of information persists even in the absence
of observation noise.

Let \(\mathbf X_D\) be the \(n\times p_D\) matrix whose \(j\)-th row
is \(x_D(S_j)^\top\).  If
\[
  g=\sum_{\ell=1}^{p_D}u_\ell e_\ell\in W_D,
\]
then the vector of its sampled values is
\[
  \bigl(g(S_1),\ldots,g(S_n)\bigr)^\top
  =
  \mathbf X_Du.
\]
The sampled values depend on the coefficient vector \(u\) only through
\(\mathbf X_Du\). Consequently, coefficient vectors differing by an element of
\(\ker(\mathbf X_D)\) give the same fitted values at every sampled subset. Thus,
\(\operatorname{rank}(\mathbf X_D)\) is the number of linearly
independent measurements of the coefficient vector supplied by the
design, whereas
\[
  p_D-\operatorname{rank}(\mathbf X_D)
  =
  \dim\ker(\mathbf X_D)
\]
is the dimension of the subspace of coefficient perturbations that
remain invisible. Such directions can occur only on \(\mathcal A_D^c\), since
\(\widehat G_D=n^{-1}\mathbf X_D^\top\mathbf X_D\) is invertible
on \(\mathcal A_D\), and hence \(\ker(\mathbf X_D)=\{0\}\) there. We measure the expected proportion of such
unidentified directions by
\begin{equation}\label{eq:rank-deficiency}
  \rho_{D,n}
  =
  \frac1{p_D}
  \E_{\nu_{d,k}^{\otimes n}}
  \left[
    p_D-\operatorname{rank}(\mathbf X_D)
  \right],
  \qquad
  1\leqslant D\leqslant m.
\end{equation}
This quantity does not depend on the chosen orthonormal basis: a
change of basis multiplies \(\mathbf X_D\) on the right by an
invertible orthogonal matrix and therefore leaves its rank unchanged.

\begin{proposition}[Rank-deficiency lower bound]
\label{prop:rank-deficiency}
For every \(s,B,\sigma^2>0\),
\begin{equation}\label{eq:rank-deficiency-lower}
  R_n^\star(\mathcal E^s_{d,k}(B))
  \geqslant
  \max_{1\leqslant D\leqslant m}
  \frac{B}{\gamma_{D,d}^{\,s}}\rho_{D,n}.
\end{equation}
Moreover, for every \(1\leqslant D\leqslant m\),
\begin{equation}\label{eq:rank-deficiency-dimension}
  \rho_{D,n}
  \geqslant
  \left(1-\frac{n}{p_D}\right)_+.
\end{equation}
At the full level \(D=m\), let
\[
  \mathcal O_n
  =
  \{S_1,\ldots,S_n\}
  \subseteq\Sdk
\]
denote the set of distinct observed subsets, and let
\(
  M_n=N-|\mathcal O_n|
\)
be the number of unobserved subsets. Then
\begin{equation}\label{eq:full-rank-deficiency}
  \rho_{m,n}
  =
  \frac{
    \E_{\nu_{d,k}^{\otimes n}}[(M_n-1)_+]
  }{N-1}.
\end{equation}
Equivalently,
\begin{equation}\label{eq:full-rank-deficiency-explicit}
  \rho_{m,n}
  =
  \frac{
    N(1-1/N)^n-1
    +\P_{\nu_{d,k}^{\otimes n}}(M_n=0)
  }{N-1}.
\end{equation}
Here,
\[
  \P_{\nu_{d,k}^{\otimes n}}(M_n=0)
  =
  \sum_{j=0}^N
  (-1)^j\binom Nj
  \left(1-\frac jN\right)^n
\]
by inclusion--exclusion; in particular, this probability is zero when
\(n<N\).
\end{proposition}

To see where the lower bound comes from, take the admissible
noiseless law \(\delta_0\) and draw \(f\) uniformly from the sphere in \(W_D\) defined by
\(\|f\|_2^2=B/\gamma_{D,d}^{\,s}\). This sphere lies in \(\mathcal E^s_{d,k}(B)\). Conditional on the
design, the observations determine the row-space component of the
coefficient vector, but not its component in \(\ker(\mathbf X_D)\).
By spherical symmetry, the component in \(\ker(\mathbf X_D)\)
has mean zero conditional on the design and the noiseless
observations. Averaging its squared norm over the spherical prior,
conditional on the design, gives
\[
  \frac{B}{\gamma_{D,d}^{\,s}}\,
  \frac{p_D-\operatorname{rank}(\mathbf X_D)}{p_D}.
\]
This unresolved component gives the displayed lower bound on the
average squared error of any estimator. Averaging over the design
then gives \eqref{eq:rank-deficiency-lower}.

At \(D=m\), the \(M_n\) unobserved subsets support a space of
centered functions of dimension \((M_n-1)_+\) that vanish at every
observed input. The subtraction of one comes from the zero-mean constraint on the
missing values. In particular, if only one subset is missing, its
value is determined by the observed values. This yields
\eqref{eq:full-rank-deficiency}.

\subsection{Centered completion}

At the full level, the rank-deficiency obstruction suggests a natural
estimator.  Repeated observations of the same subset are first
averaged, and the values at unobserved subsets are then chosen to
satisfy the centering constraint.

For every \(S\in\Sdk\), let
\(C_n(S)=\sum_{j=1}^n\one_{\{S_j=S\}}\). Then
\[
  \mathcal O_n=\{S\in\Sdk:C_n(S)>0\},
  \qquad
  M_n=N-|\mathcal O_n|
     =\#\{S\in\Sdk:C_n(S)=0\}.
\]
For \(S\in\mathcal O_n\), define
\(\overline Y(S)=C_n(S)^{-1}\sum_{j:S_j=S}Y_j\).
If \(M_n\geqslant1\), set
\[
  \widehat f^{\,\mathrm{comp}}(S)
  =\begin{cases}
    \overline Y(S),&S\in\mathcal O_n,\\
    -M_n^{-1}\sum_{T\in\mathcal O_n}\overline Y(T),
      &S\notin\mathcal O_n.
  \end{cases}
\]
If \(M_n=0\), every subset is observed, and we instead set
\[
  \widehat f^{\,\mathrm{comp}}(S)
  =\overline Y(S)-\frac1N\sum_{T\in\Sdk}\overline Y(T).
\]
Both rules produce a centered function. For each fixed \(S\), \(C_n(S)\sim\operatorname{Bin}(n,1/N)\).
With \(C\) having this distribution, define
\begin{equation}\label{eq:inverse-count-factor}
  \eta_{n,N}=\sum_{c=1}^n c^{-1}\P(C=c).
\end{equation}
\begin{proposition}[Risk of centered completion]
\label{prop:centered-completion}
For every \(s,B>0\),
\begin{equation}\label{eq:centered-completion-risk}
  \sup_{f\in\mathcal E^s_{d,k}(B)}
  \sup_{P\in\mathcal P_{\sigma^2}}
  \E_{f,P}\!\left[
    \norm{\widehat f^{\,\mathrm{comp}}-f}_2^2
  \right]
  \leqslant
  B\rho_{m,n}
  +
  2\sigma^2\eta_{n,N},
\end{equation}
where \(\eta_{n,N}\) is defined in
\eqref{eq:inverse-count-factor} and satisfies
\begin{equation}\label{eq:inverse-count-bound}
  \eta_{n,N}
  \leqslant
  \min\left\{
    1,\frac{2N}{n+1}
  \right\}.
\end{equation}
\end{proposition}

We sketch the proof of Proposition~\ref{prop:centered-completion}.
To isolate the error caused by incomplete coverage, let
\(\mathcal U_n=\Sdk\setminus\mathcal O_n\) and define
\[
  \mathcal H_n
  =\{h\in W_m:h(S)=0\text{ for every }S\in\mathcal O_n\}.
\]
This is the space of centered functions supported on the
unobserved subsets, with dimension \((M_n-1)_+\). Let \(\widehat f^{\,\mathrm{comp},0}\) be the completion estimator
applied to the exact observed values \(f(S)\). When \(M_n\geqslant1\),
it equals \(f\) on \(\mathcal O_n\) and the average of \(f\) over
\(\mathcal U_n\) on every unobserved subset; when \(M_n=0\), it
equals \(f\). Thus \(f-\widehat f^{\,\mathrm{comp},0}\) belongs to
\(\mathcal H_n\), while \(\widehat f^{\,\mathrm{comp},0}\) is
orthogonal to \(\mathcal H_n\). Consequently,
\[
  f-\widehat f^{\,\mathrm{comp},0}=P_{\mathcal H_n}f.
\]
Exchangeability of the uniform design over the \(N\) subsets gives
\[
  \E_{\nu_{d,k}^{\otimes n}}
  \|P_{\mathcal H_n}f\|_2^2
  =\rho_{m,n}\|f\|_2^2
  \leqslant B\rho_{m,n}.
\]
The remaining error from the observed cell means is bounded by
\(2\sigma^2\eta_{n,N}\). Details are given in the Supplementary
Material. We now combine this completion bound with the spectral and
rank-deficiency lower bounds to characterize the minimax risk.

\subsection{A rank-aware minimax characterization}

In view of \eqref{eq:rank-deficiency-lower}, define the design-deficiency scale
\[
  \mathfrak D_{n,d,k}(s,B)
  =
  \max_{1\leqslant D\leqslant m}
  \frac{B}{\gamma_{D,d}^{\,s}}\rho_{D,n}.
\]
This scale complements the noise-driven spectral scale
\(\mathfrak R_{n,d,k}(s,B,\sigma^2)\) defined in
\eqref{eq:spectral-minimax-scale}. The spectral scale measures the noise-driven cost of estimation over
the harmonic approximation spaces. The design-deficiency scale measures
the loss caused by directions that the sampled evaluations do not
identify, a phenomenon that is already present in the noiseless
submodel.

\begin{theorem}[Rank-aware minimax bounds]
\label{thm:rank-aware-minimax}
There exist universal constants \(c,C>0\) such that, for every
\(d\geqslant2\), \(1\leqslant k\leqslant d-1\), \(n\geqslant 1\),
and \(s,B,\sigma^2>0\),
\begin{align}
  c\left\{
    \mathfrak R_{n,d,k}(s,B,\sigma^2)
    +\mathfrak D_{n,d,k}(s,B)
  \right\}
  &\leqslant
  R_n^\star(\mathcal E^s_{d,k}(B))
  \nonumber\\
  &\leqslant
  C\gamma_{m,d}^{\,s}
  \left\{
    \mathfrak R_{n,d,k}(s,B,\sigma^2)
    +\mathfrak D_{n,d,k}(s,B)
  \right\}.
  \label{eq:rank-aware-minimax}
\end{align}
\end{theorem}

The lower inequality combines Theorem~\ref{thm:johnson-minimax}
and Proposition~\ref{prop:rank-deficiency}. Each gives a lower
bound, and their maximum controls their sum up to a factor of two.

For the upper inequality, put \(a_m=B/\gamma_{m,d}^{\,s}\).
If \(a_m\leqslant\sigma^2\), then
\(B/\sigma^2\leqslant\gamma_{m,d}^{\,s}\), and the empirical
projection bound applies. If \(a_m>\sigma^2\), use centered
completion. Its signal term satisfies
\(B\rho_{m,n}\leqslant\gamma_{m,d}^{\,s}
\mathfrak D_{n,d,k}(s,B)\), while its noise term is controlled by
\[
  \sigma^2\eta_{n,N}
  \leqslant
  4\min\left\{\sigma^2,\frac{\sigma^2(N-1)}n\right\}
  \leqslant 4\mathfrak R_{n,d,k}(s,B,\sigma^2).
\]
Since \(\gamma_{m,d}\leqslant m\), the resulting bounds match up to
constants when \(m\) and \(s\) are fixed, with no restriction on
\(B/\sigma^2\).

For a concrete rate, fix an integer \(k\geqslant1\) and \(s,B,\sigma^2>0\), let 
\(d=d_n\to\infty\), and put \(N_n=\binom{d_n}{k}\).
For all sufficiently large \(n\), \(d_n\geqslant2k\), so that
\(m=k\). The full-level spectral lower bound and the centered-completion
upper bound give
\[
  c\min\left\{
    \frac{B}{\gamma_{k,d_n}^{\,s}},
    \frac{\sigma^2(N_n-1)}n
  \right\}
  \leqslant
  R_n^\star(\mathcal E^s_{d_n,k}(B))
  \leqslant
  2B e^{-n/N_n}
  +4\sigma^2\frac{N_n}{n+1},
\]
where the upper bound uses
\(\rho_{m,n}\leqslant2e^{-n/N_n}\) and
\(\eta_{n,N_n}\leqslant2N_n/(n+1)\).
Because \(\gamma_{k,d_n}\to k\), minimax consistency holds exactly
when \(n/N_n\to\infty\). In this regime,
\(e^{-n/N_n}=o(N_n/n)\), and the risk has order
\(\sigma^2N_n/n\). For example, if \(d_n=k+\lceil\log n\rceil\),
then \(N_n\sim(\log n)^k/k!\), and the risk has order
\(\sigma^2(\log n)^k/n\). 

When \(m\) grows, the factor \(\gamma_{m,d}^{\,s}\) in the
upper bound may become substantial. In the centered-completion
regime, it enters because the signal-error bound is
\(B\rho_{m,n}\), whereas the full-level term of
\(\mathfrak D_{n,d,k}(s,B)\) is
\(B\rho_{m,n}/\gamma_{m,d}^{\,s}\). A sharper upper bound may
require using Johnson--Sobolev smoothness to estimate the values at
unobserved subsets.

Numerical experiments reported in
Section~\ref{sec:experiments} of the Supplementary Material complement
these theoretical results.  They compare prediction under different
interaction profiles and illustrate the signal-dependent fluctuation
of empirical projection, Gram stability, and the loss of information
caused by incomplete coverage.

\section{Discussion}\label{sec:discussion}

We have studied regression from noisy evaluations of uniformly
sampled \(k\)-subsets of a \(d\)-item ground set, with the aim of
estimating the response throughout this fixed-cardinality domain.
A local move between such subsets exchanges one selected item for
another, so we used the Johnson graph to describe their geometry.
Its harmonic decomposition organizes interactions, while its
Laplacian eigenvalues define smoothness through these exchanges.
Together, the level dimensions and eigenvalues determine the
approximation--estimation tradeoff behind our finite-sample risk
bounds for harmonic projection.

When \(B/\sigma^2\) is bounded, the spectral scale characterizes
the minimax risk up to constants uniformly in \(d\), \(k\), and
\(n\). For arbitrary ratios, we also account for the information
lost when sampled subsets leave parts of the response unidentified.
When the signal is large relative to the noise, centered completion
controls this loss directly. The resulting rank-aware bounds match
up to the factor \(\gamma_{m,d}^{\,s}\), and hence up to constants
when \(m\) and \(s\) are fixed.

We also studied sparsity in the form of juntas, whose responses
depend on only a few items. Their interaction spaces have smaller
dimension, but rare configurations can remain difficult to learn,
and an unknown set of relevant items adds a statistical cost. The NCI-ALMANAC pair screen illustrates the tradeoff between
restricted and full additive fits. Restricting the number of drugs helps the additive fit when training
pairs are scarce, but the selected set grows with the sample size.

Two questions remain central. Can the rank-aware bounds be sharpened
when \(m\) grows, and can we characterize the risk for juntas with
known or unknown relevant items? Further work could also address nonuniform sampling and prediction
across subset sizes, with suitable assumptions on the design and on
how responses relate across sizes.

\bibliographystyle{apalike}
\bibliography{references}

\begin{thebibliography}{}

\bibitem[Balcan et~al., 2012]{BalcanConstantinIwataWang2012}
Balcan, M.-F., Constantin, F., Iwata, S., and Wang, L. (2012).
\newblock Learning valuation functions.
\newblock In Mannor, S., Srebro, N., and Williamson, R.~C., editors, {\em
  Proc.\ 25th COLT}, volume~23, pages 4.1--4.24. PMLR.

\bibitem[Balcan et~al., 2016]{BalcanVitercikWhite2016}
Balcan, M.-F., Vitercik, E., and White, C. (2016).
\newblock Learning combinatorial functions from pairwise comparisons.
\newblock In Feldman, V., Rakhlin, A., and Shamir, O., editors, {\em Proc.\
  29th COLT}, volume~49, pages 310--335. PMLR.

\bibitem[Bannai and Ito, 1984]{BannaiIto1984}
Bannai, E. and Ito, T. (1984).
\newblock {\em Algebraic Combinatorics {I}: Association Schemes}.
\newblock Benjamin/Cummings, Menlo Park.

\bibitem[Benavoli et~al., 2023]{BenavoliAzzimontiPiga2023}
Benavoli, A., Azzimonti, D., and Piga, D. (2023).
\newblock Learning choice functions with {G}aussian processes.
\newblock In Evans, R.~J. and Shpitser, I., editors, {\em Proc.\ 39th UAI},
  volume 216, pages 141--151. PMLR.

\bibitem[Brouwer et~al., 1989]{BrouwerCohenNeumaier1989}
Brouwer, A.~E., Cohen, A.~M., and Neumaier, A. (1989).
\newblock {\em Distance-Regular Graphs}.
\newblock Springer, Berlin.

\bibitem[Cohen et~al., 2013]{CohenDavenportLeviatan2013}
Cohen, A., Davenport, M.~A., and Leviatan, D. (2013).
\newblock On the stability and accuracy of least squares approximations.
\newblock {\em Found. Comput. Math.}, 13:819--834.

\bibitem[Cohen et~al., 2019]{CohenDavenportLeviatan2019}
Cohen, A., Davenport, M.~A., and Leviatan, D. (2019).
\newblock Correction to: On the stability and accuracy of least squares
  approximations.
\newblock {\em Found. Comput. Math.}, 19:239.

\bibitem[De and Chakrabarti, 2022]{DeChakrabarti2022}
De, A. and Chakrabarti, S. (2022).
\newblock Neural estimation of submodular functions with applications to
  differentiable subset selection.
\newblock In Koyejo, S., Mohamed, S., Agarwal, A., Belgrave, D., Cho, K., and
  Oh, A., editors, {\em Adv. Neural Inf. Process. Syst.}, volume~35, pages
  19537--19552. Curran Associates, Inc.

\bibitem[Delsarte, 1973]{Delsarte1973}
Delsarte, P. (1973).
\newblock {\em An Algebraic Approach to the Association Schemes of Coding
  Theory}.
\newblock Number~10 in Philips Research Reports Supplements. N.V. Philips'
  Gloeilampenfabrieken, Eindhoven.
\newblock Doctoral dissertation, Universit{\'e} Catholique de Louvain.

\bibitem[Diaconis, 1988]{Diaconis1988}
Diaconis, P. (1988).
\newblock {\em Group Representations in Probability and Statistics}, volume~11
  of {\em Lecture Notes--Monograph Series}.
\newblock Institute of Mathematical Statistics, Hayward.

\bibitem[Diaconis, 1989]{Diaconis1989}
Diaconis, P. (1989).
\newblock A generalization of spectral analysis with application to ranked
  data.
\newblock {\em Ann. Statist.}, 17:949--979.

\bibitem[Diaconis and Rockmore, 1993]{DiaconisRockmore1993}
Diaconis, P. and Rockmore, D. (1993).
\newblock Efficient computation of isotypic projections for the symmetric
  group.
\newblock In Finkelstein, L. and Kantor, W.~M., editors, {\em Groups and
  Computation}, pages 87--104. American Mathematical Society, Providence.

\bibitem[Dolhansky and Bilmes, 2016]{DolhanskyBilmes2016}
Dolhansky, B.~W. and Bilmes, J.~A. (2016).
\newblock Deep submodular functions: Definitions and learning.
\newblock In Lee, D., Sugiyama, M., von Luxburg, U., Guyon, I., and Garnett,
  R., editors, {\em Adv. Neural Inf. Process. Syst.}, volume~29, pages
  3404--3412. Curran Associates, Inc.

\bibitem[Filmus, 2016]{Filmus2016}
Filmus, Y. (2016).
\newblock An orthogonal basis for functions over a slice of the {B}oolean
  hypercube.
\newblock {\em Electron. J. Combin.}, 23:P1.23.

\bibitem[Gottlieb, 1966]{Gottlieb1966}
Gottlieb, D.~H. (1966).
\newblock A certain class of incidence matrices.
\newblock {\em Proc. Amer. Math. Soc.}, 17:1233--1237.

\bibitem[Green et~al., 2021]{GreenBalakrishnanTibshirani2021}
Green, A., Balakrishnan, S., and Tibshirani, R.~J. (2021).
\newblock Minimax optimal regression over {S}obolev spaces via {L}aplacian
  regularization on neighborhood graphs.
\newblock In Banerjee, A. and Fukumizu, K., editors, {\em Proc.\ 24th AISTATS},
  volume 130, pages 2602--2610. PMLR.

\bibitem[Holbeck et~al., 2017]{HolbeckEtAl2017}
Holbeck, S.~L., Camalier, R., Crowell, J.~A., et~al. (2017).
\newblock The {National Cancer Institute} {ALMANAC}: A comprehensive screening
  resource for the detection of anticancer drug pairs with enhanced therapeutic
  activity.
\newblock {\em Cancer Res.}, 77:3564--3576.

\bibitem[Iglesias and Natale, 2022]{IglesiasNatale2022}
Iglesias, R. and Natale, M. (2022).
\newblock A fast {F}ourier transform for the {J}ohnson graph.
\newblock {\em J. Fourier Anal. Appl.}, 28:62.

\bibitem[Kirichenko and van Zanten, 2017]{KirichenkoVanZanten2017}
Kirichenko, A. and van Zanten, H. (2017).
\newblock Estimating a smooth function on a large graph by {B}ayesian
  {L}aplacian regularisation.
\newblock {\em Electron. J. Statist.}, 11:891--915.

\bibitem[Kirichenko and van Zanten, 2018]{KirichenkoVanZanten2018}
Kirichenko, A. and van Zanten, H. (2018).
\newblock Minimax lower bounds for function estimation on graphs.
\newblock {\em Electron. J. Statist.}, 12:651--666.

\bibitem[Liang et~al., 2024]{LiangWanDong2024}
Liang, H., Wan, X., and Dong, X. (2024).
\newblock Bayesian optimization of functions over node subsets in graphs.
\newblock In Globerson, A., Mackey, L., Belgrave, D., Fan, A., Paquet, U.,
  Tomczak, J., and Zhang, C., editors, {\em Adv. Neural Inf. Process. Syst.},
  volume~37, pages 38199--38234. Curran Associates, Inc.

\bibitem[Massart, 2007]{Massart2007}
Massart, P. (2007).
\newblock {\em Concentration Inequalities and Model Selection}.
\newblock Springer, Berlin.

\bibitem[{National Cancer Institute}, 2025]{CellMinerAlmanac}
{National Cancer Institute} (2025).
\newblock {CellMiner}: Drug almanac combo scores.
\newblock
  \url{https://discover.nci.nih.gov/cellminer/html/drug_almanac_combo_score.html}.
\newblock DTP Almanac Combo Score download, CellMiner database version 2.15,
  dated 17 September 2025.

\bibitem[O'Donnell, 2014]{ODonnell2014}
O'Donnell, R. (2014).
\newblock {\em Analysis of {B}oolean Functions}.
\newblock Cambridge University Press, Cambridge.

\bibitem[Oh et~al., 2019]{OhTomczakGavvesWelling2019}
Oh, C., Tomczak, J.~M., Gavves, E., and Welling, M. (2019).
\newblock Combinatorial {B}ayesian optimization using the graph {C}artesian
  product.
\newblock In Wallach, H., Larochelle, H., Beygelzimer, A., {d'Alch{\'e}-Buc},
  F., Fox, E., and Garnett, R., editors, {\em Adv. Neural Inf. Process. Syst.},
  volume~32, pages 2914--2924. Curran Associates, Inc.

\bibitem[Puy et~al., 2018]{PuyTremblayGribonvalVandergheynst2018}
Puy, G., Tremblay, N., Gribonval, R., and Vandergheynst, P. (2018).
\newblock Random sampling of bandlimited signals on graphs.
\newblock {\em Appl. Comput. Harmon. Anal.}, 44:446--475.

\bibitem[Qian et~al., 2017]{QianEtAl2017}
Qian, C., Shi, J.-C., Yu, Y., Tang, K., and Zhou, Z.-H. (2017).
\newblock Subset selection under noise.
\newblock In Guyon, I., von Luxburg, U., Bengio, S., Wallach, H., Fergus, R.,
  Vishwanathan, S., and Garnett, R., editors, {\em Adv. Neural Inf. Process.
  Syst.}, volume~30, pages 3563--3573. Curran Associates, Inc.

\bibitem[Shi et~al., 2024]{ShiBalasubramanianPolonik2024}
Shi, Z., Balasubramanian, K., and Polonik, W. (2024).
\newblock Adaptive and non-adaptive minimax rates for weighted
  {L}aplacian-eigenmap based nonparametric regression.
\newblock In Dasgupta, S., Mandt, S., and Li, Y., editors, {\em Proc.\ 27th
  AISTATS}, volume 238, pages 2800--2808. PMLR.

\bibitem[Stobbe and Krause, 2012]{StobbeKrause2012}
Stobbe, P. and Krause, A. (2012).
\newblock Learning {F}ourier sparse set functions.
\newblock In Lawrence, N.~D. and Girolami, M., editors, {\em Proc.\ 15th
  AISTATS}, volume~22, pages 1125--1133. PMLR.

\bibitem[Tajdini et~al., 2024]{TajdiniJainJamieson2024}
Tajdini, A., Jain, L., and Jamieson, K. (2024).
\newblock Nearly minimax optimal submodular maximization with bandit feedback.
\newblock In Globerson, A., Mackey, L., Belgrave, D., Fan, A., Paquet, U.,
  Tomczak, J., and Zhang, C., editors, {\em Adv. Neural Inf. Process. Syst.},
  volume~37, pages 96254--96281. Curran Associates, Inc.

\bibitem[Tropp, 2012]{Tropp2012}
Tropp, J.~A. (2012).
\newblock User-friendly tail bounds for sums of random matrices.
\newblock {\em Found. Comput. Math.}, 12:389--434.

\bibitem[Tsybakov, 2009]{Tsybakov2009}
Tsybakov, A.~B. (2009).
\newblock {\em Introduction to Nonparametric Estimation}.
\newblock Springer, New York.

\bibitem[Wagstaff et~al., 2019]{WagstaffEtAl2019}
Wagstaff, E., Fuchs, F., Engelcke, M., Posner, I., and Osborne, M.~A. (2019).
\newblock On the limitations of representing functions on sets.
\newblock In Chaudhuri, K. and Salakhutdinov, R., editors, {\em Proc.\ 36th
  ICML}, volume~97, pages 6487--6494. PMLR.

\bibitem[Zaheer et~al., 2017]{ZaheerEtAl2017}
Zaheer, M., Kottur, S., Ravanbakhsh, S., P{\'o}czos, B., Salakhutdinov, R., and
  Smola, A.~J. (2017).
\newblock Deep sets.
\newblock In Guyon, I., von Luxburg, U., Bengio, S., Wallach, H., Fergus, R.,
  Vishwanathan, S., and Garnett, R., editors, {\em Adv. Neural Inf. Process.
  Syst.}, volume~30, pages 3394--3404. Curran Associates, Inc.

\end{thebibliography}

\clearpage
\appendix

\section*{Supplementary Material}

This supplementary material contains numerical experiments, the
NCI-ALMANAC illustration, the canonical frame construction, and proofs
of all results stated in the main article.

\section{Numerical experiments}\label{sec:experiments}

We use \((d,k)=(12,4)\), for which \(N=495\), and independent
Gaussian noise with standard deviation \(0.5\).  Prediction error is
measured by the squared \(L^2(\nu_{12,4})\)-distance and is evaluated
exactly over all 495 subsets.

For each spectral profile, we generate 120 independent centered target
functions
\(
  f^{(1)},\ldots,f^{(120)}
\)
directly from the Johnson decomposition and normalize them so that
\(\|f^{(i)}\|_2=1\), for \(1\leqslant i\leqslant120\). Computationally, an \(L^2(\nu_{12,4})\)-orthonormal basis of each
\(V_r\) is obtained by projecting the order-\(r\) inclusion
indicators onto \(U_{r-1}^{\perp}\) and applying a QR
orthonormalization.  Within each level, the coefficient vector is
drawn from a standard Gaussian distribution and normalized to unit
Euclidean norm before being rescaled to the prescribed squared
\(L^2\)-mass.  The resulting direction is uniform on the unit sphere
of \(V_r\), so its distribution does not depend on the particular
orthonormal basis produced by the QR decomposition.

The \emph{pairwise} profile allocates 20\% of
the squared \(L^2\)-norm to \(V_1\) and 80\% to \(V_2\).  The
\emph{diffuse} profile allocates 50\%, 30\%, 15\%, and 5\% to
 \(V_1,V_2,V_3\), and \(V_4\), respectively. 

Each target \(f^{(i)}\) is evaluated at every sample size
\(
  n\in\{40,80,160,320,640\}.
\)
For each pair \((f^{(i)},n)\), we independently draw a new design and
a new noise sample.  Thus, the target functions are shared across
sample sizes, whereas the design and noise are always resampled.  At each fixed sample size, 
the reported Monte Carlo mean is computed
from 120 losses, each corresponding to a different target and an
independently generated design and noise sample.

The numerical study has two parts.  We first compare predictive
performance under the pairwise and diffuse target profiles.  We then use 
separate diagnostics to examine the random-design
phenomena identified by the theory: the signal-dependent variation of
empirical projection, the stability of least squares, and the loss of
information caused by incomplete coverage.

The prediction comparison involves five procedures.  The first is empirical harmonic projection, implemented as
\(\overline Y+\widehat f_D^{\,\mathrm{proj}}\): the sample mean
estimates the constant component, while
\(\widehat f_D^{\,\mathrm{proj}}\) estimates the nonconstant levels
through \(D\). Although the simulated targets are centered, the constant component
is included to reflect the practical case of an unknown intercept. The
second is Johnson spectral shrinkage, which leaves the constant level
unchanged and multiplies the empirical component in \(V_r\),
for every \(1\leqslant r\leqslant m=4\), by 
\[
  (1+\lambda\gamma_{r,d})^{-1}.
\]
The third is a Nadaraya--Watson estimator based on the Johnson
distance, which we call the Johnson kernel estimator.  Its
weights are proportional to
\[
  \exp\left\{
    -\operatorname{dist}_J(S,S')/h
  \right\}.
\]
The remaining two procedures are additive least squares over
\(V_0\mathbin{\oplus^\perp}V_1=U_1\) and the sample mean, which serve
as low-dimensional baselines.

The cutoff \(D\), shrinkage parameter \(\lambda\), and kernel
bandwidth \(h\) are selected using the same randomly chosen validation
quarter within each repetition.  The candidate sets are
\[
  D\in\{0,\ldots,4\},\qquad
  \lambda\in
  \left\{
    10^{-2.5+4.5j/15}:j=0,\ldots,15
  \right\},
\]
and
\[
  h\in\{0.20,0.35,0.55,0.85,1.30,2.00\}.
\]
After selection, the corresponding
estimator is refitted on the full sample. Final performance is evaluated against the known target function,
using the exact \(L^2(\nu_{12,4})\)-error over the complete domain
rather than the validation loss. Additive least squares and
the sample mean have no tuning parameter and are computed directly
from the full sample. This validation scheme is used only for the numerical comparison and
is not analyzed theoretically in the article. Stable least squares is not included in this first comparison because the Gram-stability event need not hold at these sample sizes; it is studied separately in the third panel in a regime where the
empirical Gram matrix is stable. Centered completion is likewise not used as a general prediction competitor: it is a full-level estimator designed specifically to address incomplete coverage, whose extent is displayed in the final panel.

Figure~\ref{fig:compact-experiments} summarizes the two prediction
experiments and the three random-design diagnostics described below.

\begin{figure}[!ht]
  \centering
  \includegraphics[width=\textwidth]{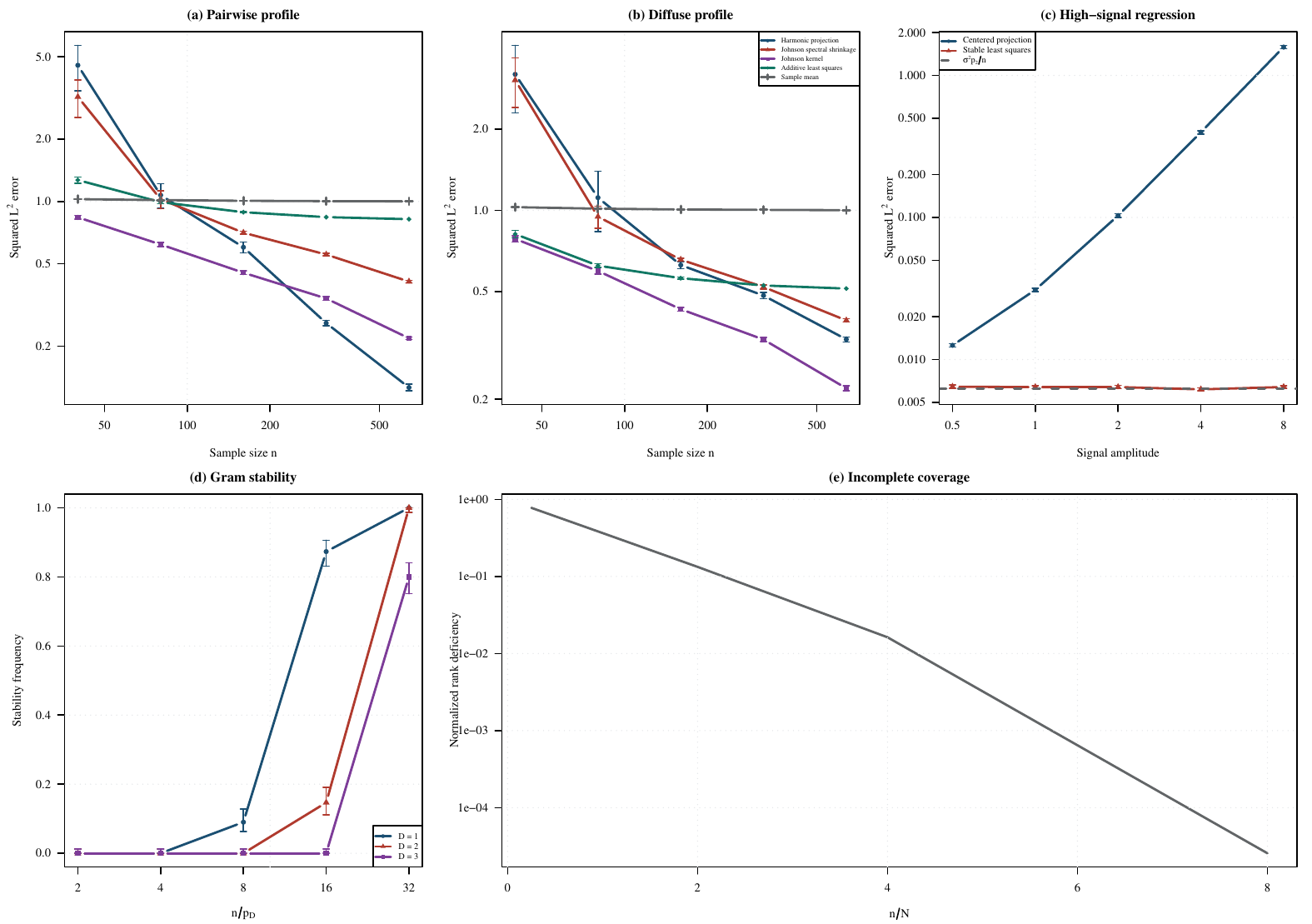}
  \caption{\textbf{Prediction and random-design mechanisms on
  \(\mathcal S_{12,4}\).}
 The first two panels show squared \(L^2(\nu_{12,4})\)-error for the
pairwise and diffuse profiles. Points are Monte Carlo means over 120 joint target-and-data
realizations.  Where visible,
error bars are normal-approximation 95\% confidence intervals for
these means, computed as the mean plus or minus 1.96 Monte Carlo
standard errors. The third panel compares centered empirical projection and stable least squares at \(D=2\) and \(n=2600\) as a fixed signal is rescaled; the dashed
line is \(\sigma^2p_2/n\), and results are based on 200 independent
repetitions at each amplitude. The fourth panel reports the frequency of
\(\|\widehat G_D-I_{p_D}\|_{\mathrm{op}}\leqslant1/2\) for
\(D=1,2,3\), based on 300 independent Gram matrices at each
configuration; bars are 95\% Wilson confidence intervals for the
corresponding stability probabilities. The fifth panel plots the exact normalized full-level rank deficiency
\(\rho_{m,n}\) from
\eqref{eq:full-rank-deficiency-explicit} as a function of \(n/N\).}
  \label{fig:compact-experiments}
\end{figure}

The pairwise experiment shows when the larger interaction space
becomes estimable.  At \(n=320\), harmonic projection has mean error
0.258, compared with 0.340 for the Johnson kernel, 0.554 for spectral
shrinkage, 0.840 for additive least squares, and 1.004 for the sample
mean.  At \(n=640\), the corresponding errors are 0.126, 0.218,
0.410, 0.821, and 1.001. Since \(p_2=65\), fitting the second harmonic level is costly at the
smallest sample sizes. The validation-selected projection estimator
is then highly variable and occasionally selects an excessively large
cutoff.  At \(n=320\) and \(n=640\), however, the validation rule
selects \(D=2\) in every repetition, and estimating the pairwise
component becomes beneficial.

The diffuse profile favors gradual regularization.  The Johnson kernel
has the smallest mean error throughout the displayed range, reaching
0.220 at \(n=640\), compared with 0.333 for validation-selected hard
projection. This is consistent with the allocation of squared \(L^2\)-norm across
several harmonic levels. Hard projection must select a single cutoff: it retains all components
below that cutoff without shrinkage and discards all components above
it.  By contrast, the Johnson kernel averages observations across
nearby subsets without imposing a hard interaction cutoff. The comparison illustrates different estimation mechanisms
without implying that one estimator dominates uniformly. 

The third panel compares the centered empirical projection
\(\widehat f_2^{\,\mathrm{proj}}\) and stable least squares over
\(W_2\) as the signal strength increases
relative to the observation noise. Unlike the preceding prediction experiments, this experiment fixes
one of the pairwise targets described above, denoted by
\(f_0\in W_2\), and considers 
\(
  f_t=tf_0\),  \(t\in\{0.5,1,2,4,8\}.
\)
Since \(\|f_0\|_2=1\), the parameter \(t\) is the \(L^2\)-amplitude
of the signal.  For any fixed \(s>0\), the smallest
Johnson--Sobolev budget for which
\(f_t\in\mathcal E^s_{d,k}(B)\) is
\[
  B_t
  =
  I_{J,d,k}^{(2,s)}(f_t)
  =
  t^2I_{J,d,k}^{(2,s)}(f_0).
\]
Thus, with the noise variance held fixed, increasing \(t\) increases
the ratio \(B_t/\sigma^2\).

We take \(D=2\) and \(n=2600=40p_2\). Since \(f_t\in W_2\), no harmonic component is omitted and neither
estimator incurs approximation error. For each value of
\(t\), we generate 200 independent design and noise samples;
the target direction \(f_0\) itself remains unchanged.  Each design
produces a new empirical Gram matrix \(\widehat G_2\), and we check
whether
\[
  \mathcal A_2
  =
  \left\{
    \|\widehat G_2-I_{p_2}\|_{\mathrm{op}}
    \leqslant\frac12
  \right\}
\]
occurs.  This stability event occurs in all 200 repetitions at every
displayed value of \(t\).  Consequently, stable least squares coincides with ordinary least
squares over \(W_2\) in every repetition. On this event, empirical and
population squared norms are comparable throughout \(W_2\), so the
least-squares fit is well conditioned.

As \(t\) increases from \(0.5\) to \(8\), the least-squares
risk remains between \(0.00617\) and \(0.00644\), close to the
observation-noise benchmark
\[
  \frac{\sigma^2p_2}{n}=0.00625.
\]
By contrast, the projection risk increases from \(0.0126\) to
\(1.582\).  This illustrates the distinction established in
Theorem~\ref{thm:stable-ls}: once the Gram matrix is stable, least
squares removes the random-design fluctuation of the component already
contained in \(W_2\), whereas empirical projection continues to incur
a variance that increases with the squared amplitude of the signal.

The fourth panel examines the stability event itself.  For each
\(D\in\{1,2,3\}\) and each displayed value of \(n/p_D\), we generate
300 independent designs and estimate
\[
  \P_{\nu_{d,k}^{\otimes n}}(\mathcal A_D),
  \qquad
  \mathcal A_D
  =
  \left\{
    \|\widehat G_D-I_{p_D}\|_{\mathrm{op}}
    \leqslant\frac12
  \right\}.
\]
At \(n/p_D=16\), the empirical frequencies are \(0.873\), \(0.147\),
and \(0\) for \(D=1,2,3\), respectively.  At \(n/p_D=32\), they are
\(1\), \(1\), and \(0.80\).  The transition therefore depends not
only on the number of observations per fitted coefficient but also on
the dimension \(p_D\).  This is consistent with
Proposition~\ref{prop:johnson-gram}, whose sufficient stability
condition is of order \(p_D\log p_D\).

The final panel illustrates the incomplete-coverage mechanism using
the exact expression \eqref{eq:full-rank-deficiency-explicit}.  The
normalized full-level rank deficiency decreases from \(0.366\) at
\(n=N\) to \(0.133\) at \(n=2N\), \(0.0163\) at \(n=4N\), and
\(2.56\times10^{-5}\) at \(n=8N\).  These values quantify the normalized expected dimension of the entire
unresolved centered subspace.

All experiments concern the same fixed-cardinality domain
\(\mathcal S_{12,4}\).  The prediction experiments use Gaussian noise
and targets generated from the harmonic decomposition underlying the
theory, whereas the Gram experiment and the exact rank-deficiency
illustration concern the random design alone. These experiments illustrate the mechanisms
identified by the theory rather than establish broad empirical
superiority.

\section{NCI-ALMANAC drug-pair illustration}
\label{supp:almanac}

The NCI-ALMANAC experiment evaluates pairs of anticancer drugs on
the NCI-60 panel of cancer cell lines \citep{HolbeckEtAl2017}.
Each cell line is a population of cancer cells grown in the laboratory,
and we treat its drug-pair scores as a separate regression problem.
We use the CellMiner \emph{DTP Almanac Combo Score} table, database
version 2.15, dated 17 September 2025 \citep{CellMinerAlmanac}.
Thus an input is a two-drug subset, with \(k=2\), and the 104
distinct drugs are the items. Vemurafenib appears under two NSC
identifiers; we merge these identifiers and discard the row pairing
them with each other. For each cell line and unordered drug pair, we
average the available numeric scores across repeated rows before
drawing any split. This gives 5,232 distinct pairs. We omit missing
scores separately for each cell line and exclude the entirely missing
\texttt{ME:MDA-N} column, leaving 59 regression problems. Cell lines
were not selected according to predictive performance.

For each cell line, we draw without replacement
\(n\in\{250,500,1{,}000,2{,}000,3{,}000\}\) observed pairs for training;
the remaining observed pairs form the test set. We repeat this split
with five fixed random seeds at each \(n\). The benchmarks are the
training-set mean and a full additive fit with an intercept and one
indicator per drug. With unstandardized membership indicators, the fit minimizes
the sum of squared residuals plus the sum of squared drug coefficients.
The intercept is unpenalized. The active-item fit uses
the same penalty but includes indicators only for drugs in \(J\), so
its nonconstant part lies in \(W_{J,1}\).

Within each training sample, 60\% of the pairs fit candidate models,
20\% select successive drugs along a forward path capped at 80 drugs,
and the remaining
20\% select the support size from
\[
  q\in\{0,5,10,15,20,30,40,60,80,104\}.
\]
Here the degree is fixed at \(D=1\), so the second validation sample
selects only the support size. The choice \(q=0\) gives the intercept-only fit. At \(q=104\),
we fit the full additive model separately. The nonconstant space
of this model is \(V_1\), of dimension 103. We then refit the selected
model on all \(n\) training pairs. The two internal validation samples
separate construction of the path from selection of its final size,
and the outer test pairs play no role in either choice.

The full order-two Johnson space has dimension
\(\binom{104}{2}-1=5{,}355\), larger than every training sample in this
experiment. At this order, a model can assign a separate effect to
each pair, but the effect of a held-out pair is not directly observed.
We therefore use additive models, which predict the score of a held-out
pair from the effects estimated for its two drugs in other pairs. The full
additive fit and its active-item version both have nonconstant parts
in the order-one Johnson space.

Figure~\ref{fig:almanac-real} reports test mean squared error divided
by the test error of the training-set mean. Ratios are averaged
equally over the 59 cell lines and five splits per cell line. The
bars are descriptive 95\% \(t\)-intervals across cell-line means;
they do not imply independent biological samples.
\begin{figure}[!ht]
  \centering
  \includegraphics[width=\textwidth]{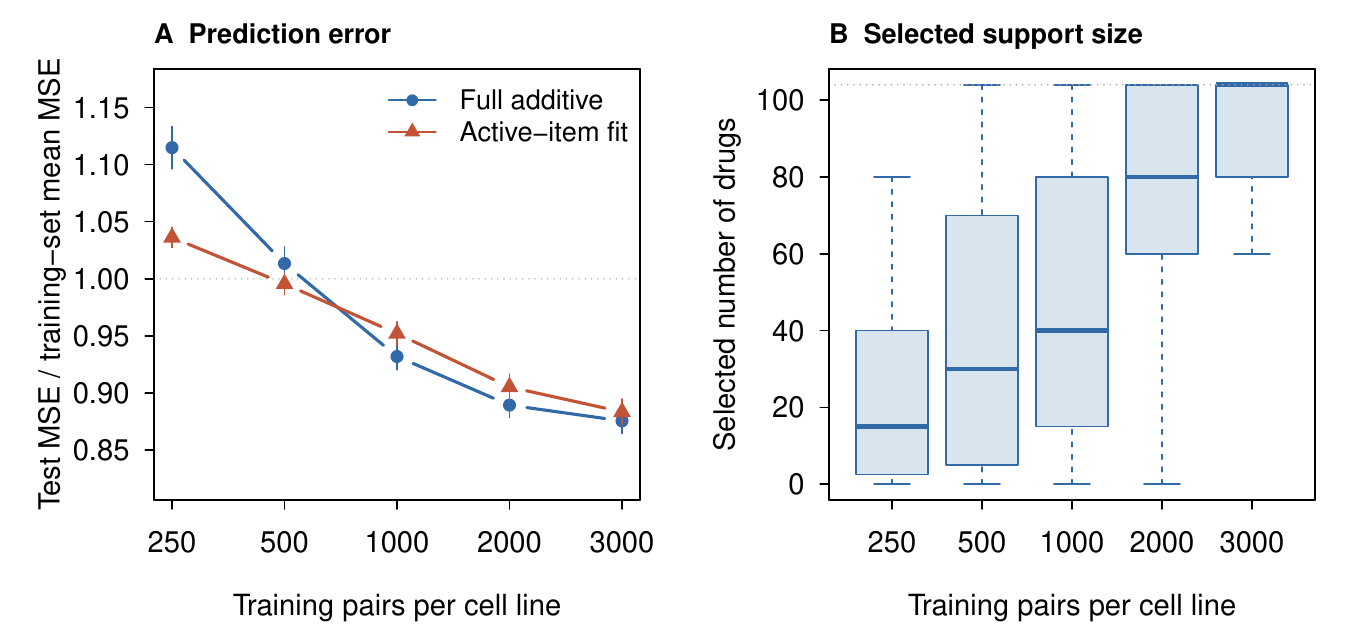}
  \caption{\textbf{Held-out drug-pair prediction in NCI-ALMANAC.}
  Panel A shows relative test mean squared error for the full
  additive and validation-selected active-item fits. Its error bars are descriptive 95\% \(t\)-intervals across cell-line means.  The horizontal
  reference at one is the training-set mean. Panel B shows the
  selected number of drugs. Each box in Panel B summarizes the 295 selected support
  sizes from 59 cell lines and five splits per line at a fixed
  training size.}
  \label{fig:almanac-real}
\end{figure}

At \(n=250\), relative error is \(1.115\) for the full additive
model and \(1.036\) for the active-item fit; both have higher test error than the predictor based on the training-set mean. At \(n=1{,}000\), the corresponding values are \(0.932\) and \(0.952\), and at \(n=3{,}000\) they are \(0.876\)
and \(0.883\). The median selected support grows from 15 drugs at
\(n=250\) to all 104 at \(n=3{,}000\). Thus active-item restriction
helps relative to fitting every drug effect when data are scarce,
but these scores do not support a uniformly small active set.  The example illustrates how the relative performance of sparse and
full additive fits changes with the training size. 

\section{Johnson harmonic analysis and the canonical frame}
\label{supp:harmonic}

\subsection{Proof of Proposition~\ref{prop:johnson-spectrum}}
\label{proof:prop:johnson-spectrum}

At level \(r=0\), the space \(V_0\) consists of the constant
functions.  Hence, \(K\one=\one\), which proves
\eqref{eq:johnson-eigenvalue} because \(\theta_{0,d,k}=1\).

Fix \(1\leqslant r\leqslant m\), let \(T\subseteq[d]\) have
cardinality \(r\), and recall that
\[
  h_T(S)=\one_{\{T\subseteq S\}}.
\]
We first compute \((Kh_T)(S)\) by counting the \(k(d-k)\) possible
ordered swaps from a fixed \(S\in\Sdk\).

If \(T\subseteq S\), the inclusion remains valid after the swap
exactly when the outgoing item lies in \(S\setminus T\).  There are
\(k-r\) possible outgoing items and \(d-k\) possible incoming items,
giving \((k-r)(d-k)\) successful swaps.  If
\(|T\setminus S|=1\), the incoming item must be the unique missing
element of \(T\), while the outgoing item can be any of the
\(k-r+1\) elements of \(S\setminus T\).  If at least two elements of
\(T\) are missing, one exchange cannot produce a subset containing
\(T\).  Dividing these counts by \(k(d-k)\) gives
\begin{equation}\label{supp:eq:KhT-first}
  (Kh_T)(S)
  =
  \left(1-\frac rk\right)h_T(S)
  +
  \frac{k-r+1}{k(d-k)}
  \one_{\{|T\setminus S|=1\}}.
\end{equation}

The remaining indicator satisfies
\begin{equation}\label{supp:eq:missing-one}
  \one_{\{|T\setminus S|=1\}}
  =
  \sum_{\substack{R\subset T\\|R|=r-1}}h_R(S)
  -
  rh_T(S).
\end{equation}
Indeed, the sum counts the \((r-1)\)-subsets of \(T\) that are
contained in \(S\).  If \(T\subseteq S\), all \(r\) such subsets are
counted.  If exactly one element of \(T\) is missing from \(S\), only
the subset obtained by removing that element is counted.  If at least
two elements are missing, deleting a single element from \(T\) cannot
produce a subset contained in \(S\).  The sum is therefore equal to
\(r\), \(1\), or \(0\), respectively, and subtracting \(rh_T(S)\)
proves \eqref{supp:eq:missing-one}.

Substituting \eqref{supp:eq:missing-one} into
\eqref{supp:eq:KhT-first} gives the functional identity
\[
  (K-\theta_{r,d,k}I)h_T
  =
  \frac{k-r+1}{k(d-k)}
  \sum_{\substack{R\subset T\\|R|=r-1}}h_R
  \in U_{r-1},
\]
because
\[
  1-\frac rk
  -\frac{r(k-r+1)}{k(d-k)}
  =
  1-\frac{r(d-r+1)}{k(d-k)}
  =
  \theta_{r,d,k}.
\]
Since the functions \(h_T\), with \(|T|=r\), span \(U_r\), it follows
that
\begin{equation}\label{supp:eq:triangular-space}
  (K-\theta_{r,d,k}I)U_r
  \subseteq U_{r-1}.
\end{equation}

We also need the invariance of \(U_{r-1}\) under \(K\).  If
\(r\geqslant2\), applying the preceding calculation at order \(r-1\)
gives
\[
  (K-\theta_{r-1,d,k}I)h_C\in U_{r-2}
  \qquad
  \text{for every }C\subseteq[d]\text{ with }|C|=r-1.
\]
Since \(h_C\in U_{r-1}\) and \(U_{r-2}\subseteq U_{r-1}\), this
implies \(Kh_C\in U_{r-1}\).  These indicators span \(U_{r-1}\), so
\(K\) preserves \(U_{r-1}\).  When \(r=1\), the same conclusion
follows from \(U_0=\operatorname{span}\{\one\}\) and
\(K\one=\one\).

Because \(K\) is self-adjoint, the invariance of \(U_{r-1}\) implies
that \(U_{r-1}^{\perp}\) is also invariant.  Indeed, if
\(f\perp U_{r-1}\) and \(g\in U_{r-1}\), then
\[
  \inner{Kf}{g}
  =
  \inner{f}{Kg}
  =
  0,
\]
because \(Kg\in U_{r-1}\).

Now let \(f\in V_r=U_r\cap U_{r-1}^{\perp}\).  Relation
\eqref{supp:eq:triangular-space} gives
\[
  (K-\theta_{r,d,k}I)f\in U_{r-1}.
\]
On the other hand, both \(Kf\) and \(f\) belong to
\(U_{r-1}^{\perp}\), so the same difference belongs to
\(U_{r-1}^{\perp}\).  Since
\(U_{r-1}\cap U_{r-1}^{\perp}=\{0\}\), we conclude that
\[
  (K-\theta_{r,d,k}I)f=0,
\]
which proves \eqref{eq:johnson-eigenvalue}.

It remains to verify the Dirichlet identity.  Suppose that
\(S_0\sim\nu_{d,k}\) and, conditionally on \(S_0=S\), let \(S_1\)
have transition probabilities \(q(S,\cdot)\).  We first check that
\(S_1\) is also uniform on \(\Sdk\).  For every \(S'\in\Sdk\),
detailed balance gives
\[
\begin{aligned}
  \mathbb P(S_1=S')
  &=
  \sum_{S\in\Sdk}\nu_{d,k}(S)q(S,S')\\
  &=
  \sum_{S\in\Sdk}\nu_{d,k}(S')q(S',S)
  =
  \nu_{d,k}(S'),
\end{aligned}
\]
because \(\sum_{S\in\Sdk}q(S',S)=1\).  Thus,
\(\nu_{d,k}\) is stationary for the one-exchange walk.

For a fixed \(S\in\Sdk\), the conditional distribution of \(S_1\)
gives
\[
  \E\!\left[f(S_1)^2\mid S_0=S\right]
  =
  \frac1{k(d-k)}
  \sum_{a\in S}\sum_{b\notin S}f(S^{a\to b})^2.
\]
Averaging this identity with respect to
\(S_0\sim\nu_{d,k}\), and using \(S_1\sim\nu_{d,k}\), gives
\begin{equation}\label{supp:eq:reverse-square}
  \frac1{k(d-k)}
  \E_{S\sim\nu_{d,k}}
  \left[
    \sum_{a\in S}\sum_{b\notin S}f(S^{a\to b})^2
  \right]
  =
  \E_{S\sim\nu_{d,k}}[f(S)^2]
  =
  \norm{f}_2^2.
\end{equation}

Similarly, conditioning on \(S_0\) gives
\begin{align*}
\frac1{k(d-k)}
  \E_{S\sim\nu_{d,k}}
  \left[
    \sum_{a\in S}\sum_{b\notin S}
    f(S)f(S^{a\to b})
  \right]
&=
  \E_{S\sim\nu_{d,k}}
  \left[
    f(S)\,
    \E\!\left[f(S_1)\mid S_0=S\right]
  \right]
\\
&=
  \E_{S\sim\nu_{d,k}}[f(S)Kf(S)]
  =
  \inner{f}{Kf}.
\end{align*}
Expanding the squared differences, using
\eqref{supp:eq:reverse-square} for the second quadratic term and the
preceding identity for the cross term, yields
\begin{align*}
 \frac1{2k(d-k)}
  \E_{S\sim\nu_{d,k}}
  \left[
    \sum_{a\in S}\sum_{b\notin S}
    \{f(S)-f(S^{a\to b})\}^2
  \right]
 &=
  \E_{S\sim\nu_{d,k}}[f(S)^2]
  -
  \E_{S\sim\nu_{d,k}}[f(S)Kf(S)]
\\
 &=
  \inner{f}{(I-K)f},
\end{align*}
which proves \eqref{eq:johnson-dirichlet-form}.

\subsection{A canonical equivariant frame}\label{ssec:canonical-frame}

The decomposition into the spaces \(V_r\) is canonical, whereas an
orthonormal basis within a given level is not. Recall that, in a finite-dimensional Hilbert space \(\mathcal H\), a frame is a finite spanning family
\(\{\psi_a:a\in\mathcal I\}\). Its frame operator is
\[
  \mathcal Fg
  =
  \sum_{a\in\mathcal I}
  \langle g,\psi_a\rangle\psi_a,
  \qquad g\in\mathcal H.
\]
The frame is tight if there exists \(C>0\) such that
\(
  \mathcal F=CI_{\mathcal H}.
\)
In that case,
\[
  g
  =
  C^{-1}
  \sum_{a\in\mathcal I}
  \langle g,\psi_a\rangle\psi_a.
\]
Taking the inner product of the reconstruction identity with \(g\)
gives
\[
  \|g\|_{\mathcal H}^2
  =
  C^{-1}
  \sum_{a\in\mathcal I}
  \langle g,\psi_a\rangle_{\mathcal H}^2,
\]
and therefore
\[
  \sum_{a\in\mathcal I}
  \langle g,\psi_a\rangle_{\mathcal H}^2
  =
  C\|g\|_{\mathcal H}^2.
\]
Thus, a tight frame may be redundant while retaining reconstruction
and energy identities analogous to those of an orthonormal basis.  We
now construct such a frame in a way that respects the permutation
symmetry of the fixed-cardinality domain.

Fix \(1\leqslant r\leqslant m\) and
\(T\subseteq[d]\) with \(|T|=r\).  For \(S\in\Sdk\), let
\[
  L_T(S)=|S\cap T|
\]
be the number of items shared by \(S\) and \(T\).  If
\(S\sim\nu_{d,k}\), then \(L_T(S)\) has the hypergeometric
distribution
\begin{equation}\label{eq:hypergeometric-weight}
  \P_{S\sim\nu_{d,k}}\{L_T(S)=\ell\}
  =
  w_{r,d,k}(\ell)
  :=
  \frac{\binom r\ell\binom{d-r}{k-\ell}}{\binom dk},
  \qquad
  0\leqslant\ell\leqslant r.
\end{equation}
Indeed, among the \(\binom dk\) possible \(k\)-subsets, those having
overlap \(\ell\) with \(T\) are obtained by choosing \(\ell\) items
from \(T\) and \(k-\ell\) items from its complement.  Moreover,
\(r\leqslant m=\min(k,d-k)\) ensures that every overlap
\(\ell\in\{0,\ldots,r\}\) is attainable.  Hence,
\[
  w_{r,d,k}(\ell)>0,
  \qquad 0\leqslant\ell\leqslant r.
\]

Our aim is to transform the overlap \(L_T(S)=|S\cap T|\) into a
normalized function that isolates the pure harmonic level \(V_r\).
The natural construction uses orthogonal polynomials for the
hypergeometric distribution of \(L_T(S)\).

Let \(\mathcal P_r\) be the space of polynomials of degree at most
\(r\).  On this space, define
\[
  \langle p,q\rangle_{w}
  =
  \sum_{\ell=0}^r
  w_{r,d,k}(\ell)p(\ell)q(\ell),
  \qquad
  p,q\in\mathcal P_r.
\]
This is an inner product because all the weights are positive and no
nonzero polynomial in \(\mathcal P_r\) can vanish at all \(r+1\)
distinct points \(0,\ldots,r\). Within the \((r+1)\)-dimensional space \(\mathcal P_r\), the subspace
\(\mathcal P_{r-1}\) has dimension \(r\) and hence codimension one.
Its orthogonal complement is therefore one-dimensional.  Every
nonzero element of this complement has degree exactly \(r\), and any
two such elements are proportional.

Subsection~\ref{supp:explicit} constructs an explicit nonzero generator
of this one-dimensional space from the coefficients of the \(r\)-th
finite-difference operator.  It also shows that the value of this
generator at \(r\) is nonzero.  After normalizing with respect to
\(\langle\cdot,\cdot\rangle_w\) and choosing its sign, we obtain the
unique polynomial \(H_{r,d,k}\) characterized by
\[
  \langle H_{r,d,k},p\rangle_w=0
  \quad\text{for every }p\in\mathcal P_{r-1},
  \qquad
  \langle H_{r,d,k},H_{r,d,k}\rangle_w=1,
  \qquad
  H_{r,d,k}(r)>0.
\]
We then lift this polynomial from the overlap variable to the full
domain by defining
\begin{equation}\label{eq:canonical-frame}
  \phi_{r,T}(S)
  =
  H_{r,d,k}(|S\cap T|),
  \qquad
  S\in\Sdk.
\end{equation}
At level \(r=0\), set \(\phi_{0,\emptyset}\equiv1\).

The proof of Theorem~\ref{thm:canonical-frame}, given in
Subsection~\ref{supp:proof-canonical-frame}, also establishes that
\[
  \phi_{r,T}
  =
  \frac{P_rh_T}{\norm{P_rh_T}_2}.
\]
Thus, \(\phi_{r,T}\) is precisely the normalized order-\(r\) harmonic
component of the inclusion indicator \(h_T\).

\begin{theorem}[Canonical tight frame]\label{thm:canonical-frame}
For every \(0\leqslant r\leqslant m\), the family
\[
  \left\{
    \phi_{r,T}:T\subseteq[d],\ |T|=r
  \right\}
\]
consists of unit-norm functions in \(V_r\) and spans \(V_r\).  With
\begin{equation}\label{eq:frame-factor}
  A_{r,d,k}
  =
  \frac{\dim(V_r)}{\binom dr}
  =
  1-\frac{r}{d-r+1},
\end{equation}
the frame satisfies, for every
\(f\in L^2(\Sdk,\nu_{d,k})\),
\begin{align}
  P_rf
  &=
  A_{r,d,k}
  \sum_{\substack{T\subseteq[d]\\|T|=r}}
  \inner{f}{\phi_{r,T}}\phi_{r,T},
  \label{eq:frame-reconstruction}\\
  \norm{P_rf}_2^2
  &=
  A_{r,d,k}
  \sum_{\substack{T\subseteq[d]\\|T|=r}}
  \inner{f}{\phi_{r,T}}^2.
  \label{eq:frame-parseval}
\end{align}
The construction is permutation-equivariant: for every permutation
\(\pi\) of \([d]\),
\[
  \phi_{r,\pi(T)}(\pi(S))
  =
  \phi_{r,T}(S),
  \qquad
  S\in\Sdk.
\]
\end{theorem}

The frame uses \(\binom dr\) vectors to represent the space \(V_r\),
whose dimension is
\(\binom dr-\binom d{r-1}\), with the convention
\(\binom d{-1}=0\). Its exact redundancy, defined as the
number of frame vectors divided by the dimension of their span, is
therefore \(A_{r,d,k}^{-1}\).  This representation will be used to
write the projection estimator in an explicit form that is invariant
under relabeling of the ground set.  The statistical risk bounds
depend only on the harmonic spaces \(V_r\), not on a particular frame
or orthonormal basis.

\subsection{Overlap polynomial}
\label{supp:explicit}

The case \(r=0\), for which
\(\phi_{0,\emptyset}\equiv 1\), is immediate and does not require
an overlap polynomial.  We therefore fix
\(1\leqslant r\leqslant m\). All the hypergeometric weights
\(w_{r,d,k}(\ell)\), \(0\leqslant\ell\leqslant r\), are strictly
positive: the inequalities \(r\leqslant k\) and
\(r\leqslant d-k\) ensure that every overlap
\(\ell\in\{0,\ldots,r\}\) is attainable.

We first recall the finite-difference identity underlying the
construction.  For a function \(p\) on the integers, define its
forward difference by
\[
  (\Delta p)(x)=p(x+1)-p(x).
\]
Iterating this operator \(r\) times gives
\begin{equation}\label{supp:eq:forward-difference}
  (\Delta^rp)(0)
  =
  \sum_{\ell=0}^r
  (-1)^{r-\ell}\binom r\ell p(\ell).
\end{equation}
If \(p\) is a polynomial of degree at most \(r-1\), then each
application of \(\Delta\) lowers its degree by one, and hence
\(
  \Delta^rp\equiv0.
\)
The coefficients used below are the normalized coefficients in
\eqref{supp:eq:forward-difference}.  Define
\begin{equation}\label{supp:eq:lambda-Z}
  \lambda_{r,\ell}
  =
  \frac{(-1)^{r-\ell}}{\ell!(r-\ell)!}
  =
  \frac1{r!}
  (-1)^{r-\ell}\binom r\ell,
  \qquad
  Z_{r,d,k}
  =
  \sum_{\ell=0}^r
  \frac{\lambda_{r,\ell}^2}{w_{r,d,k}(\ell)}.
\end{equation}
Since all the weights are positive,
\(0<Z_{r,d,k}<\infty\).

We now prescribe the values
\begin{equation}\label{supp:eq:H-values}
  H_{r,d,k}(\ell)
  :=
  \frac{\lambda_{r,\ell}}
       {w_{r,d,k}(\ell)\sqrt{Z_{r,d,k}}},
  \qquad
  \ell=0,\ldots,r.
\end{equation}
Because the interpolation nodes \(0,\ldots,r\) are distinct, the
Lagrange interpolation theorem yields a unique polynomial of degree
at most \(r\) taking the values in
\eqref{supp:eq:H-values}.  We use the notation \(H_{r,d,k}\) both
for this polynomial and for its restriction to the interpolation
grid.

\begin{lemma}[Top hypergeometric polynomial]
\label{supp:lem:top-polynomial}
For every \(1\leqslant r\leqslant m\), the interpolating polynomial
\(H_{r,d,k}\) has degree exactly \(r\) and satisfies
\[
  \sum_{\ell=0}^r
  w_{r,d,k}(\ell)H_{r,d,k}(\ell)p(\ell)
  =
  0
\]
for every polynomial \(p\) of degree at most \(r-1\).  Moreover,
\[
  \sum_{\ell=0}^r
  w_{r,d,k}(\ell)H_{r,d,k}(\ell)^2
  =
  1,
  \qquad
  H_{r,d,k}(r)>0.
\]
\end{lemma}

\begin{proof}
Fix \(1\leqslant r\leqslant m\), and let \(p\) be a polynomial of
degree at most \(r-1\). By
\eqref{supp:eq:forward-difference},
\[
  \sum_{\ell=0}^r
  \lambda_{r,\ell}p(\ell)
  =
  \frac1{r!}(\Delta^rp)(0)
  =
  0.
\]
Using \eqref{supp:eq:H-values}, we obtain
\[
  \sum_{\ell=0}^r
  w_{r,d,k}(\ell)H_{r,d,k}(\ell)p(\ell)
  =
  \frac1{\sqrt{Z_{r,d,k}}}
  \sum_{\ell=0}^r\lambda_{r,\ell}p(\ell)
  =0,
\]
which proves the orthogonality assertion.  By \eqref{supp:eq:lambda-Z},
\[
  \sum_{\ell=0}^r
  w_{r,d,k}(\ell)H_{r,d,k}(\ell)^2
  =
  \frac1{Z_{r,d,k}}
  \sum_{\ell=0}^r
  \frac{\lambda_{r,\ell}^2}{w_{r,d,k}(\ell)}
  =1.
\]
If \(H_{r,d,k}\) had degree at most \(r-1\), the orthogonality
assertion could be applied with \(p=H_{r,d,k}\), giving
\[
  \sum_{\ell=0}^r
  w_{r,d,k}(\ell)H_{r,d,k}(\ell)^2=0,
\]
in contradiction with its unit norm.  Hence,
\(\deg(H_{r,d,k})=r\).  Finally,
\(\lambda_{r,r}=1/r!>0\), while
\(w_{r,d,k}(r)>0\) and \(Z_{r,d,k}>0\).  Formula
\eqref{supp:eq:H-values} therefore gives
\(H_{r,d,k}(r)>0\).
\end{proof}

The lemma shows that the interpolating polynomial constructed here
coincides with the polynomial \(H_{r,d,k}\) introduced in
Subsection~\ref{ssec:canonical-frame}.
Indeed, within the \((r+1)\)-dimensional space of polynomials of
degree at most \(r\), the orthogonal complement of the polynomials of
degree at most \(r-1\) is one-dimensional.  Unit normalization and
positivity at \(r\) select a unique representative of this
one-dimensional space.

We now lift this overlap polynomial to functions on \(\Sdk\) and
verify that the resulting functions belong to \(V_r\).

\begin{lemma}\label{supp:lem:frame-harmonicity}
For every \(1\leqslant r\leqslant m\) and every
\(T\subseteq[d]\) with \(|T|=r\), the function
\(\phi_{r,T}\) defined in \eqref{eq:canonical-frame} belongs to
\(V_r\), has unit norm, and satisfies
\[
  \phi_{r,\pi(T)}(\pi(S))
  =
  \phi_{r,T}(S),
  \qquad S\in\Sdk,
\]
for every permutation \(\pi\) of \([d]\).
\end{lemma}

\begin{proof}
The equivariance property follows from
\[
  |\pi(S)\cap\pi(T)|
  =
  |\pi(S\cap T)|
  =
  |S\cap T|.
\]
Indeed,
\[
  \phi_{r,\pi(T)}(\pi(S))
  =
  H_{r,d,k}(|\pi(S)\cap\pi(T)|)
  =
  H_{r,d,k}(|S\cap T|)
  =
  \phi_{r,T}(S).
\]

The norm of \(\phi_{r,T}\) is determined by the distribution of the
overlap \(L_T(S)=|S\cap T|\).  Using
\eqref{eq:hypergeometric-weight} and
Lemma~\ref{supp:lem:top-polynomial}, we obtain
\[
\begin{aligned}
  \|\phi_{r,T}\|_2^2
  &=
  \E_{S\sim\nu_{d,k}}
  \left[
    H_{r,d,k}(L_T(S))^2
  \right]
\\
  &=
  \sum_{\ell=0}^r
  w_{r,d,k}(\ell)H_{r,d,k}(\ell)^2
  =
  1.
\end{aligned}
\]

It remains to verify the two conditions defining
\(
  V_r=U_r\cap U_{r-1}^{\perp}.
\)
We first prove that \(\phi_{r,T}\in U_r\).  For
\(0\leqslant j\leqslant r\),
\begin{equation}\label{supp:eq:binomial-inclusion}
  \binom{|S\cap T|}{j}
  =
  \sum_{\substack{R\subseteq T\\|R|=j}}h_R(S).
\end{equation}
Both sides count the \(j\)-subsets of \(S\cap T\).  Moreover, the
polynomials
\[
  \ell\longmapsto\binom{\ell}{j},
  \qquad j=0,\ldots,r,
\]
form a basis of the space of polynomials of degree at most \(r\):
the \(j\)-th polynomial has degree \(j\) and leading coefficient
\(1/j!\).  Since \(H_{r,d,k}\) has degree \(r\), there exist
coefficients \(c_0,\ldots,c_r\) such that
\[
  H_{r,d,k}(\ell)
  =
  \sum_{j=0}^r c_j\binom{\ell}{j}.
\]
Evaluating this identity at \(\ell=|S\cap T|\) and applying
\eqref{supp:eq:binomial-inclusion} gives
\[
  \phi_{r,T}(S)
  =
  \sum_{j=0}^r c_j\binom{|S\cap T|}{j}
=
  \sum_{j=0}^r c_j
  \sum_{\substack{R\subseteq T\\|R|=j}}h_R(S).
\]
Each function \(h_R\) in the last display belongs to \(U_j\), and
\(U_j\subseteq U_r\) by nesting.  Therefore,
\(\phi_{r,T}\in U_r\).

We next show that \(\phi_{r,T}\perp U_{r-1}\).  By nesting,
\(U_{r-1}\) is spanned by the inclusion indicators \(h_R\) with
\(|R|\leqslant r-1\).  Fix such a set \(R\), and write
\[
  a=|R\cap T|,
  \qquad
  b=|R\setminus T|,
  \qquad
  a+b=|R|.
\]
All probabilities and conditional expectations below are taken under
\(S\sim\nu_{d,k}\).  Conditional on \(L_T(S)=\ell\), the set
\(S\cap T\) is uniformly distributed among the \(\ell\)-subsets of
\(T\), while \(S\setminus T\) is uniformly distributed among the
\((k-\ell)\)-subsets of \([d]\setminus T\).  Hence, the conditional
probability that \(R\cap T\) is contained in \(S\cap T\) is, with the convention that \(\binom uv=0\) whenever \(v<0\) or \(v>u\),
\[
  \frac{\binom{r-a}{\ell-a}}{\binom r\ell}
  =
  \frac{\binom\ell a}{\binom ra},
\]
and the conditional probability that \(R\setminus T\) is contained
in \(S\setminus T\) is
\[
  \frac{\binom{d-r-b}{k-\ell-b}}
       {\binom{d-r}{k-\ell}}
  =
  \frac{\binom{k-\ell}b}{\binom{d-r}b}.
\]
The two selections are conditionally independent, and therefore
\[
  \E
  \left[
    h_R(S)\mid L_T(S)=\ell
  \right]
  =
  \frac{\binom\ell a}{\binom ra}
  \frac{\binom{k-\ell}b}{\binom{d-r}b}.
\]
As a function of \(\ell\), the right-hand side is a polynomial of
degree at most \(a+b=|R|\leqslant r-1\).  Indeed,
\(\binom{\ell}{a}\) and \(\binom{k-\ell}{b}\) are polynomials in
\(\ell\) of degrees \(a\) and \(b\), respectively.  The
orthogonality property of \(H_{r,d,k}\) established in
Lemma~\ref{supp:lem:top-polynomial} now gives
\[
\begin{aligned}
  \inner{\phi_{r,T}}{h_R}
  &=
  \E_{S\sim\nu_{d,k}}
  \left[
    H_{r,d,k}(L_T(S))h_R(S)
  \right]
\\
  &=
  \E_{S\sim\nu_{d,k}}
  \left[
    H_{r,d,k}(L_T(S))
    \E
    \left[
      h_R(S)\mid L_T(S)
    \right]
  \right]
  =0.
\end{aligned}
\]
Thus, \(\phi_{r,T}\perp U_{r-1}\).  Together with
\(\phi_{r,T}\in U_r\), this proves that
\(\phi_{r,T}\in V_r\).
\end{proof}

\subsection{Proof of Theorem~\ref{thm:canonical-frame}}
\label{supp:proof-canonical-frame}

The assertion for \(r=0\) follows from
\(\phi_{0,\emptyset}\equiv1\),
\(V_0=\operatorname{span}\{\one\}\), and
\(A_{0,d,k}=1\).  Fix \(1\leqslant r\leqslant m\).
By Lemma~\ref{supp:lem:frame-harmonicity}, all the functions
\(\phi_{r,T}\), with \(T\subseteq[d]\) and \(|T|=r\), belong to
\(V_r\) and have unit norm.

For \(g\in V_r\), define the frame operator
\[
  \mathcal F_rg
  =
  \sum_{\substack{T\subseteq[d]\\|T|=r}}
  \inner{g}{\phi_{r,T}}\phi_{r,T}.
\]
A finite family is a tight frame for \(V_r\) precisely when its frame
operator is a positive scalar multiple of the identity on \(V_r\).
We now prove that this is the case and determine the scalar.

The kernel of \(\mathcal F_r\), relative to the uniform inner
product, is
\[
  \mathcal Q_r(S,S')
  =
  \sum_{\substack{T\subseteq[d]\\|T|=r}}
  \phi_{r,T}(S)\phi_{r,T}(S'),
\]
in the sense that
\[
  (\mathcal F_rg)(S)
  =
  \frac1N
  \sum_{S'\in\Sdk}
  \mathcal Q_r(S,S')g(S').
\]
The permutation equivariance established in
Lemma~\ref{supp:lem:frame-harmonicity} gives
\[
  \phi_{r,T}(\pi(S))
  =
  \phi_{r,\pi^{-1}(T)}(S).
\]
Reindexing the sum defining \(\mathcal Q_r\) therefore yields
\[
  \mathcal Q_r(\pi(S),\pi(S'))
  =
  \mathcal Q_r(S,S')
\]
for every permutation \(\pi\) of \([d]\).

All the sets considered below belong to \(\Sdk\) and therefore have
the same cardinality \(k\). For two pairs \((S,S')\) and \((R,R')\), there exists a permutation
\(\pi\) of \([d]\) satisfying
\[
  \pi(S)=R
  \qquad\text{and}\qquad
  \pi(S')=R'
\]
if and only if
\[
  |S\cap S'|=|R\cap R'|.
\]
Indeed, because all four sets have cardinality \(k\), equality of the
intersection sizes implies equality of the cardinalities of the four
corresponding parts: the intersection, the two set differences, and
the complement of the union.  Bijections between these four pairs of
parts combine to form the required permutation of \([d]\). Consequently, \(\mathcal Q_r(S,S')\) depends only on
\(\operatorname{dist}_J(S,S')\).  Thus, there exist constants
\(q_{r,0},\ldots,q_{r,m}\) such that
\[
  \mathcal Q_r(S,S')
  =
  \sum_{\ell=0}^m
  q_{r,\ell}
  \one_{\{\operatorname{dist}_J(S,S')=\ell\}}.
\]

For \(0\leqslant\ell\leqslant m\), let \(A_\ell\) denote the
Johnson distance matrix indexed by \(\Sdk\), with entries
\[
  (A_\ell)_{S,S'}
  =
  \one_{\{\operatorname{dist}_J(S,S')=\ell\}}.
\]
We regard \(A_\ell\) as an operator on functions
\(g:\Sdk\to\mathbb R\), acting by
\[
\begin{aligned}
  (A_\ell g)(S)
  &=
  \sum_{S'\in\Sdk}
  (A_\ell)_{S,S'}g(S')
\\
  &=
  \sum_{\substack{S'\in\Sdk\\
  \operatorname{dist}_J(S,S')=\ell}}
  g(S').
\end{aligned}
\]
Thus, \(A_\ell g\) sums the values of \(g\) over all subsets at
Johnson distance \(\ell\) from \(S\).  Using the preceding expression
for \(\mathcal Q_r\), we obtain, for every \(g\in V_r\),
\[
\begin{aligned}
  (\mathcal F_rg)(S)
  &=
  \frac1N
  \sum_{S'\in\Sdk}
  \sum_{\ell=0}^m
  q_{r,\ell}(A_\ell)_{S,S'}g(S')
\\
  &=
  \frac1N
  \sum_{\ell=0}^m
  q_{r,\ell}(A_\ell g)(S).
\end{aligned}
\]
Equivalently, as an operator on \(V_r\),
\[
  \mathcal F_r
  =
  \left.
  \frac1N
  \sum_{\ell=0}^m q_{r,\ell}A_\ell
  \right|_{V_r}.
\]

The Johnson graph is distance-regular, so each distance matrix
\(A_\ell\) is a polynomial in its adjacency matrix \(A_1\)
\citep{BrouwerCohenNeumaier1989}.  Since
\(A_1=k(d-k)K\), Proposition~\ref{prop:johnson-spectrum} then implies
that \(A_\ell\) acts as multiplication by a scalar on every harmonic
space \(V_j\). Since \(\mathcal F_r\) is a
linear combination of these matrices, its restriction to \(V_r\)
also acts as multiplication by a scalar.  Thus, there exists
\(c_r\in\mathbb R\) such that
\[
  \mathcal F_rg=c_rg,
  \qquad g\in V_r.
\]

We determine \(c_r\) by taking the trace on \(V_r\).  The rank-one
operator
\[
  g\longmapsto
  \inner{g}{\phi_{r,T}}\phi_{r,T}
\]
has trace \(\norm{\phi_{r,T}}_2^2=1\).  Since there are
\(\binom dr\) subsets \(T\subseteq[d]\) of cardinality \(r\),
\[
\begin{aligned}
  c_r\dim(V_r)
  &=
  \operatorname{tr}(\mathcal F_r)
\\
  &=
  \sum_{\substack{T\subseteq[d]\\|T|=r}}
  \norm{\phi_{r,T}}_2^2
  =
  \binom dr.
\end{aligned}
\]
Hence,
\[
  c_r
  =
  \frac{\binom dr}{\dim(V_r)}
  =
  A_{r,d,k}^{-1}
  >0.
\]
The second equality uses \eqref{eq:frame-factor}. It follows that
\(
  \mathcal F_r
  =
  A_{r,d,k}^{-1}I_{V_r},
\)
which proves tightness.

The positivity of the scalar also implies that the frame functions
span \(V_r\).  For
every \(v\in V_r\),
\(
  v
  =
  \mathcal F_r(A_{r,d,k}v).
\)
By definition, every value of \(\mathcal F_r\) is a linear
combination of the functions \(\phi_{r,T}\).  Hence, \(V_r\) is
contained in their span.  The reverse inclusion follows from
Lemma~\ref{supp:lem:frame-harmonicity}, since every
\(\phi_{r,T}\) belongs to \(V_r\).  Therefore, the frame functions
span \(V_r\).

For an arbitrary \(f\in L^2(\Sdk,\nu_{d,k})\), the fact that
\(\phi_{r,T}\in V_r\) gives
\[
  \inner{f}{\phi_{r,T}}
  =
  \inner{P_rf}{\phi_{r,T}}.
\]
Applying
\[
  A_{r,d,k}\mathcal F_r=I_{V_r}
\]
to \(P_rf\) proves \eqref{eq:frame-reconstruction}. Taking the inner product of
\eqref{eq:frame-reconstruction} with \(P_rf\), and using again
\(\inner{P_rf}{\phi_{r,T}}=\inner{f}{\phi_{r,T}}\), gives
\eqref{eq:frame-parseval}.

It remains to establish the projection identity
\[
  \phi_{r,T}
  =
  \frac{P_rh_T}{\norm{P_rh_T}_2},
\]
announced in Subsection~\ref{ssec:canonical-frame}.
Let \(G_T\) be the group of permutations of \([d]\) that
preserve \(T\) setwise, and define
\[
  (\mathcal U_\pi f)(S)
  =
  f(\pi^{-1}(S)).
\]
The action of a permutation on an inclusion indicator simply
relabels its indexing subset.  Indeed, for every
\(R\subseteq[d]\),
\[
\begin{aligned}
  (\mathcal U_\pi h_R)(S)
  &=
  h_R(\pi^{-1}(S))
\\
  &=
  \one_{\{R\subseteq\pi^{-1}(S)\}}
  =
  \one_{\{\pi(R)\subseteq S\}}
  =
  h_{\pi(R)}(S).
\end{aligned}
\]
Since \(\pi\) maps the collection of \(j\)-subsets of \([d]\)
bijectively onto itself, the action \(\mathcal U_\pi\) permutes the
generators
\[
  \{h_R:R\subseteq[d],\ |R|=j\}
\]
of \(U_j\).  Consequently,
\[
  \mathcal U_\pi(U_j)=U_j,
  \qquad 0\leqslant j\leqslant m.
\]

Now let \(\pi\in G_T\).  By definition of \(G_T\), the permutation
\(\pi\) preserves \(T\) setwise, that is, \(\pi(T)=T\).  Applying the
preceding identity with \(R=T\) gives
\[
  \mathcal U_\pi h_T
  =
  h_{\pi(T)}
  =
  h_T.
\]
Thus, \(h_T\) is invariant under every permutation in \(G_T\).

The action also preserves the uniform inner product.  Indeed, the map
\(S\mapsto\pi^{-1}(S)\) is a bijection of \(\Sdk\), and therefore
\[
\begin{aligned}
  \inner{\mathcal U_\pi f}{\mathcal U_\pi g}
  &=
  \frac1N
  \sum_{S\in\Sdk}
  f(\pi^{-1}(S))g(\pi^{-1}(S))
\\
  &=
  \frac1N
  \sum_{R\in\Sdk}
  f(R)g(R)
  =
  \inner{f}{g}.
\end{aligned}
\]
The isometry \(\mathcal U_\pi\) also preserves
\(U_{r-1}^{\perp}\).  Indeed, if
\(f\in U_{r-1}^{\perp}\) and \(u\in U_{r-1}\), then
\[
  \inner{\mathcal U_\pi f}{u}
  =
  \inner{f}{\mathcal U_{\pi^{-1}}u}
  =
  0,
\]
because
\(\mathcal U_{\pi^{-1}}u\in U_{r-1}\).
Since \(\mathcal U_\pi\) also preserves \(U_r\), it preserves
\[
  V_r=U_r\cap U_{r-1}^{\perp}.
\]

We next show that \(P_r\) commutes with this action.  Since
\(P_rf\in V_r\) and \(V_r\) is invariant under \(\mathcal U_\pi\),
\[
  \mathcal U_\pi P_rf\in V_r.
\]
Moreover, for every \(v\in V_r\),
\[
\begin{aligned}
  \inner{\mathcal U_\pi(f-P_rf)}{v}
  &=
  \inner{f-P_rf}{\mathcal U_{\pi^{-1}}v}
  =
  0,
\end{aligned}
\]
because \(\mathcal U_{\pi^{-1}}v\in V_r\), whereas
\(f-P_rf\perp V_r\).  Thus,
\[
  \mathcal U_\pi f
  =
  \mathcal U_\pi P_rf
  +
  \mathcal U_\pi(f-P_rf)
\]
is the orthogonal decomposition of \(\mathcal U_\pi f\) into a
component in \(V_r\) and a component in \(V_r^\perp\).  Therefore,
\[
  P_r(\mathcal U_\pi f)
  =
  \mathcal U_\pi(P_rf),
\]
or equivalently,
\[
  P_r\mathcal U_\pi
  =
  \mathcal U_\pi P_r.
\]
Applying this identity to \(h_T\) gives, for every \(\pi\in G_T\),
\[
  \mathcal U_\pi(P_rh_T)
  =
  P_r(\mathcal U_\pi h_T)
  =
  P_rh_T.
\]
Hence, \(P_rh_T\) is invariant under \(G_T\).

The \(G_T\)-orbits in \(\Sdk\) are the level sets of
\(L_T(S)=|S\cap T|\). Consequently, if \(g\in V_r\) is \(G_T\)-invariant, there exist
values \(y_0,\ldots,y_r\) such that
\(
  g(S)=y_\ell
\)
 whenever \(L_T(S)=\ell\).

Because \(r\leqslant m\), all overlap values
\(0,\ldots,r\) are attainable.  Lagrange interpolation therefore
gives a unique polynomial \(q\) of degree at most \(r\) satisfying
\(q(\ell)=y_\ell\) for \(0\leqslant\ell\leqslant r\).  Hence,
\[
  g(S)=q(L_T(S)).
\]  
For \(0\leqslant j\leqslant r-1\),
\eqref{supp:eq:binomial-inclusion} gives
\[
  S\longmapsto\binom{L_T(S)}{j}
  \in U_j\subseteq U_{r-1}.
\]
Since the polynomials
\(\ell\mapsto\binom{\ell}{j}\),
\(0\leqslant j\leqslant r-1\), form a basis of the polynomials of
degree at most \(r-1\), it follows that
\[
  S\longmapsto p(L_T(S))
  \in U_{r-1}
\]
for every polynomial \(p\) of degree at most \(r-1\). Since \(g\in V_r\subseteq U_{r-1}^{\perp}\), every such polynomial
\(p\) satisfies
\[
  0
  =
  \inner{g}{p\circ L_T}
  =
  \sum_{\ell=0}^r
  w_{r,d,k}(\ell)q(\ell)p(\ell).
\]
Thus \(q\) belongs to the orthogonal complement of the polynomials of
degree at most \(r-1\) within the polynomials of degree at most \(r\).
This orthogonal complement is one-dimensional, and
Lemma~\ref{supp:lem:top-polynomial} shows that it is spanned by
\(H_{r,d,k}\).  Therefore \(q\) is a scalar multiple of
\(H_{r,d,k}\). It follows that every
\(G_T\)-invariant vector in \(V_r\) is a scalar multiple of
\(\phi_{r,T}\).  Conversely,
Lemma~\ref{supp:lem:frame-harmonicity} shows that \(\phi_{r,T}\) is
a \(G_T\)-invariant unit vector in \(V_r\).  Hence the
\(G_T\)-invariant subspace of \(V_r\) is exactly \(\operatorname{span}\{\phi_{r,T}\}\).  Since \(P_rh_T\) is
\(G_T\)-invariant, there exists \(\beta_{r,T}\in\mathbb R\) such that
\[
  P_rh_T
  =
  \beta_{r,T}\phi_{r,T}.
\]
It remains to determine the proportionality constant and its sign. Using self-adjointness of \(P_r\), together with
\(P_r\phi_{r,T}=\phi_{r,T}\), we obtain
\[
\begin{aligned}
  \beta_{r,T}
  &=
  \inner{P_rh_T}{\phi_{r,T}}
   =
  \inner{h_T}{P_r\phi_{r,T}}
   =
  \inner{h_T}{\phi_{r,T}}
\\
  &=
  \P_{S\sim\nu_{d,k}}(T\subseteq S)
  H_{r,d,k}(r)
  >0.
\end{aligned}
\]
Because \(\norm{\phi_{r,T}}_2=1\), it follows that
\[
  \norm{P_rh_T}_2
  =
  \beta_{r,T}.
\]
Dividing the proportionality identity by this norm gives
\[
  \phi_{r,T}
  =
  \frac{P_rh_T}{\norm{P_rh_T}_2},
\]
as claimed.

\subsection{Frame representation of the projection estimator}

Applying the reconstruction identity in
Theorem~\ref{thm:canonical-frame} separately on the harmonic levels
\(V_1,\ldots,V_D\) gives the alternative representation
\[
  \widehat f_D^{\,\mathrm{proj}}
  =
  \sum_{r=1}^D A_{r,d,k}
  \sum_{\substack{T\subseteq[d]\\|T|=r}}
  \widehat\beta_{r,T}\phi_{r,T},
  \qquad
  \widehat\beta_{r,T}
  =
  \frac1n\sum_{j=1}^n
  Y_j\phi_{r,T}(S_j).
\]
This formula is explicit and equivariant under relabeling of the
ground set, but it uses the redundant canonical frame and is not
intended as a computationally optimal implementation.

\section{Projection and minimax proofs}
\label{supp:projection-minimax}

\subsection{Proof of Proposition~\ref{prop:projection-risk}}

When \(D=0\), one has \(W_0=\{0\}\),
\(\widehat f_0^{\,\mathrm{proj}}=0\), and \(p_0=0\).  Since
\(\gamma_{1,d}=1\),
\[
  \E_{f,P}
  \left[
    \norm{\widehat f_0^{\,\mathrm{proj}}-f}_2^2
  \right]
  =
  \norm{f}_2^2
  \leqslant B
  =
  b_0(s,B).
\]
Thus, the result holds for \(D=0\).  We henceforth fix
\(1\leqslant D\leqslant m\),
\(f\in\mathcal E^s_{d,k}(B)\), and
\(P\in\mathcal P_{\sigma^2}\).

Let \((S,Y)\) denote a generic observation having the same
distribution as each pair \((S_j,Y_j)\).  Thus,
\[
  S\sim\nu_{d,k},
  \qquad
  Y=f(S)+\varepsilon,
  \qquad
  \varepsilon\sim P,
\]
where \(\varepsilon\) is independent of \(S\).

Let \(e_1,\ldots,e_{p_D}\) be an orthonormal basis of \(W_D\), and
write
\[
  \theta_\ell
  =
  \inner{f}{e_\ell},
  \qquad
  1\leqslant\ell\leqslant p_D.
\]
Since the noise is centered and independent of the design,
\[
  \E_{f,P}[\widehat\theta_\ell]
  =
  \E_{f,P}[Ye_\ell(S)]
  =
  \E_{S\sim\nu_{d,k}}[f(S)e_\ell(S)]
  =
  \theta_\ell.
\]
Hence,
\[
  P_{\leqslant D}^{\circ}f
  =
  \sum_{\ell=1}^{p_D}\theta_\ell e_\ell
\]
is the expectation of the empirical projection estimator.

We next establish the constant-diagonal identity
\eqref{eq:constant-kernel-diagonal}. For a permutation \(\pi\) of \([d]\), recall the action
\[
  (\mathcal U_\pi g)(S)
  =
  g(\pi^{-1}(S)).
\]
As shown in
Subsection~\ref{supp:proof-canonical-frame},
\(\mathcal U_\pi\) preserves every harmonic space \(V_r\) and is an
isometry for the uniform inner product. It therefore preserves
\(W_D=V_1\mathbin{\oplus^\perp}\cdots
\mathbin{\oplus^\perp}V_D\). Thus,
if \(e_1,\ldots,e_{p_D}\) is an orthonormal basis of \(W_D\), then so
is
\[
  \mathcal U_{\pi^{-1}}e_1,\ldots,
  \mathcal U_{\pi^{-1}}e_{p_D}.
\]
Since the projection kernel is independent of the orthonormal basis
used to represent \(W_D\), it follows that
\[
\begin{aligned}
  \mathcal K_D(\pi(S),\pi(S'))
  &=
  \sum_{\ell=1}^{p_D}
  e_\ell(\pi(S))e_\ell(\pi(S'))
\\
  &=
  \sum_{\ell=1}^{p_D}
  (\mathcal U_{\pi^{-1}}e_\ell)(S)
  (\mathcal U_{\pi^{-1}}e_\ell)(S')
=
  \mathcal K_D(S,S').
\end{aligned}
\]

The permutation action on \(\Sdk\) is transitive.  Indeed, for any
\(S,R\in\Sdk\), choose a bijection from \(S\) onto \(R\) and a
bijection from \([d]\setminus S\) onto \([d]\setminus R\).
Together, these two bijections define a permutation \(\pi\) of
\([d]\) such that \(\pi(S)=R\).  The preceding invariance identity
then gives
\[
  \mathcal K_D(R,R)
  =
  \mathcal K_D(S,S).
\]
Thus, the diagonal
\(S\mapsto\mathcal K_D(S,S)\) is constant over \(\Sdk\).

Its value is determined by averaging under the uniform measure:
\[
\begin{aligned}
  \E_{S\sim\nu_{d,k}}
  [\mathcal K_D(S,S)]
  &=
  \sum_{\ell=1}^{p_D}
  \E_{S\sim\nu_{d,k}}[e_\ell(S)^2]
\\
  &=
  \sum_{\ell=1}^{p_D}\norm{e_\ell}_2^2
  =
  p_D.
\end{aligned}
\]
Consequently,
\[
  \mathcal K_D(S,S)
  =
  \sum_{\ell=1}^{p_D}e_\ell(S)^2
  =
  p_D,
  \qquad S\in\Sdk,
\]
which proves \eqref{eq:constant-kernel-diagonal}.

We now bound the coefficient-estimation term.  By the definitions of
the estimator and of the population projection,
\[
  \widehat f_D^{\,\mathrm{proj}}
  -
  P_{\leqslant D}^{\circ}f
  =
  \sum_{\ell=1}^{p_D}
  (\widehat\theta_\ell-\theta_\ell)e_\ell.
\]
Since the basis is orthonormal, the squared norm is the sum of the
squared coefficient differences.  Using also the unbiasedness of
\(\widehat\theta_\ell\), we obtain
\[
\begin{aligned}
  \E_{f,P}\!\left[
    \norm{
      \widehat f_D^{\,\mathrm{proj}}
      -P_{\leqslant D}^{\circ}f
    }_2^2
  \right]
&=
  \sum_{\ell=1}^{p_D}
  \E_{f,P}
  \left[
    (\widehat\theta_\ell-\theta_\ell)^2
  \right]\\
  &=
  \sum_{\ell=1}^{p_D}
  \Var_{f,P}(\widehat\theta_\ell).
\end{aligned}
\]
Since the observations are i.i.d.,
\[
  \Var_{f,P}(\widehat\theta_\ell)
  =
  \frac1n
  \Var_{f,P}\!\left(Ye_\ell(S)\right)
\leqslant
  \frac1n
  \E_{f,P}
  \left[
    Y^2e_\ell(S)^2
  \right].
\]
Summing over \(\ell\) and using the constant-diagonal identity gives
\[
\begin{aligned}
  \E_{f,P}\!\left[
    \norm{
      \widehat f_D^{\,\mathrm{proj}}
      -P_{\leqslant D}^{\circ}f
    }_2^2
  \right]
  &\leqslant
  \frac1n
  \E_{f,P}
  \left[
    Y^2
    \sum_{\ell=1}^{p_D}e_\ell(S)^2
  \right]
\\
  &=
  \frac{p_D}{n}\E_{f,P}[Y^2].
\end{aligned}
\]
The last equality is precisely where
\(\mathcal K_D(S,S)=p_D\) is used.

Because \(Y=f(S)+\varepsilon\), and because \(\varepsilon\) is
centered and independent of \(S\),
\[
  \E_{f,P}[f(S)\varepsilon]
  =
  0.
\]
Therefore,
\[
  \E_{f,P}[Y^2]
  =
  \norm{f}_2^2+\E_P[\varepsilon^2]
\leqslant
  B+\sigma^2.
\]
We conclude that
\[
  \E_{f,P}\!\left[
    \norm{
      \widehat f_D^{\,\mathrm{proj}}
      -P_{\leqslant D}^{\circ}f
    }_2^2
  \right]
  \leqslant
  (\sigma^2+B)\frac{p_D}{n}.
\]

Finally, for every realization of the sample,
\[
  \widehat f_D^{\,\mathrm{proj}}
  -
  P_{\leqslant D}^{\circ}f
  \in W_D,
\]
whereas
\[
  f-P_{\leqslant D}^{\circ}f
  \in W_D^\perp.
\]
Pythagoras therefore gives
\[
  \norm{\widehat f_D^{\,\mathrm{proj}}-f}_2^2
  =
  \norm{
    \widehat f_D^{\,\mathrm{proj}}
    -P_{\leqslant D}^{\circ}f
  }_2^2
+
  \norm{
    f-P_{\leqslant D}^{\circ}f
  }_2^2.
\]
If \(D<m\), the harmonic tail bound
\eqref{eq:harmonic-tail} gives
\[
  \norm{
    f-P_{\leqslant D}^{\circ}f
  }_2^2
  =
  \sum_{r>D}\norm{P_rf}_2^2
  \leqslant
  b_D(s,B).
\]
If \(D=m\), then
\(W_m=V_0^\perp\); since \(f\) is centered,
\(f\in W_m\) and
\(P_{\leqslant m}^{\circ}f=f\).  Hence, the approximation term is
zero. 

Combining the
estimation and approximation bounds, and then taking the supremum
over \(f\in\mathcal E^s_{d,k}(B)\) and
\(P\in\mathcal P_{\sigma^2}\), proves
\eqref{eq:projection-risk}.

\subsection{Proof of Theorem~\ref{thm:johnson-minimax}}

Write
\[
  a_D=a_D(s,B),
  \qquad
  v_D=v_D(n,\sigma^2),
  \qquad
  \mathfrak R=\mathfrak R_{n,d,k}(s,B,\sigma^2).
\]

\emph{Lower bound.}
Fix \(1\leqslant D\leqslant m\), and put \(p=p_D\).
For every \(g\in W_D\), the spectral decomposition of the
Johnson--Sobolev energy and the monotonicity of
\(r\mapsto\gamma_{r,d}\) give
\[
  I_{J,d,k}^{(2,s)}(g)
  =
  \sum_{r=1}^D
  \gamma_{r,d}^{\,s}\|P_rg\|_2^2
  \leqslant
  \gamma_{D,d}^{\,s}\|g\|_2^2.
\]
Consequently,
\[
  \mathcal B_D
  :=
  \left\{
    g\in W_D:\|g\|_2^2\leqslant a_D
  \right\}
  \subseteq
  \mathcal E^s_{d,k}(B).
\]

We restrict the noise distribution to the Gaussian member
\(G=N(0,\sigma^2)\) of \(\mathcal P_{\sigma^2}\).  Since the design
law does not depend on the response function, the chain rule for
Kullback--Leibler divergence gives, for \(g,g'\in W_D\),
\begin{align}
  \operatorname{KL}(\mathbb P_{g,G},\mathbb P_{g',G})
  &=
  \frac{1}{2\sigma^2}
  \sum_{j=1}^n
  \E_{S_j\sim\nu_{d,k}}
  \bigl[\{g(S_j)-g'(S_j)\}^2\bigr]
  \notag\\
  &=
  \frac{n}{2\sigma^2}\|g-g'\|_2^2.
  \label{supp:eq:regression-KL}
\end{align}

Let \(e_1,\ldots,e_p\) be an orthonormal basis of \(W_D\).  For
\(\omega,\omega'\in\{0,1\}^p\), define their Hamming distance by
\[
  \operatorname{Ham}(\omega,\omega')
  =
  \sum_{\ell=1}^p
  \one_{\{\omega_\ell\ne\omega'_\ell\}}.
\]
Massart's version of the Varshamov--Gilbert lemma
\citep[Lemma~4.7]{Massart2007}, applied with \(\alpha=1/2\), provides
a set \(\Omega\subseteq\{0,1\}^p\) such that
\begin{equation}\label{supp:eq:massart-packing}
  \operatorname{Ham}(\omega,\omega')>\frac p4
  \quad\text{for all distinct }\omega,\omega'\in\Omega,
  \qquad
  \log|\Omega|>\frac p8.
\end{equation}
Translating \(\Omega\) by any one of its elements, using
coordinatewise addition modulo two, preserves these properties.
We may therefore assume that \(0\in\Omega\).

For \(\omega\in\Omega\), set
\[
  g_\omega
  =
  \delta\sum_{\ell=1}^p\omega_\ell e_\ell,
  \qquad
  \delta^2
  =
  c_1
  \min\left\{
    \frac{a_D}{p},
    \frac{\sigma^2}{n}
  \right\},
\]
where \(c_1\in(0,1]\) is a sufficiently small universal constant.
Orthonormality gives
\[
  \|g_\omega\|_2^2
  \leqslant p\delta^2
  \leqslant a_D,
\]
so every \(g_\omega\) belongs to \(\mathcal B_D\).  Moreover,
\begin{equation}\label{supp:eq:packing-separation}
  \|g_\omega-g_{\omega'}\|_2^2
  =
  \delta^2\operatorname{Ham}(\omega,\omega')
  >
  \frac{p\delta^2}{4},
  \qquad \omega\ne\omega'.
\end{equation}

Suppose that \(p\geqslant8\), and let \(M=|\Omega|-1\).  Since
\(0\in\Omega\), the zero response may be used as the reference
distribution.  By \eqref{supp:eq:regression-KL},
\[
  \operatorname{KL}(\mathbb P_{g_\omega,G},
                     \mathbb P_{0,G})
  \leqslant
  \frac{c_1p}{2},
  \qquad
  \omega\in\Omega\setminus\{0\}.
\]
Moreover, \(\log|\Omega|>p/8\) and \(p\geqslant8\) imply
\(|\Omega|\geqslant3\), and hence
\(M=|\Omega|-1\geqslant|\Omega|/2\).  Therefore,
\[
  \log M
  \geqslant
  \log|\Omega|-\log2
  >
  \frac p8-\log2
  \geqslant
  \frac{1-\log2}{8}\,p.
\]
Thus, by choosing \(c_1>0\) sufficiently small, the preceding
Kullback--Leibler divergences are bounded by a sufficiently small
fixed fraction of \(\log M\), as required by Fano's inequality.

By \eqref{supp:eq:packing-separation}, the signals are separated by
more than \(2s_0\) in \(L^2\), where
\[
  s_0=\frac{\sqrt p\,\delta}{4}.
\]
Fano's inequality in the form stated by
\citet[Theorem~2.5]{Tsybakov2009} gives a universal constant \(c_{\mathrm F}>0\) such that
\[
  \inf_{\widehat f}
  \sup_{\omega\in\Omega}
  \mathbb P_{g_\omega,G}
  \left(
    \|\widehat f-g_\omega\|_2\geqslant s_0
  \right)
  \geqslant c_{\mathrm F}.
\]
Since
\[
  \E_{g_\omega,G}
  \left[
    \|\widehat f-g_\omega\|_2^2
  \right]
  \geqslant
  s_0^2
  \mathbb P_{g_\omega,G}
  \left(
    \|\widehat f-g_\omega\|_2\geqslant s_0
  \right),
\]
it follows that
\[
  R_n^\star(\mathcal E^s_{d,k}(B))
  \geqslant
  c_{\mathrm F}s_0^2
  =
  \frac{c_{\mathrm F}p\delta^2}{16}
  =
  \frac{c_{\mathrm F}c_1}{16}
  \min\left\{
    a_D,\frac{\sigma^2p}{n}
  \right\}
  =
  c_{\geqslant8}\min\{a_D,v_D\},
\]
where
\[
  c_{\geqslant8}
  =
  \frac{c_{\mathrm F}c_1}{16}>0
\]
is universal. The condition
\(p\geqslant8\) is needed only for this uniform application of Fano's
inequality; the packing \eqref{supp:eq:massart-packing} exists for
every \(p\geqslant1\).

It remains to cover the finitely many dimensions \(1\leqslant p<8\). Let \(e_1\) be any unit vector in \(W_D\)
and consider
\[
  g_+=te_1,
  \qquad
  g_-=-te_1,
  \qquad
  t>0,
  \qquad
  t^2
  =
  \min\left\{
    a_D,\frac{\sigma^2}{n}
  \right\}.
\]
Both signals belong to \(\mathcal B_D\), and
\[
  \|g_+-g_-\|_2=2t.
\]
Moreover, \eqref{supp:eq:regression-KL} gives
\[
  \operatorname{KL}(\mathbb P_{g_+,G},
                     \mathbb P_{g_-,G})
  =
  \frac{n}{2\sigma^2}\|g_+-g_-\|_2^2
  =
  \frac{2nt^2}{\sigma^2}
  \leqslant2.
\]

To connect testing and estimation, associate with any estimator
\(\widehat f\) the minimum-distance test
\[
  \psi_{\widehat f}
  \in
  \underset{\eta\in\{+,-\}}{\arg\min}\,
  \|\widehat f-g_\eta\|_2,
\]
with an arbitrary rule for breaking ties.  Since the two signals are
separated by \(2t\),
\[
  \{\psi_{\widehat f}\neq\eta\}
  \subseteq
  \{\|\widehat f-g_\eta\|_2\geqslant t\},
  \qquad \eta\in\{+,-\}.
\]
Consequently,
\[
  \inf_{\widehat f}
  \max_{\eta\in\{+,-\}}
  \mathbb P_{g_\eta,G}
  \left(
    \|\widehat f-g_\eta\|_2\geqslant t
  \right)
  \geqslant p_{e,1},
\]
where
\[
  p_{e,1}
  :=
  \inf_{\psi}
  \max_{\eta\in\{+,-\}}
  \mathbb P_{g_\eta,G}(\psi\neq\eta),
\]
the infimum being taken over all measurable tests
\(\psi\) with values in \(\{+,-\}\). By the Kullback--Leibler version of the two-hypothesis bound
\citep[Theorem~2.2(iii)]{Tsybakov2009},
\[
  p_{e,1}
  \geqslant
  \frac14\exp(-2).
\]
For every estimator \(\widehat f\) and every
\(\eta\in\{+,-\}\),
\[
  \E_{g_\eta,G}
  \left[
    \|\widehat f-g_\eta\|_2^2
  \right]
  \geqslant
  t^2
  \mathbb P_{g_\eta,G}
  \left(
    \|\widehat f-g_\eta\|_2\geqslant t
  \right).
\]
Combining this inequality with the preceding reduction from
estimation to testing gives
\[
  R_n^\star(\mathcal E^s_{d,k}(B))
  \geqslant
  t^2p_{e,1}
  \geqslant
  \frac{t^2}{4}\exp(-2)
  =
  c_3
  \min\left\{
    a_D,\frac{\sigma^2}{n}
  \right\},
\]
where
\(
  c_3
  :=
  \frac{1}{4}\exp(-2)>0
\)
is universal. Since \(1\leqslant p<8\),
\[
  \min\left\{
    a_D,\frac{\sigma^2}{n}
  \right\}
  \geqslant
  \frac18
  \min\left\{
    a_D,\frac{\sigma^2p}{n}
  \right\}
  =
  \frac18\min\{a_D,v_D\}.
\]
Consequently, with
\[
  c
  =
  \min\left\{
    c_{\geqslant8},\frac{c_3}{8}
  \right\}>0,
\]
we have proved, for every \(1\leqslant D\leqslant m\),
\[
  R_n^\star(\mathcal E^s_{d,k}(B))
  \geqslant
  c\min\{a_D,v_D\}.
\]
Maximizing over \(D\) gives
\[
  R_n^\star(\mathcal E^s_{d,k}(B))
  \geqslant
  c\,\mathfrak R,
\]
which proves the lower bound.

\emph{Upper bound.}
The sequence \((v_D)_{D=1}^m\) is increasing because
\[
  p_D=\binom dD-1
\]
is increasing for \(D\leqslant m\leqslant d/2\).  Similarly,
\((a_D)_{D=1}^m\) is decreasing because
\(D\mapsto\gamma_{D,d}\) is increasing on this range.  Adopt the
conventions
\[
  v_0=0,
  \qquad
  a_0=+\infty,
\]
and let
\[
  D_\star
  =
  \max\left\{
    0\leqslant D\leqslant m:
    v_D\leqslant a_D
  \right\}.
\]
The set in this definition is nonempty because it contains \(D=0\).

We use \(D_\star\) to compare the optimized projection bound with
the spectral scale \(\mathfrak R\). The comparison needed for the
upper bound is
\[
  \min_{0\leqslant D\leqslant m}
  \{v_D+b_D(s,B)\}
  \leqslant 2\mathfrak R.
\]

If \(D_\star=m\), then \(v_m\leqslant a_m\), and hence
\[
  \min_{0\leqslant D\leqslant m}
  \left\{
    v_D+\one_{\{D<m\}}a_{D+1}
  \right\}
  \leqslant
  v_m
  =
  \min\{v_m,a_m\}
  \leqslant
  \mathfrak R.
\]

Suppose now that \(D_\star<m\).  By maximality,
\(
  v_{D_\star+1}>a_{D_\star+1},
\)
so
\[
  \mathfrak R
  \geqslant
  \min\{v_{D_\star+1},a_{D_\star+1}\}
  =
  a_{D_\star+1}.
\]
We also have \(v_{D_\star}\leqslant\mathfrak R\).
For \(D_\star\geqslant1\), this follows from
\(\min\{v_{D_\star},a_{D_\star}\}=v_{D_\star}\);
for \(D_\star=0\), it follows from \(v_0=0\).
Using \(D=D_\star\) in the minimum therefore gives
\[
  \min_{0\leqslant D\leqslant m}
  \left\{
    v_D+\one_{\{D<m\}}a_{D+1}
  \right\}
  \leqslant
  v_{D_\star}+a_{D_\star+1}
  \leqslant
  2\mathfrak R.
\]
Thus, in every case,
\begin{equation}\label{supp:eq:crossing}
  \min_{0\leqslant D\leqslant m}
  \left\{
    v_D+\one_{\{D<m\}}a_{D+1}
  \right\}
  \leqslant
  2\mathfrak R.
\end{equation}

Since \(b_D(s,B)=\one_{\{D<m\}}a_{D+1}\), equation
\eqref{supp:eq:crossing} proves the comparison announced above.
Put \(\tau=1+B/\sigma^2\). The cutoff \(D_{\mathrm{or}}\)
defined in the main text minimizes \(\tau v_D+b_D(s,B)\).
By Proposition~\ref{prop:projection-risk}, the worst-case risk of
\(\widehat f_{D_{\mathrm{or}}}^{\,\mathrm{proj}}\) is at most
\[
  \tau v_{D_{\mathrm{or}}}+b_{D_{\mathrm{or}}}(s,B)
  =
  \min_{0\leqslant D\leqslant m}\{\tau v_D+b_D(s,B)\}.
\]
As \(\tau\geqslant1\), \eqref{supp:eq:crossing} gives
\[
  \min_{0\leqslant D\leqslant m}\{\tau v_D+b_D(s,B)\}
  \leqslant
  \tau\min_{0\leqslant D\leqslant m}\{v_D+b_D(s,B)\}
  \leqslant 2\tau\mathfrak R.
\]
This proves the upper bound in \eqref{eq:johnson-minimax}.

If \(B/\sigma^2\leqslant C_{\mathrm{snr}}\), then the two bounds just
proved yield
\[
  c\,\mathfrak R_{n,d,k}(s,B,\sigma^2)
  \leqslant
  R_n^\star(\mathcal E^s_{d,k}(B))
  \leqslant
  2(1+C_{\mathrm{snr}})
  \mathfrak R_{n,d,k}(s,B,\sigma^2).
\]
The constant \(c>0\) is universal, and the upper constant depends
only on \(C_{\mathrm{snr}}\).  The comparison is therefore uniform
over all \(d,k,n,s,B,\sigma^2\) satisfying the assumptions of the
theorem and \(B/\sigma^2\leqslant C_{\mathrm{snr}}\).

\subsection{Proof of Proposition~\ref{prop:active-dimension}}
\label{supp:active-items}

Let \(H_J\) be the group of permutations of \([d]\) that fix every
item in \(J\). Two sets in \(\Sdk\) lie in the same \(H_J\)-orbit
if and only if they have the same intersection with \(J\).
Consequently, the functions fixed by \(H_J\) are precisely those in
\(\mathcal Q_J\). The group permutes the indicator functions spanning \(U_r\), so
\(U_r\) and \(U_{r-1}\) are invariant. Its action also preserves the
\(L^2\) inner product, and hence \(U_{r-1}^{\perp}\). Thus
\(V_r=U_r\cap U_{r-1}^{\perp}\) is invariant, and the orthogonal
projection \(P_r\) commutes with the group action. Thus \(P_rf\) is fixed by
\(H_J\) whenever \(f\in\mathcal Q_J\), proving
\(P_r\mathcal Q_J\subseteq\mathcal Q_J\).

To find the dimensions, we first identify \(U_r\cap\mathcal Q_J\).
Averaging over \(H_J\) leaves functions in \(\mathcal Q_J\)
unchanged and maps \(U_r\) into this intersection. Let \(\mathsf A_J\) denote averaging over \(H_J\). We apply it
to the indicators \(h_T\), \(|T|=r\), that span \(U_r\). For \(|T|=r\), write \(A=T\cap J\), \(b=|T\setminus J|\),
and \(z(S)=|S\cap J|\). The group average of \(h_T\) is
\[
  h_A(S)\frac{\binom{k-z(S)}b}{\binom{d-q}b}.
\]
Indeed, once \(A\subseteq S\), the fraction is the proportion of
\(b\)-subsets of \([d]\setminus J\) contained in \(S\setminus J\).
Since \(z(S)=\sum_{i\in J}h_{\{i\}}(S)\), this average is a
polynomial of degree at most \(|A|+b=r\) in the membership indicators
of items in \(J\). Using \(h_{\{i\}}^2=h_{\{i\}}\), we can express this polynomial
as a linear combination of \(h_C\) with \(C\subseteq J\) and
\(|C|\leqslant r\).

If \(f\in U_r\cap\mathcal Q_J\), write
\(f=\sum_T a_T h_T\). Averaging leaves \(f\) unchanged, while
linearity gives \(f=\sum_T a_T\,\mathsf A_J h_T\). Each term
\(\mathsf A_J h_T\) lies in the span above. Hence
\(U_r\cap\mathcal Q_J\) is contained in that span. For the
reverse inclusion, each \(h_C\) with \(C\subseteq J\) and
\(|C|\leqslant r\) belongs to \(U_r\) and depends only on
\(S\cap J\). Thus
\[
  U_r\cap\mathcal Q_J
  =\operatorname{span}\{h_C:C\subseteq J,\ |C|\leqslant r\}.
\]

For \(q\leqslant d-k\), every intersection pattern
\(A=S\cap J\) with \(|A|\leqslant\min(q,k)\) is attainable:
complete \(A\) with \(k-|A|\) items outside \(J\).
Suppose that
\[
  \sum_{\substack{C\subseteq J\\ |C|\leqslant r}} a_C h_C=0.
\]
Evaluating at a set \(S_A\) with \(S_A\cap J=A\), for
\(A\subseteq J\) and \(|A|\leqslant r\), gives
\[
  \sum_{C\subseteq A}a_C=0.
\]
Starting with \(A=\varnothing\)
and proceeding by increasing \(|A|\) shows that every \(a_C=0\).
Thus the \(h_C\) are linearly independent through order
\(r_J=\min(q,m)=\min(q,k)\), where the last equality
uses \(q\leqslant d-k\). 

Since each \(P_t\) preserves \(\mathcal Q_J\), the harmonic
decomposition restricts to
\[
  U_r\cap\mathcal Q_J
  =\bigoplus_{t=0}^r(V_t\cap\mathcal Q_J).
\]
Taking the difference of the dimensions at \(r\) and \(r-1\)
gives \(\dim(V_r\cap\mathcal Q_J)=\binom qr\) for
\(1\leqslant r\leqslant r_J\). Since \(W_D\) excludes the
constant level \(V_0\),
\[
  W_{J,D}
  =\bigoplus_{r=1}^D(V_r\cap\mathcal Q_J),
  \qquad
  \dim W_{J,D}=\sum_{r=1}^D\binom qr
\]
for \(0\leqslant D\leqslant r_J\).

A function in \(\mathcal Q_J\) can take an arbitrary value at each
attainable pattern \(S\cap J\). Hence
\(\dim\mathcal Q_J=\sum_{t=0}^{r_J}\binom qt\), which equals
\(\dim(U_{r_J}\cap\mathcal Q_J)\). Thus
\(\mathcal Q_J\subseteq U_{r_J}\). Every \(f\in\mathcal Q_J\) therefore satisfies \(P_rf=0\) for
\(r>r_J\), and \(\sum_{r=0}^{r_J}P_rf=f\). Since the spaces \(W_D\) are nested and \(\mathcal Q_J\) has no
higher harmonic components, \(W_{J,D}=W_{J,r_J}\) for
\(r_J<D\leqslant m\).

Let \(S\sim\nu_{d,k}\). For
\(f\in\mathcal E^s_{d,k}(B)\cap\mathcal Q_J\),
\(P\in\mathcal P_{\sigma^2}\), and \(0\leqslant D\leqslant r_J\),
orthonormality and the definition of \(\Lambda_{J,D}\) give
\[
  \E\mathcal K_{J,D}(S,S)=p_{q,D}
  \quad\text{and}\quad
  \E[f(S)^2\mathcal K_{J,D}(S,S)]
  \leqslant\Lambda_{J,D}\|f\|_2^2.
\]
Using these facts,
the empirical-coefficient calculation of
Proposition~\ref{prop:projection-risk} gives
\[
 \E_{f,P}\!\left[
    \|\widehat f_{J,D}^{\,\mathrm{proj}}
       -P_{W_{J,D}}f\|_2^2
  \right]
  \leqslant
  \frac{\sigma^2p_{q,D}+B\Lambda_{J,D}}n.
\]
Since \(P_{W_{J,D}}f=\sum_{r=1}^D P_rf\), the harmonic-tail
bound controls the remaining approximation error, which is zero
at \(D=r_J\). For the fixed-support lower bound, let
\(1\leqslant D\leqslant r_J\). The ball of squared radius
\(B/\gamma_{D,d}^{\,s}\) in \(W_{J,D}\) belongs to
\(\mathcal E^s_{d,k}(B)\cap\mathcal Q_J\); the testing argument
of Theorem~\ref{thm:johnson-minimax} then applies in dimension
\(p_{q,D}\).

\paragraph{Pairwise kernel diagonal.}
Suppose \(k=D=2\) and \(2\leqslant q\leqslant d-2\).
Put \(N=\binom d2\), choose a pair \(S_0\subseteq J\), and set
\(u(S)=\one_{\{S=S_0\}}-1/N\). Since \(S\cap J=S_0\) exactly when
\(S=S_0\), the centered function \(u\) belongs to \(W_{J,2}\).
Moreover,
\[
  \|u\|_2^2=N^{-1}(1-N^{-1}),
  \qquad
  \frac{u(S_0)^2}{\|u\|_2^2}=N-1.
\]
The projection-kernel inequality
\(u(S_0)^2\leqslant\|u\|_2^2\mathcal K_{J,2}(S_0,S_0)\)
therefore gives \(\mathcal K_{J,2}(S_0,S_0)\geqslant N-1\).
Since \(W_{J,2}\subseteq W_2=\one^\perp\), its kernel diagonal
is at most \(\mathcal K_2(S,S)=N-1\). Hence
\(\Lambda_{J,2}=N-1\). Its diagonal cannot be constant, because
its average is \(p_{q,2}<N-1\).

\subsection{Proof of Proposition~\ref{prop:unknown-active-lower}}
\label{supp:unknown-active-lower}

Fix \(1\leqslant q\leqslant d/4\), and restrict the noise law
to the Gaussian member \(G=N(0,\sigma^2)\) of
\(\mathcal P_{\sigma^2}\), as in the proof of
Theorem~\ref{thm:johnson-minimax}. To apply Fano's inequality,
we construct many active sets of size \(q\) that are well
separated in symmetric difference. Let
\(\mathcal C\) contain all \(q\)-subsets of \([d]\), and start with
\(\mathcal J=\varnothing\). While \(\mathcal C\) is nonempty,
choose \(J\in\mathcal C\), add it to \(\mathcal J\), and remove
from \(\mathcal C\) every \(J'\) with
\(|J\triangle J'|<q/4\), including \(J\) itself.
The retained supports are therefore separated by at least \(q/4\).
Note that the family \(\mathcal J\) has at least two members, since a
\(q\)-subset disjoint from the first selected support remains
available after the first step. We now give a sharper lower bound
on \(|\mathcal J|\).

A \(q\)-set at distance \(2t\) from \(J\) is obtained by removing
\(t\) items from \(J\) and adding \(t\) from its complement.
Thus each step discards at most
\[
  V_q=\sum_{0\leqslant t<q/8}
      \binom qt\binom{d-q}t
\]
supports, including the selected one, and
\[
  |\mathcal J|\geqslant\frac{\binom dq}{V_q}.
\]
For \(0<t<q/8\), the bound
\(\binom{d-q}{t}\leqslant(e(d-q)/t)^t\) and the
monotonicity of \(u\log(e(d-q)/u)\) for \(u\leqslant d-q\) give
\[
  \binom{d-q}t
  \leqslant\left(\frac{8ed}{q}\right)^{q/8}.
\]
The same inequality holds for \(t=0\). Since
\(\sum_t\binom qt\leqslant2^q\), it follows that
\[
  V_q\leqslant2^q\left(\frac{8ed}{q}\right)^{q/8}.
\]
Now use \(\binom dq\geqslant(d/q)^q\) together with
\(|\mathcal J|\geqslant\binom dq/V_q\).
Writing \(L=\log(d/q)\), we obtain
\[
  \log|\mathcal J|
  \geqslant
  q\left\{\frac78L-\log2-\frac{\log8+1}{8}\right\}
  \geqslant\frac{q}{20}(1+L)
  =\frac{q}{20}\log(ed/q),
\]
where the second inequality uses \(L\geqslant\log4\).
We have therefore established both properties needed below,
\[
  |J\triangle J'|\geqslant q/4
  \quad(J,J'\in\mathcal J,\ J\ne J'),
  \qquad
  \log|\mathcal J|\geqslant
  \frac{q}{20}\log(ed/q).
\]
Put \(X_i(S)=\one_{\{i\in S\}}\) and \(\alpha=k/d\).
For each \(J\in\mathcal J\), define
\[
  f_J(S)=a\sum_{i\in J}\{X_i(S)-\alpha\}.
\]
Each \(f_J\) is centered, belongs to \(V_1\cap\mathcal Q_J\),
and has Johnson--Sobolev energy \(\|f_J\|_2^2\), because
\(\gamma_{1,d}=1\). The covariance of the membership indicators
gives
\[
  \|f_J\|_2^2
  =a^2\alpha(1-\alpha)
    \frac{q(d-q)}{d-1},
  \qquad
  \|f_J-f_{J'}\|_2^2
  =a^2\alpha(1-\alpha)
    \frac{d}{d-1}|J\triangle J'|.
\]

Set
\[
  T=\min\left\{B,\frac{\sigma^2q\log(ed/q)}n\right\},
  \qquad
  a^2=\frac{c_2T}{\alpha(1-\alpha)q},
\]
where \(c_2\in(0,1]\) will be chosen below. Then
\[
  \|f_J\|_2^2
  =c_2T\,\frac{d-q}{d-1}
  \leqslant c_2T\leqslant B,
\]
so every signal belongs to \(\mathcal E^s_{d,k}(B;q)\).
For distinct \(J,J'\in\mathcal J\), the packing separation gives
\[
  \|f_J-f_{J'}\|_2^2
  =\frac{c_2T}{q}\frac{d}{d-1}|J\triangle J'|
  \geqslant\frac{c_2T}{4}.
\]

Since \(q\leqslant d/4\), we also have
\(\|f_J\|_2^2\geqslant3c_2T/4\). Together with the
preceding bound for \(\|f_J-f_{J'}\|_2^2\), this shows that
\(\{0\}\cup\{f_J:J\in\mathcal J\}\) has pairwise squared
\(L^2\)-separation at least \(c_2T/4\).

Put \(M=|\mathcal J|\), and use the zero response as the
reference distribution. The Gaussian regression identity
\eqref{supp:eq:regression-KL} gives, for every
\(J\in\mathcal J\),
\[
  \operatorname{KL}(\mathbb P_{f_J,G},\mathbb P_{0,G})
  =\frac{n}{2\sigma^2}\|f_J\|_2^2
  \leqslant\frac{c_2nT}{2\sigma^2}
  \leqslant\frac{c_2q}{2}\log(ed/q)
  \leqslant10c_2\log M.
\]
Choose \(c_2\) as a sufficiently small universal constant, and
set \(s_0=\sqrt{c_2T}/8\). Fano's inequality in the form used in
the proof of Theorem~\ref{thm:johnson-minimax}
\citep[Theorem~2.5]{Tsybakov2009} gives a universal constant
\(c_{\mathrm F}>0\) such that
\[
  \inf_{\widehat f}
  \sup_{h\in\{0\}\cup\{f_J:J\in\mathcal J\}}
  \mathbb P_{h,G}\bigl(\|\widehat f-h\|_2\geqslant s_0\bigr)
  \geqslant c_{\mathrm F}.
\]
Consequently,
\[
  R_n^\star\bigl(\mathcal E^s_{d,k}(B;q)\bigr)
  \geqslant c_{\mathrm F}s_0^2
  =\frac{c_{\mathrm F}c_2}{64}\,T.
\]
This proves the proposition.

\section{Random-design proofs}
\label{supp:random-design}

\subsection{Proof of Proposition~\ref{prop:johnson-gram}}

Fix \(1\leqslant D\leqslant m\). By \eqref{eq:constant-kernel-diagonal}, the leverage score is constant, i.e.,
\( \|x_D(S)\|_2^2=p_D\), \(S\in\Sdk\). Write
\[
  \widehat G_D
  =
  \sum_{j=1}^n Z_j,
  \qquad
  Z_j
  =
  \frac1n x_D(S_j)x_D(S_j)^\top.
\]
The matrices \(Z_1,\ldots,Z_n\) are independent and positive
semidefinite.  Moreover, orthonormality gives
\[
  \E_{S_j\sim\nu_{d,k}}[Z_j]
  =
  \frac1n I_{p_D},
  \qquad
  \sum_{j=1}^n
  \E_{S_j\sim\nu_{d,k}}[Z_j]
  =
  I_{p_D}.
\]
Since a rank-one matrix \(xx^\top\) has unique nonzero eigenvalue
\(\|x\|_2^2\), the constant-leverage identity also gives
\[
  \lambda_{\max}(Z_j)
  =
  \frac1n\|x_D(S_j)\|_2^2
  =
  \frac{p_D}{n}
  \qquad\text{almost surely}.
\]

Thus, in the notation of the matrix Chernoff inequalities
\citep[Corollary~5.2]{Tropp2012},
\[
  \mu_{\min}
  =
  \lambda_{\min}\left(
    \sum_{j=1}^n\E_{S_j\sim\nu_{d,k}}[Z_j]
  \right)
  =1,
  \qquad
  \mu_{\max}
  =
  \lambda_{\max}\left(
    \sum_{j=1}^n\E_{S_j\sim\nu_{d,k}}[Z_j]
  \right)
  =1,
\]
and the uniform eigenvalue bound is \(R=p_D/n\).

At relative deviation \(1/2\), the upper-tail inequality gives
\[
  \P_{\nu_{d,k}^{\otimes n}}
  \left\{
    \lambda_{\max}(\widehat G_D)>\frac32
  \right\}
  \leqslant
  p_D
  \exp\left(-c_0\frac{n}{p_D}\right),
\]
where
\[
  c_0
  =
  \frac32\log\left(\frac32\right)-\frac12.
\]
The corresponding lower-tail exponent is
\[
  c_-
  =
  \frac12+\frac12\log\left(\frac12\right)
  \approx0.1534,
\]
which is larger than
\(c_0\approx0.1082\).
Therefore
\[
  \P_{\nu_{d,k}^{\otimes n}}
  \left\{
    \lambda_{\min}(\widehat G_D)<\frac12
  \right\}
  \leqslant
  p_D
  \exp\left(-c_0\frac{n}{p_D}\right).
\]
Since \(\widehat G_D\) is symmetric,
\[
  \|\widehat G_D-I_{p_D}\|_{\mathrm{op}}
  =
  \max\left\{
    \lambda_{\max}(\widehat G_D)-1,\,
    1-\lambda_{\min}(\widehat G_D)
  \right\}.
\]
Consequently,
\[
  \mathcal A_D^c
  =
  \left\{
    \lambda_{\max}(\widehat G_D)>\frac32
  \right\}
  \cup
  \left\{
    \lambda_{\min}(\widehat G_D)<\frac12
  \right\}.
\]
A union bound consequently yields
\[
  \P_{\nu_{d,k}^{\otimes n}}(\mathcal A_D^c)
  \leqslant
  2p_D
  \exp\left(-c_0\frac{n}{p_D}\right),
\]
which is \eqref{eq:gram-concentration}.

\subsection{Proof of Theorem~\ref{thm:stable-ls}}

Fix \(1\leqslant D\leqslant m\),
\(f\in\mathcal E^s_{d,k}(B)\), and
\(P\in\mathcal P_{\sigma^2}\).  As in the main text, write
\[
  g_D=P_{\leqslant D}^{\circ}f,
  \qquad
  h_D=f-g_D,
\]
and let \(\theta_D\) be the coefficient vector of \(g_D\) in the
chosen orthonormal basis.  Then
\(g_D(S)=x_D(S)^\top\theta_D\), and population orthogonality gives
\begin{equation}\label{supp:eq:residual-centered}
  \E_{S\sim\nu_{d,k}}[x_D(S)h_D(S)]=0.
\end{equation}
Set
\[
  z_h=\frac1n\sum_{j=1}^n x_D(S_j)h_D(S_j),
  \qquad
  z_\varepsilon
  =\frac1n\sum_{j=1}^n x_D(S_j)\varepsilon_j.
\]
Then \eqref{eq:ls-error-decomposition} reads, on
\(\mathcal A_D\),
\[
  \widehat\theta_D^{\,\mathrm{LS}}-\theta_D
  =
  \widehat G_D^{-1}(z_h+z_\varepsilon).
\]

Conditional on the design, \(z_h\), \(\widehat G_D\), and
\(\one_{\mathcal A_D}\) are fixed, whereas
\(\E_{f,P}[z_\varepsilon\mid S_1,\ldots,S_n]=0\).  Consequently, the
cross term between
\(\widehat G_D^{-1}z_h\) and
\(\widehat G_D^{-1}z_\varepsilon\) has conditional expectation zero.
Since \(\norm{\widehat G_D^{-1}}_{\mathrm{op}}\leqslant2\) on
\(\mathcal A_D\),
\begin{equation}\label{supp:eq:stable-coef-risk}
  \E_{f,P}\!\left[
    \one_{\mathcal A_D}
    \norm{\widehat\theta_D^{\,\mathrm{LS}}-\theta_D}_2^2
  \right]
  \leqslant4\E_{f,P}[\norm{z_h}_2^2]
             +4\E_{f,P}[\norm{z_\varepsilon}_2^2].
\end{equation}

The summands defining \(z_h\) are centered by
\eqref{supp:eq:residual-centered}.  Independence across observations
and constant leverage give
\begin{align}
  \E_{f,P}[\norm{z_h}_2^2]
  &=\frac1n\E_{S\sim\nu_{d,k}}
    [h_D(S)^2\norm{x_D(S)}_2^2]
    =\frac{p_D}{n}\norm{h_D}_2^2.
    \label{supp:eq:zh-moment}
\end{align}
Similarly, independence and centering of the noise eliminate all
cross-observation terms, and
\begin{align}
  \E_{f,P}[\norm{z_\varepsilon}_2^2]
  &=\frac1n\E_P[\varepsilon_1^2]
    \E_{S\sim\nu_{d,k}}[\norm{x_D(S)}_2^2]
  \leqslant\sigma^2\frac{p_D}{n}.
  \label{supp:eq:zeps-moment}
\end{align}

On \(\mathcal A_D\), orthonormality of the chosen basis and
orthogonality of \(h_D\) to \(W_D\) give
\[
  \norm{\widehat f_D^{\,\mathrm{LS}}-f}_2^2
  =\norm{\widehat\theta_D^{\,\mathrm{LS}}-\theta_D}_2^2
   +\norm{h_D}_2^2.
\]
On \(\mathcal A_D^c\), the estimator is zero and the loss is
\(\norm{f}_2^2\leqslant B\) by
\eqref{eq:ellipsoid-l2-bound}. Combining this observation with
\eqref{supp:eq:stable-coef-risk}, \eqref{supp:eq:zh-moment},
and \eqref{supp:eq:zeps-moment} gives
\[
  \E_{f,P}\!\left[
    \norm{\widehat f_D^{\,\mathrm{LS}}-f}_2^2
  \right]
  \leqslant
  \left(1+4\frac{p_D}{n}\right)\norm{h_D}_2^2
  +4\sigma^2\frac{p_D}{n}
  +B\P_{\nu_{d,k}^{\otimes n}}(\mathcal A_D^c).
\]
The harmonic tail bound gives \(\norm{h_D}_2^2\leqslant b_D(s,B)\),
and Proposition~\ref{prop:johnson-gram}, together with the trivial
bound \(\P_{\nu_{d,k}^{\otimes n}}(\mathcal A_D^c)\leqslant1\), gives
\[
  \P_{\nu_{d,k}^{\otimes n}}(\mathcal A_D^c)
  \leqslant
  \min\left\{1,2p_D\exp\left(-c_0\frac{n}{p_D}\right)\right\}.
\]
Taking the stated suprema proves \eqref{eq:stable-ls-risk}.

\subsection{Proof of Proposition~\ref{prop:rank-deficiency}}

Fix \(1\leqslant D\leqslant m\), and abbreviate
\[
  a_D=\frac{B}{\gamma_{D,d}^{\,s}}.
\]
As in the proof of Theorem~\ref{thm:johnson-minimax}, the Euclidean
ball
\[
  \mathcal B_D
  =
  \left\{
    g\in W_D:\|g\|_2^2\leqslant a_D
  \right\}
\]
is contained in \(\mathcal E^s_{d,k}(B)\).  We restrict the noise law to the degenerate distribution
\(\delta_0\), under which \(\varepsilon=0\) almost surely.  This
noiseless law belongs to \(\mathcal P_{\sigma^2}\) for every
\(\sigma^2>0\).

Let \(\Pi_D\) be the uniform distribution, with respect to the
\(L^2(\nu_{d,k})\) Euclidean structure of \(W_D\), on the sphere of
squared radius \(a_D\), and let \(\Theta\sim\Pi_D\).  Since this prior
is supported on \(\mathcal E^s_{d,k}(B)\), the minimax risk is at
least its Bayes risk in the noiseless submodel:
\[
  R_n^\star(\mathcal E^s_{d,k}(B))
  \geqslant
  \inf_{\widehat f}
  \int
  \E_{g,\delta_0}
  \left[
    \|\widehat f-g\|_2^2
  \right]
  \Pi_D(\mathrm dg).
\]

Because the prior is supported on \(W_D\), Pythagoras gives, for every
\(g\in W_D\) and every function-valued estimator \(\widehat f\),
\[
  \|\widehat f-g\|_2^2
  =
  \|P_{\leqslant D}^{\circ}\widehat f-g\|_2^2
  +
  \|(I-P_{\leqslant D}^{\circ})\widehat f\|_2^2.
\]
Projecting an estimator onto \(W_D\) therefore cannot increase its
loss under this prior.  It is consequently enough to consider
estimators taking values in \(W_D\). 

As in the main text, the chosen orthonormal basis identifies \(W_D\)
isometrically with \(\mathbb R^{p_D}\). Conditional on the design,
regard \(\Theta\) as its coefficient vector under this identification.
Write
\[
  \mathbb R^{p_D}
  =
  \operatorname{row}(\mathbf X_D)
  \mathbin{\oplus^\perp}
  \ker(\mathbf X_D).
\]
Let
\[
  \Theta_{\mathrm{row}}
  =
  P_{\operatorname{row}(\mathbf X_D)}\Theta,
  \qquad
  \Theta_{\mathrm{ker}}
  =
  P_{\ker(\mathbf X_D)}\Theta.
\]
Then
\[
  \Theta
  =
  \Theta_{\mathrm{row}}+\Theta_{\mathrm{ker}}.
\]
In the noiseless submodel,
\[
  \mathbf Y
  =
  \mathbf X_D\Theta
  =
  \mathbf X_D\Theta_{\mathrm{row}},
\]
because
\(\mathbf X_D\Theta_{\mathrm{ker}}=0\).  The observations determine
\(\Theta_{\mathrm{row}}\) uniquely.  Indeed, if
\(u,v\in\operatorname{row}(\mathbf X_D)\) satisfy
\(\mathbf X_Du=\mathbf X_Dv\), then
\[
  u-v
  \in
  \operatorname{row}(\mathbf X_D)
  \cap
  \ker(\mathbf X_D)
  =
  \{0\}.
\]
Thus, for a fixed design, the data determine
\(\Theta_{\mathrm{row}}\) but leave
\(\Theta_{\mathrm{ker}}\) unobserved. Conditional on the identified
value of \(\Theta_{\mathrm{row}}\), the spherical prior distributes
\(\Theta_{\mathrm{ker}}\) uniformly on the sphere in
\(\ker(\mathbf X_D)\) centered at zero and having squared radius
\[
  a_D-\|\Theta_{\mathrm{row}}\|_2^2.
\]
Here, \(\Theta_{\mathrm{row}}\) is data-dependent: the design
determines \(\operatorname{row}(\mathbf X_D)\), and
\(\mathbf Y\) determines the unique vector
\(u\in\operatorname{row}(\mathbf X_D)\) satisfying
\(\mathbf X_Du=\mathbf Y\). In particular, \(\Theta_{\mathrm{ker}}\) and
\(-\Theta_{\mathrm{ker}}\) have the same conditional distribution, so
\[
  \E_{\Theta}
  \left[
    \Theta_{\mathrm{ker}}
    \,\middle|\,
    S_1,\ldots,S_n,\mathbf Y
  \right]
  =0.
\]
Since
\(\Theta=\Theta_{\mathrm{row}}+\Theta_{\mathrm{ker}}\) and
\(\Theta_{\mathrm{row}}\) is determined by the data,
\[
\begin{aligned}
  \E_{\Theta}
  \left[
    \Theta
    \,\middle|\,
    S_1,\ldots,S_n,\mathbf Y
  \right]
  &=
  \Theta_{\mathrm{row}}
  +
  \E_{\Theta}
  \left[
    \Theta_{\mathrm{ker}}
    \,\middle|\,
    S_1,\ldots,S_n,\mathbf Y
  \right]\\
  &=
  \Theta_{\mathrm{row}}.
\end{aligned}
\]
Thus, \(\Theta_{\mathrm{row}}\) is the posterior mean and hence the
Bayes estimator under squared-error loss. Accordingly,
for a fixed design and the noiseless observations
\(\mathbf Y=\mathbf X_D\Theta\),
\[
  \widehat\Theta^{\,B}
  (\mathbf X_D,\mathbf X_D\Theta)
  =
  P_{\operatorname{row}(\mathbf X_D)}\Theta.
\]
Consequently, its Bayes risk conditional on the design is
\[
  \E_{\Theta}
  \left[
    \left\|
      \widehat\Theta^{\,B}
      (\mathbf X_D,\mathbf X_D\Theta)-\Theta
    \right\|_2^2
    \,\middle|\,
    S_1,\ldots,S_n
  \right]
  =
  \E_{\Theta}
  \left[
    \|P_{\ker(\mathbf X_D)}\Theta\|_2^2
    \,\middle|\,
    S_1,\ldots,S_n
  \right].
\]

It remains to compute this expectation.  The uniform distribution on
the sphere of squared radius \(a_D\) in \(\mathbb R^{p_D}\) is
isotropic: all coordinate directions have the same second moment, and
their squared coordinates sum to \(a_D\).  Consequently,
\[
  \E_{\Theta}[\Theta\Theta^\top]
  =
  \frac{a_D}{p_D}I_{p_D}.
\]
For the orthogonal projector
\(P_{\ker(\mathbf X_D)}\), we have
\[
  \|P_{\ker(\mathbf X_D)}\Theta\|_2^2
  =
  \Theta^\top P_{\ker(\mathbf X_D)}\Theta.
\]
Since \(\Theta\) is independent of the design and the projector is
fixed conditionally on \(S_1,\ldots,S_n\),
\begin{align*}
  \E_{\Theta}
  \left[
    \|P_{\ker(\mathbf X_D)}\Theta\|_2^2
    \,\middle|\,
    S_1,\ldots,S_n
  \right]
  &=
  \operatorname{tr}
  \left\{
    P_{\ker(\mathbf X_D)}
    \E_{\Theta}[\Theta\Theta^\top]
  \right\}\\
  &=
  \frac{a_D}{p_D}
  \operatorname{tr}
  \left\{
    P_{\ker(\mathbf X_D)}
  \right\}\\
  &=
  a_D
  \frac{
    p_D-\operatorname{rank}(\mathbf X_D)
  }{p_D}.
\end{align*}
The last equality follows because the trace of an orthogonal projector
equals the dimension of its range, here
\[
  \dim\ker(\mathbf X_D)
  =
  p_D-\operatorname{rank}(\mathbf X_D).
\]
Averaging over the design gives
\[
  R_n^\star(\mathcal E^s_{d,k}(B))
  \geqslant
  a_D\rho_{D,n}.
\]
Since this holds for every \(1\leqslant D\leqslant m\), maximizing
over \(D\) proves \eqref{eq:rank-deficiency-lower}.

Moreover,
\[
  \operatorname{rank}(\mathbf X_D)
  \leqslant\min(n,p_D),
\]
and hence
\[
  p_D-\operatorname{rank}(\mathbf X_D)
  \geqslant(p_D-n)_+.
\]
Dividing by \(p_D\) and taking expectations proves
\eqref{eq:rank-deficiency-dimension}.

We finally verify the full-level identities.  Recall that
\(\mathcal O_n\) is the set of observed subsets and that
\[
  M_n=N-|\mathcal O_n|.
\]
Set \(o_n=|\mathcal O_n|=N-M_n\).  Repeated observations merely
duplicate rows of \(\mathbf X_m\), so its rank depends only on the
\(o_n\) distinct observed subsets.  Recall also that \(W_m\) is the
\((N-1)\)-dimensional space of all centered functions on \(\Sdk\).

If \(o_n<N\), evaluation at the \(o_n\) distinct observed subsets has
rank \(o_n\).  Indeed, let \(T_1,\ldots,T_{o_n}\) denote these
subsets and prescribe arbitrary values \(y_1,\ldots,y_{o_n}\).
Choose one unobserved subset \(T_0\), assign it the value
\(
  -\sum_{j=1}^{o_n}y_j,
\)
and assign zero to every other unobserved subset.  The resulting
function is centered and takes the prescribed values at
\(T_1,\ldots,T_{o_n}\).  Hence, the evaluation map from \(W_m\) onto
\(\mathbb R^{o_n}\) is surjective and has rank \(o_n\).

If \(o_n=N\), every subset in \(\Sdk\) is observed.  The complete
evaluation map sends \(f\in W_m\) to its vector of \(N\) values.
Since \(f\) is centered, this vector satisfies the single linear
constraint
\[
  \sum_{S\in\Sdk}f(S)=0.
\]
Conversely, every vector satisfying this constraint defines a centered
function on \(\Sdk\).  The image of the evaluation map is therefore
an \((N-1)\)-dimensional hyperplane, and its rank is \(N-1\).

Combining the two cases gives, for every realization of the design,
\[
  \operatorname{rank}(\mathbf X_m)
  =
  \min(o_n,N-1).
\]
Since \(p_m=N-1\) and \(M_n=N-o_n\), it follows that
\[
\begin{aligned}
  p_m-\operatorname{rank}(\mathbf X_m)
  &=
  (N-1)-\min(o_n,N-1)\\
  &=
  (M_n-1)_+.
\end{aligned}
\]
Taking expectations and dividing by \(p_m=N-1\) in
\eqref{eq:rank-deficiency} proves
\eqref{eq:full-rank-deficiency}.

It remains to derive the explicit expression.  Each fixed subset is
absent from all \(n\) independent draws with probability
\((1-1/N)^n\).  Summing these \(N\) absence indicators gives
\[
  \E_{\nu_{d,k}^{\otimes n}}[M_n]
  =
  N\left(1-\frac1N\right)^n.
\]
Moreover,
\[
  (M_n-1)_+
  =
  M_n-1+\one_{\{M_n=0\}}.
\]
Therefore,
\[
  \E_{\nu_{d,k}^{\otimes n}}[(M_n-1)_+]
  =
  N\left(1-\frac1N\right)^n
  -1
  +\P_{\nu_{d,k}^{\otimes n}}(M_n=0),
\]
which proves \eqref{eq:full-rank-deficiency-explicit}.

To conclude, we evaluate the coverage probability appearing in this
expression.  For each \(S\in\Sdk\), let \(A_S\) be the event that
\(S\) is absent from the sample.  Since
\[
  \{M_n=0\}
  =
  \bigcap_{S\in\Sdk}A_S^c,
\]
the inclusion--exclusion formula gives
\[
  \P_{\nu_{d,k}^{\otimes n}}(M_n=0)
  =
  \sum_{\mathcal J\subseteq\Sdk}
  (-1)^{|\mathcal J|}
  \P_{\nu_{d,k}^{\otimes n}}
  \left(
    \bigcap_{S\in\mathcal J}A_S
  \right).
\]
If \(|\mathcal J|=j\), avoiding every subset in \(\mathcal J\) in
all \(n\) draws has probability \((1-j/N)^n\).  Since there are
\(\binom Nj\) such collections \(\mathcal J\), it follows that
\[
  \P_{\nu_{d,k}^{\otimes n}}(M_n=0)
  =
  \sum_{j=0}^N
  (-1)^j\binom Nj
  \left(1-\frac jN\right)^n.
\]
When \(n<N\), this probability is zero because \(n\) draws cannot
cover \(N\) distinct subsets. This completes the proof.

\subsection{Proof of Proposition~\ref{prop:centered-completion}}

For every deterministic set \(\mathcal U\subseteq\Sdk\), write
\[
  \mathcal H(\mathcal U)
  =
  \left\{
    g\in W_m:
    g(S)=0
    \text{ for every }S\notin\mathcal U
  \right\}.
\]
Recall that
\[
  \mathcal U_n
  =
  \Sdk\setminus\mathcal O_n
\]
is the random set of unobserved subsets.  Then, consistently with the
notation used in the main text,
\(
  \mathcal H_n
  =
  \mathcal H(\mathcal U_n).
\)
Conditional on the design, \(\mathcal H_n\) is the space of centered
functions supported on the fixed set \(\mathcal U_n\).  If \(M_n\geqslant1\), its values on the \(M_n\) unobserved subsets
satisfy one linear constraint, so its dimension is \(M_n-1\).  If
\(M_n=0\), then \(\mathcal H_n=\{0\}\).  Hence, in both cases,
\[
  \dim(\mathcal H_n)
  =
  (M_n-1)_+.
\]

Let \(\widehat f^{\,\mathrm{comp},0}\) be the noiseless completion
estimator introduced in the main text.  It agrees with \(f\) on
\(\mathcal O_n\) and is constant on \(\mathcal U_n\).  As shown there,
\[
  f-\widehat f^{\,\mathrm{comp},0}
  =
  P_{\mathcal H_n}f.
\]
Equivalently,
\[
  \widehat f^{\,\mathrm{comp},0}-f
  =
  -P_{\mathcal H_n}f.
\]
This identity also holds when \(M_n=0\), since then
\(\mathcal H_n=\{0\}\) and
\(\widehat f^{\,\mathrm{comp},0}=f\).

Let \(\tau:\Sdk\to\Sdk\) be any permutation of the \(N\) subsets, and
let \(R_\tau\) be its action on functions:
\[
  (R_\tau g)(S)
  =
  g(\tau^{-1}(S)).
\]
Because the design points are sampled uniformly from \(\Sdk\), the
random set of unobserved subsets \(\mathcal U_n\) has the same
distribution as \(\tau(\mathcal U_n)\).  Moreover, relabeling the
unobserved subsets relabels the corresponding subspace:
\[
  \mathcal H(\tau(\mathcal U_n))
  =
  R_\tau\mathcal H_n.
\]
Since \(R_\tau\) preserves the uniform inner product, the associated
orthogonal projectors satisfy
\[
  P_{\mathcal H(\tau(\mathcal U_n))}
  =
  R_\tau P_{\mathcal H_n}R_\tau^{-1}.
\]
Averaging over the design and using the equality in distribution of
\(\mathcal U_n\) and \(\tau(\mathcal U_n)\) gives
\[
  R_\tau\mathcal T_nR_\tau^{-1}
  =
  \mathcal T_n,
  \qquad
  \mathcal T_n
  :=
  \E_{\nu_{d,k}^{\otimes n}}
  [P_{\mathcal H_n}].
\]
Thus, \(\mathcal T_n\) is the deterministic operator obtained by
averaging the random projector \(P_{\mathcal H_n}\) over the design;
in particular,
\[
  \mathcal T_nf
  =
  \E_{\nu_{d,k}^{\otimes n}}
  [P_{\mathcal H_n}f],
\] 
and  \(\mathcal T_n\) commutes with every permutation of the \(N\)
subsets.

Fix an arbitrary ordering of the \(N\) elements of \(\Sdk\), and
identify each function \(f\in L^2(\Sdk,\nu_{d,k})\) with its vector
of values
\[
  (f(S))_{S\in\Sdk}\in\mathbb R^N.
\]
Under this identification, \(\mathcal T_n\) is a linear operator on
\(\mathbb R^N\).  Let \(t_n(S,T)\) denote the entry of its matrix
indexed by \(S,T\in\Sdk\). The identity
\(R_\tau\mathcal T_nR_\tau^{-1}=\mathcal T_n\) implies
\[
  t_n(\tau(S),\tau(T))
  =
  t_n(S,T)
\]
for every permutation \(\tau\) of \(\Sdk\).  For any \(S,R\in\Sdk\), there exists a permutation \(\tau\) such that
\(\tau(S)=R\).  Likewise, for any \(S\ne T\) and \(R\ne U\), there
exists a permutation \(\tau\) such that
\(\tau(S)=R\) and \(\tau(T)=U\).  Hence all diagonal entries of the
matrix are equal, and all off-diagonal entries are equal. Therefore, for some scalars \(a_n\) and \(b_n\),
\[
  \mathcal T_n
  =
  (a_n-b_n)I_N+b_n\mathbf 1\mathbf 1^\top.
\]
Equivalently, for every function \(f\) and every \(S\in\Sdk\),
\[
  (\mathcal T_nf)(S)
  =
  (a_n-b_n)f(S)
  +
  b_n\sum_{T\in\Sdk}f(T).
\]
Every space \(\mathcal H_n\) consists of centered functions and is
therefore orthogonal to the constant function \(\mathbf 1\).
Consequently,
\[
  P_{\mathcal H_n}\mathbf 1=0
  \qquad\text{and hence}\qquad
  \mathcal T_n\mathbf 1=0.
\]
Finally, if \(f\in W_m\), then \(f\) is centered and
\(\mathbf 1^\top f=0\).  Therefore,
\[
  \mathcal T_nf
  =
  (a_n-b_n)f.
\]
Thus, the restriction of \(\mathcal T_n\) to \(W_m\) is a scalar
multiple of the identity.

Write
\(
  c_n=a_n-b_n,
\)
so that
\[
  \mathcal T_nf=c_nf,
  \qquad f\in W_m.
\]
Moreover, \(\mathcal T_n\mathbf 1=0\).  Since
\(L^2(\Sdk,\nu_{d,k})\) is the orthogonal sum of the constant
functions and the \((N-1)\)-dimensional space \(W_m\), it follows that
\[
  \operatorname{tr}(\mathcal T_n)
  =
  c_n(N-1).
\]
On the other hand, linearity of the trace and the fact that the trace
of an orthogonal projector equals the dimension of its range give
\[
\begin{aligned}
  \operatorname{tr}(\mathcal T_n)
  =
  \E_{\nu_{d,k}^{\otimes n}}
  \left[
    \operatorname{tr}(P_{\mathcal H_n})
  \right]
  &=
  \E_{\nu_{d,k}^{\otimes n}}
  [\dim(\mathcal H_n)]\\
  &=
  \E_{\nu_{d,k}^{\otimes n}}
  [(M_n-1)_+].
\end{aligned}
\]
Therefore,
\[
  c_n
  =
  \frac{
    \E_{\nu_{d,k}^{\otimes n}}[(M_n-1)_+]
  }{N-1}
  =
  \rho_{m,n}.
\]
Consequently, for every centered \(f\),
\begin{align}
  \E_{\nu_{d,k}^{\otimes n}}
  \left[
    \|P_{\mathcal H_n}f\|_2^2
  \right]
  &=
  \E_{\nu_{d,k}^{\otimes n}}
  \left[
    \langle f,P_{\mathcal H_n}f\rangle
  \right]
  \notag\\
  &=
  \langle f,\mathcal T_nf\rangle
  =
  \rho_{m,n}\|f\|_2^2.
  \label{supp:eq:completion-signal}
\end{align}
Here we used that an orthogonal projector is self-adjoint and
idempotent, so
\(\|P_{\mathcal H_n}f\|_2^2
=\langle f,P_{\mathcal H_n}f\rangle\).

We next control the contribution of observation noise.  Conditional
on the design, define
\[
  \overline\varepsilon(S)
  =
  \frac1{C_n(S)}
  \sum_{j:S_j=S}\varepsilon_j,
  \qquad S\in\mathcal O_n.
\]
These variables are conditionally independent and centered, and
\[
  \E_P
  \left[
    \overline\varepsilon(S)^2
    \,\middle|\,
    S_1,\ldots,S_n
  \right]
  \leqslant
  \frac{\sigma^2}{C_n(S)}.
\]
Since the completion rule is linear in the cell means and the noise
is centered, for every \(S\in\Sdk\),
\[
  \E_{f,P}
  \left[
    \widehat f^{\,\mathrm{comp}}(S)
    \,\middle|\,
    S_1,\ldots,S_n
  \right]
  =
  \widehat f^{\,\mathrm{comp},0}(S).
\]

If \(M_n\geqslant1\), the noise component
\(
  \widehat f^{\,\mathrm{comp}}
  -
  \widehat f^{\,\mathrm{comp},0}
\)
is equal to \(\overline\varepsilon(S)\) on \(\mathcal O_n\) and to
\[
  -\frac1{M_n}
  \sum_{T\in\mathcal O_n}\overline\varepsilon(T)
\]
on \(\mathcal U_n\).  The squared \(L^2(\nu_{d,k})\)-norm of this noise component is
\[
  \frac1N
  \left\{
    \sum_{S\in\mathcal O_n}\overline\varepsilon(S)^2
    +
    \frac1{M_n}
    \left(
      \sum_{S\in\mathcal O_n}\overline\varepsilon(S)
    \right)^2
  \right\}.
\]
Since the cell averages are conditionally independent and
centered,
\[
  \E_P
  \left[
    \left(
      \sum_{S\in\mathcal O_n}\overline\varepsilon(S)
    \right)^2
    \,\middle|\,
    S_1,\ldots,S_n
  \right]
  =
  \sum_{S\in\mathcal O_n}
  \E_P
  \left[
    \overline\varepsilon(S)^2
    \,\middle|\,
    S_1,\ldots,S_n
  \right].
\]
Consequently,
\begin{align}
  &\E_{f,P}\!\left[
    \left\|
      \widehat f^{\,\mathrm{comp}}
      -
      \widehat f^{\,\mathrm{comp},0}
    \right\|_2^2
    \,\middle|\,
    S_1,\ldots,S_n
  \right]
  \nonumber\\
  &\qquad\leqslant
  \frac{\sigma^2}{N}
  \left(1+\frac1{M_n}\right)
  \sum_{S\in\mathcal O_n}\frac1{C_n(S)}
  \leqslant
  \frac{2\sigma^2}{N}
  \sum_{S\in\mathcal O_n}\frac1{C_n(S)}.
  \label{supp:eq:completion-noise}
\end{align}
If \(M_n=0\), every subset is observed and
\(\widehat f^{\,\mathrm{comp},0}=f\).  Moreover,
\[
  \overline Y(S)
  =
  f(S)+\overline\varepsilon(S),
  \qquad S\in\Sdk.
\]
Since \(f\) is centered,
\[
  \frac1N\sum_{T\in\Sdk}f(T)=0,
\]
and hence
\begin{align*}
  \widehat f^{\,\mathrm{comp}}(S)
  -
  \widehat f^{\,\mathrm{comp},0}(S)
  &=
  \overline Y(S)
  -
  \frac1N\sum_{T\in\Sdk}\overline Y(T)
  -
  f(S)\\
  &=
  \overline\varepsilon(S)
  -
  \frac1N\sum_{T\in\Sdk}\overline\varepsilon(T).
\end{align*}
Subtracting the uniform mean cannot increase the average squared
norm.  Taking conditional expectation therefore gives
\[
  \E_{f,P}\!\left[
    \left\|
      \widehat f^{\,\mathrm{comp}}
      -
      \widehat f^{\,\mathrm{comp},0}
    \right\|_2^2
    \,\middle|\,
    S_1,\ldots,S_n
  \right]
  \leqslant
  \frac{\sigma^2}{N}
  \sum_{S\in\Sdk}\frac1{C_n(S)}.
\]
Since \(\mathcal O_n=\Sdk\) when \(M_n=0\), the preceding bound is
smaller than the right-hand side of
\eqref{supp:eq:completion-noise}.  Thus,
\eqref{supp:eq:completion-noise} holds in both cases.

The decomposition
\[
  \widehat f^{\,\mathrm{comp}}-f
  =
  \bigl(
    \widehat f^{\,\mathrm{comp},0}-f
  \bigr)
  +
  \bigl(
    \widehat f^{\,\mathrm{comp}}
    -
    \widehat f^{\,\mathrm{comp},0}
  \bigr)
\]
has conditionally vanishing cross term, because the first component
is fixed by the design and the second has conditional mean zero.
Let \(C\sim\operatorname{Bin}(n,1/N)\), and define
\(r(0)=0\) and \(r(c)=1/c\) for every integer \(c\geqslant1\).
Then
\[
  \eta_{n,N}=\E[r(C)].
\]
For each fixed \(S\in\Sdk\), the count \(C_n(S)\) has the same
distribution as \(C\).  Moreover, since \(r(0)=0\),
\[
  \sum_{S\in\mathcal O_n}\frac1{C_n(S)}
  =
  \sum_{S\in\Sdk}r(C_n(S)).
\]
Therefore,
\begin{align*}
  \frac1N
  \E_{\nu_{d,k}^{\otimes n}}
  \left[
    \sum_{S\in\mathcal O_n}\frac1{C_n(S)}
  \right]
  &=
  \frac1N
  \sum_{S\in\Sdk}
  \E_{\nu_{d,k}^{\otimes n}}
  [r(C_n(S))]\\
  &=
  \E[r(C)]
  =
  \eta_{n,N}.
\end{align*}
Combining this identity with
\eqref{supp:eq:completion-signal} and
\eqref{supp:eq:completion-noise} gives
\[
  \E_{f,P}
  \left[
    \|\widehat f^{\,\mathrm{comp}}-f\|_2^2
  \right]
  \leqslant
  \rho_{m,n}\|f\|_2^2
  +
  2\sigma^2\eta_{n,N}.
\]
Since \(\|f\|_2^2\leqslant B\), taking the stated suprema proves
\eqref{eq:centered-completion-risk}.

It remains to bound \(\eta_{n,N}\).  Since
\[
  0\leqslant r(c)\leqslant1,
  \qquad c\geqslant0,
\]
we have \(\eta_{n,N}\leqslant1\).  Moreover,
\[
  r(c)\leqslant\frac2{c+1},
  \qquad c\geqslant0,
\]
and therefore
\[
  \eta_{n,N}
  \leqslant
  2\E\left[\frac1{C+1}\right].
\]
The identity
\[
  \frac1{c+1}\binom nc
  =
  \frac1{n+1}\binom{n+1}{c+1}
\]
gives
\begin{align*}
  \E\left[\frac1{C+1}\right]
  &=
  \sum_{c=0}^n
  \frac1{c+1}\binom nc
  \left(\frac1N\right)^c
  \left(1-\frac1N\right)^{n-c}\\
  &=
  \frac{N}{n+1}
  \left\{
    1-\left(1-\frac1N\right)^{n+1}
  \right\}\\
  &\leqslant
  \frac{N}{n+1}.
\end{align*}
Combining the two bounds proves
\eqref{eq:inverse-count-bound}.

\subsection{Proof of Theorem~\ref{thm:rank-aware-minimax}}

For brevity, write
\[
  \mathfrak R
  =
  \mathfrak R_{n,d,k}(s,B,\sigma^2),
  \qquad
  \mathfrak D
  =
  \mathfrak D_{n,d,k}(s,B),
  \qquad
  a
  =
  \frac{B}{\gamma_{m,d}^{\,s}},
  \qquad
  \Gamma
  =
  \gamma_{m,d}^{\,s}.
\]
Since \(m\geqslant1\), the monotonicity of the spectral weights and
\(\gamma_{1,d}=1\) imply
\(
  \Gamma\geqslant1.
\)

Theorem~\ref{thm:johnson-minimax} gives
\(R_n^\star(\mathcal E^s_{d,k}(B))\geqslant c\mathfrak R\),
where \(c\) denotes the constant in that theorem.
Proposition~\ref{prop:rank-deficiency} also gives
\(R_n^\star(\mathcal E^s_{d,k}(B))\geqslant\mathfrak D\).
Hence
\[
  R_n^\star(\mathcal E^s_{d,k}(B))
  \geqslant \max\{c\mathfrak R,\mathfrak D\}
  \geqslant \frac{\min(c,1)}{2}
             (\mathfrak R+\mathfrak D).
\]
This proves the lower inequality in
\eqref{eq:rank-aware-minimax} with a universal constant.

For the upper bound, Theorem~\ref{thm:johnson-minimax}, together with
\(B=\Gamma a\), gives
\begin{equation}\label{supp:eq:rank-aware-projection-upper}
  R_n^\star(\mathcal E^s_{d,k}(B))
  \leqslant
  2\left(
    1+\frac{\Gamma a}{\sigma^2}
  \right)\mathfrak R.
\end{equation}
On the other hand, Proposition~\ref{prop:centered-completion} gives
\[
  R_n^\star(\mathcal E^s_{d,k}(B))
  \leqslant
  B\rho_{m,n}
  +
  2\sigma^2\eta_{n,N}.
\]
The term \(D=m\) in the definition of \(\mathfrak D\) satisfies
\[
  \mathfrak D
  \geqslant
  \frac{B}{\Gamma}\rho_{m,n}
  =
  a\rho_{m,n}.
\]
Consequently,
\[
  B\rho_{m,n}
  =
  \Gamma a\rho_{m,n}
  \leqslant
  \Gamma\mathfrak D,
\]
and hence
\begin{equation}\label{supp:eq:rank-aware-completion-upper}
  R_n^\star(\mathcal E^s_{d,k}(B))
  \leqslant
  \Gamma\mathfrak D
  +
  2\sigma^2\eta_{n,N}.
\end{equation}

Put
\[
  v_m
  =
  \frac{\sigma^2(N-1)}n.
\]
This is precisely the quantity \(v_D(n,\sigma^2)\) at \(D=m\),
because \(p_m=N-1\). The bounds in \eqref{eq:inverse-count-bound}, together with
\(N\geqslant2\), imply
\begin{equation}\label{supp:eq:eta-vm}
  \sigma^2\eta_{n,N}
  \leqslant
  4\min\{\sigma^2,v_m\}.
\end{equation}
Indeed, if \(n\leqslant N-1\), then
\(v_m\geqslant\sigma^2\), and
\(\eta_{n,N}\leqslant1\) gives
\[
  \sigma^2\eta_{n,N}
  \leqslant
  \sigma^2
  =
  \min\{\sigma^2,v_m\}.
\]
If \(n\geqslant N\), then \(v_m<\sigma^2\), and
\[
  \eta_{n,N}
  \leqslant
  \frac{2N}{n+1}
  \leqslant
  \frac{4(N-1)}n,
\]
so
\[
  \sigma^2\eta_{n,N}
  \leqslant
  4v_m
  =
  4\min\{\sigma^2,v_m\}.
\]

We now distinguish two cases.  If \(a\leqslant\sigma^2\), then
\[
  1+\frac{\Gamma a}{\sigma^2}
  \leqslant
  1+\Gamma
  \leqslant
  2\Gamma.
\]
Equation~\eqref{supp:eq:rank-aware-projection-upper} therefore gives
\[
  R_n^\star(\mathcal E^s_{d,k}(B))
  \leqslant
  4\Gamma\mathfrak R
  \leqslant
  4\Gamma(\mathfrak R+\mathfrak D).
\]

Suppose instead that \(a>\sigma^2\).  The term \(D=m\) in the maximum
defining \(\mathfrak R\) gives
\[
  \mathfrak R
  \geqslant
  \min\{a,v_m\}.
\]
Since \(a>\sigma^2\),
\[
  \min\{\sigma^2,v_m\}
  \leqslant
  \min\{a,v_m\}
  \leqslant
  \mathfrak R.
\]
Combining
\eqref{supp:eq:rank-aware-completion-upper} and
\eqref{supp:eq:eta-vm} yields
\[
  R_n^\star(\mathcal E^s_{d,k}(B))
  \leqslant
  \Gamma\mathfrak D
  +
  8\mathfrak R
  \leqslant
  8\Gamma(\mathfrak R+\mathfrak D),
\]
where the last inequality uses \(\Gamma\geqslant1\).

The two cases prove the upper inequality in
\eqref{eq:rank-aware-minimax}, with the universal constant \(C=8\).

\end{document}